\documentclass{article}
\usepackage{amsfonts, latexsym, amssymb,amsmath,amsthm}
\usepackage{graphicx}
\usepackage{caption}

\makeatletter
\def\bee{\begin{equation}}
\def\eee{\end{equation}}
\def\back{\backslash}
\def\ti{\tilde}
\def \al{\alpha}
\def \be{\beta}
\def \ga{\gamma}
\def \de{\delta}
\def \ep{\varepsilon}
\def \ze{\zeta}
\def \la{\lambda}
\def \si{\sigma}
\def \vp{\varphi}
\def \om{\omega}
\def \Ga{\Gamma}

\def \Th{\Theta}

\def \operatorname#1{\mathop{\rm #1}}
\def \p#1{{\frac{\pi \imath}{#1}}}
\def \pii {\pi \imath}

\def \cm{\Cal M}
\def \const{\operatorname{const}}
\def \eq{\equiv}
\def \cd{\partial}
\def \bar{\overline}
\def \diag{\operatorname{diag}}
\def \tr{\operatorname{tr}}
\def \th {\Theta}
\def \0{{\Theta_0}}
\def \1{{\Theta_1}}
\def \tin{\Theta_\infty}
\def \op{\operatorname}
\def \num#1{(\ref{#1})}

\def \underset#1#2{\mathrel{\mathop{#2}\limits_{#1}}}

\@addtoreset{equation}{section}
\renewcommand{\theequation}{\arabic{section}.\arabic{equation}}
\def\bm{\left(\begin{array}{cc}}
\def\em{\end{array}\right)}
\def\Bbb{\mathbb}
\def\Cal{\mathcal}
\newtheorem{theorem}{Theorem}[section]
\newtheorem{remark}{Remark}[section]
\newtheorem{corollary}{Corollary}[section]
\newtheorem{statement}{Statement}[section]
\newtheorem{proposition}{Proposition}[section]
\newtheorem{conjecture}{Conjecture}[section]
\def\text#1{\mbox{#1}}

\def\lim{\mathop{\mbox{lim}}\limits}

\makeatother
\begin{document}
\title{Connection Formulae for Asymptotics of the
Fifth Painlev\'e Transcendent on the Imaginary Axis: I}
\author{F.~V.~Andreev\thanks{E-mail: F-Andreev@wiu.edu}\;
and A.~V.~Kitaev\thanks{E-mail: kitaev@pdmi.ras.ru}\\
Department of Mathematics,
Western Illinois University,
Macomb, IL 61455, USA\\
Steklov Mathematical Institute, Fontanka 27, St.Petersburg, 191023, Russia}

\date{April 12, 2019}
\maketitle

\begin{abstract}
Leading terms of asymptotic expansions for the general complex solutions
of the fifth Painlev\'e equation as $t\to\imath\infty$ are found.
These asymptotics are parameterized by monodromy data of the associated linear ODE,
$$
\frac{d}{d\la}Y= \left(\frac t2\si_3 +
\frac{A_0}\la+\frac{A_1}{\la-1}\right)Y.
$$
The parametrization allows one to derive connection formulas for the asymptotics.
We provide numerical verification of the results. Important special cases of the
connection formulas are also considered.
\vspace{24pt}\vspace{24pt}
\vspace{24pt}\\
Short title: Connection formulas for $P_5$
\end{abstract}
\maketitle
\newpage

\setcounter{page}2
\section{Introduction}\label{sec:Introduction}

We study asymptotics\footnote{The notion ``asymptotics'' is used throughout the
paper to abbreviate the expression ``the leading terms of asymptotic
expansion''}
as $t\to\infty$, $\Re(t)=\mathcal O(1)$, of isomonodromy deformations (with respect to parameter $t$) of the
following linear ODE
\begin{equation}
\label{mainl}
\frac{d}{d\la}Y=A(\la, t)Y=\left(\frac t2\si_3 +
\frac{A_0}\la+\frac{A_1}{\la-1}\right)Y.
\end{equation}
Here $\sigma_3=\bm 1&0\\0&-1\em$ and
the matrices $A_p$ ($p=0,1$) are independent of $\lambda$.
This paper is a continuation of our earlier work~\cite{AK} on asymptotics of the isomonodromy deformations of
system~\eqref{mainl} on the real $t$-axis. For convenience of the reader, we recall here the basic notation
and some results obtained in~\cite{AK}.

Following Jimbo and Miwa~\cite{JM}, we consider the following parameterization of the matrices
$A_p$,
\begin{eqnarray*}
&\label{eq:mat0}
A_0=\left(\begin{array}{cc}
z +\frac\0{2}&-u(z +\0 )\\
\frac zu& -z-\frac\02\end{array}\!\!\right),
&\\
&
A_1=\left(\begin{array}{cc}
-z -\dfrac{\0+\tin}2
&uy\big(
z +\frac{\0-\1+\tin}2\big)
\\
-\frac 1 {uy}\big( z+\frac {
\0+\1+\tin}2\big)
&z+\dfrac {\0+\tin}2
\end{array}\right).\label{eq:mat1}&
\end{eqnarray*}
Then the isomonodromy deformations of Equation~(\ref{mainl})  with respect to $t$
are governed by the following system of nonlinear ODEs,
which is called the Isomonodromy Deformation System (IDS)
\begin{eqnarray}
t\frac{ dy}{dt} &=& ty - 2z (y-1)^2 -(y-1)
\big( \frac{\0 -\1 +\tin}2 y-\frac{3\0 +\1 +\tin}2\big),
\label{eq:ids1}\\
t \frac{dz}{dt} &=& yz\big(z+\frac{\0-\1 +\tin}2\big) -\frac
1y (z+\0) \big(z+\frac{ \0+\1+\tin}2\big),\label{eq:ids2}\\
t\frac d{dt} \operatorname{log} u
&=&-2 z-\0 +y\big(z +\frac{\0-\1+\tin}2 \big)
+\frac 1y\big( z+\frac{ \0+\1+\tin}2 \big).\label{eq:ids3}
\end{eqnarray}
In this system $\Theta_\nu$ ($\nu=0,\,1,\,\infty$) are complex constants
considered as parameters.
Eliminating function $z=z(t)$ from
Equation~(\ref{eq:ids2}) by using Equation~(\ref{eq:ids1}), one finds that function
$y=y(t)$ satisfies the fifth Painlev\'e equation $(P_5)$:
\begin{eqnarray}
\frac{d^2y}{dt^2} =
( \frac 1{2y} +\frac 1{y-1} )
(\frac {dy}{dt})^2 -\frac{dy}{tdt}
+\frac{(y-1)^2}{t^2} (\hat\al y +\frac{\hat\be}y)
+\hat\ga\frac yt+\hat\de\frac{y(y+1)}{y-1},\label{eq:P5}\\
\hat\al=\frac 12
\big(\frac{\0-\1 +\tin}2\big)^2,\;\hat\be =-\frac 12
\big(\frac{\0-\1 -\tin}2\big)^2,\;
\hat\ga=1-\0-\1,\;\hat\de=-\frac 12.
\label{eq:coeffP5}
\end{eqnarray}

Together with the functions $y(t)$, $z(t)$, and $u(t)$,
we are also interested in the so-called $\zeta$-function \cite{JM, O2, AK}\footnote{Function $\zeta(t)$
exactly coincides with function $\sigma(t)$ introduced by Jimbo and Miwa in \cite{JM}.},
which is associated with the Hamiltonian structure
for System (\ref{eq:ids1}), (\ref{eq:ids2}) and the corresponding $\tau$-function.
Function $\ze(t)$, in terms of functions $y(t)$ and $z(t)$,
is defined as follows
\begin{eqnarray}
&\ze(t)=\frac 14((\0+\tin)^2-\Th_1^2)
-zt-\nonumber&\\
&-\big( z-\frac 1y
\big( z+\frac{\0+\1+\tin}2\big)\big)
\big( z+\0-y
\big( z+\frac{\0-\1+\tin}2\big)\big),\label{eq:zeta-def}& \\
&\frac{\cd \ze}{\cd t}=\frac {d\zeta}{dt}=-z.\label{eq:zeta-derivativ}&
\end{eqnarray}
This function has been proven to be an important
object in physical and geometrical applications, where
the connection problem for its asymptotics arises in a natural way.

As in the work~\cite{AK} we use for our studies the method of isomonodromy deformations: this method
is based on parametrization of asymptotics of IDS~\eqref{eq:ids1}-\eqref{eq:ids3} via the monodromy data of
system~\eqref{mainl}. Such parametrization allows one to find the connection formulae for asymptotics of
a given solution at different singular points (in our case $0$ and $\infty$) and at different directions
in the neighborhood of an essential singular point (in our case $t\to\pm i\infty$ and $t\to\pm\infty$).
This parametrization is based on asymptotic solution of the direct and inverse monodromy problems for
Equation~\eqref{mainl}. There are various approaches how one can achieve these goals:

(1) Combination of local asymptotic analysis of IDS with the subsequent asymptotic solution of the direct and
inverse monodromy problems. Note that IDS can always be presented as the first order system of ODEs with the
quadratic r.-h.s. with respect to unknown functions, so that local asymptotics of those systems can be constructed
in a regular way. However, the local asymptotic analysis can be also based on various asymptotic ideas and approaches.
The following asymptotic solutions of the direct and inverse monodromy problems also can be performed
by using slightly different asymptotic methods;

(2) The inverse monodromy problem can be formulated as a matrix Riemann-Hilbert conjugation problem in the
complex plane with its further asymptotic solution with the help of the Deift-Zhou asymptotic method;

(3) The method we use here is in general similar to (1), however, it suffers substantial differences.
We do not use any apriori information about the local asymptotic expansions for IDS~\eqref{eq:ids1}-\eqref{eq:ids3}
or $P_5$ \eqref{eq:P5}. Instead, for the solution of the
direct monodromy problem for Equation~\eqref{mainl} some asymptotic assumptions on the matrix elements of
of Matrices~\eqref{eq:mat0} and \eqref{eq:mat1}, which are much less detailed comparing with the local asymptotics,
are assumed.
These assumptions are dictated by our ability to perform asymptotic estimates which finally lead us to asymptotic
solution of the direct monodromy problem. In curse of these calculations appears a number of such assumptions.
Then asymptotic solution of the inverse monodromy problem gives us local asymptotics of the matrix elements together
with their parameterization via the monodromy data. After that we verify, whether thus obtained asymptotics passes all the
assumptions imposed. This verification provide us with the restrictions on monodromy data, for which our analysis
is valid. So, this methodology allows one to drop out the preliminary stage of the local asymptotic analysis of IDS.
It is substantially based on the fact that we deal with the isomonodromy deformations. We do not touch at any stage of
our asymptotic analysis either IDS~\eqref{eq:ids1}-\eqref{eq:ids3}, nor Equation~\eqref{eq:P5}.
Therefore, we call this asymptotic method as the Method of Isomonodromy Deformations.

As of now the main source of the connection results for $P_5$ with the normalization $\delta<0$ on the pure
imaginary axis is the papers \cite{MT1,MT2} by McCoy and Tang. In these papers they consider a special case of
$P_5$ \eqref{eq:P5} with the parameters $\theta_1=\theta_2\equiv\theta$ and $\theta_\infty=2n\in\mathbb Z$ with
the primary motivation to serve applications related with the study of correlation functions for the $2D$ Ising
model, level spacing distribution in the theory of random matrices and one-particle density matrix of the one-dimensional
impenetrable Bose gas. Methodologically, for studying asymptotics as $t\to\infty$, these authors applied the scheme (1)
outlined above and for asymptotics as $t\to0$ used the corresponding results obtained by Jimbo~\cite{J}. McCoy and Tang
not only proved for $t\to i\infty$ the following asymptotics
\begin{equation}
 \label{Mcf}
y=\de t^{-4\vp+\tin}e^t(1+\mathcal{O}(t^{-\ep}),
\end{equation}
where $\delta,\phi\in\mathbb C$ are parameters of the $P_5$ transcendent, but also found the connection
formulae for asymptotics of this solution as $t\to0$ and $t\to+\infty$.

The case of $P_5$ studied by McCoy and Tang is known to be equivalent to a special case of the third Painlev\'e
equation, so it is not a "truly" fifth Painlev\'e transcendent. Our main motivation for this study
was to extend the connection results to the case of the true $P_5$ transcendent by following methodology (3)
in the above list.


We also study the real reduction of the solutions; this is impossible to get real reduction from solution
(\ref{Mcf}).

Recently, two papers devoted to asymptotics on pure imaginary axis were published (see \cite{ML1} and \cite{ML2}).
In these papers the asymptotics of the general {\it real} reduction of the $P_5$ on pure imaginary axis were
obtained. These results are in complete agreement with our Corollaries \ref{cor:real-th3.1}
and \ref{cor:real-th3.2}. Also, in \cite{ML1} and \cite{ML2}, the general asymptotic
of the form
\begin{equation}
y(t)\sim 1 -c t^\sigma,
\label{ML0}
\end{equation}
found by Jimbo in \cite{J}, is rigorously and directly proved.

Our results are more advanced than that obtained in \cite{ML1} and \cite{ML2}
in several aspects. First, we establish the connection formulas which allow
one to find the asymptotic parameters at infinity for given parameters
in the asymptotic expansion at zero (say, $c$ and $\sigma$ in the previous formula). This is due to we have computed
the monodromy data which is an important result in itself. Second, we found and parameterized asymptotics
of the {\it general complex} solution, not just real ones. Third,
(\ref{ML0}) is not the only possible solution at zero: there is a lot of others
and we give a complete list of them for the case of general $\Theta$-parameters.
Finally, our approach is completely different: we use the isomonodromy deformation method (IDM).

Having said this let us note that that the proof of results obtained by IDM can be
justified with the help of the scheme suggested in \cite{K}. This scheme requires a more careful attention to error
estimates, than that presented in this paper. For the experienced reader it is clear that the estimates possess  the
properties required for launching the scheme \cite{K}. At the same time the explicit presentation of these estimates would
substantially blow up the size of the paper without adding any new information. Since we do not provide all the details
we use the word {\it derivation}, rather than the {\it proof} in the corresponding sections. It is important to mention
that there is another possibility of the justification of asymptotics obtained by IDM, it is an application of the
well-known Wasow theorem (see Theorem 35.1 of \cite{W}):
As long as the leading term of asymptotics is obtained, one can develop it into the complete asymptotic series (see
Appendix~\ref{sec:allterms}), after that the Wasow theorem implies the existence of the solution of IDS with the
prescribed asymptotics. In that scheme our derivation constitutes the proof of correspondence between coefficients
of the leading terms of asymptotics and monodromy data. The latter proof does not require any special properties of
the error estimates and is enough for the justification. Here, however, we do not give the complete details for
application of the Wasow theorem, so the word derivation is correct in this sense too.

Although, there are no doubts in the correctness of IDM, surely, there might be some ad hoc faults,
in formulae because of the personal reasons, some of them indicated below, we provide our formulae with examples
of the numerical verification, which can be useful for he reader interested in application of our results and comparison
them with the results obtained by the other authors.

We also refer to the papers~\cite{ML1,ML2} mentioned above, which contains rigorous proofs of the asymptotic results,
which coincides with special cases of our formulae.

Recently appeared a number of
papers~\cite{ShSh0,ShSh-inf,LoZeZho,ZeZha} where different justification schemes for asymptotics of the fifth
Painlev\'e transcendents has been used. They concern the results for the real axis we discuss in \cite{AK}, we expect
that they will be working for the imaginary axis too.

The paper is organized as follows.
In Section~\ref{sec:def-monodromy}, we define the monodromy data for Equation~(\ref{mainl}).
The main results are presented in Sections~\ref{sec:res} and \ref{sec:zero}, for asymptoics as $t\to i\infty$ and
$t\to0$, respectively.
In Sections~\ref{sec:der1} and \ref{sec:derII} brief derivations of the results,
stated in Theorems \ref{th:new1} and \ref{th:allmon}, respectively, are
presented. In Appendix~\ref{sec:trans}, we consider Schlesinger transformations for
$P_5$.
Using these Schlesinger transformations, we can derive Theorems \ref{th:new2}
and \ref{th:allmon} in an alternative way. In Appendix~\ref{sec:allterms}, we consider the complete asymptotic series for
solutions found in Section~\ref{sec:res}, and explain how to find the first terms of these series.
In Section~\ref{sec:mccoy}, we compare our results with those obtained
in paper \cite{MT2}. Section~\ref{sec:zero} is devoted to presentation of the results
for small argument, $t\to 0$. Comparing to our previous paper we present a refined formula
for the leading term of small $t$ asymptotics.
In Section~\ref{sec:deg0} we deal with the degenerated cases of asymptotics as $t\to 0$.

To demonstrate how our results can be used,
we consider deriving connection formulas for an equation, we met in applications.
Reader can find the details in Section~\ref{sec:spec-meromorphic-solution}.

This paper is written in stylistics close to that of the paper~\cite{AK} published quite some time ago. This
happened because the first draft of this work was written in 1997, soon after preprints ~\cite{AK1}
and \cite{AK2} were finished.
Paper~\cite{AK} is the unification of these preprints, which were originally
written in different notation. So, after \cite{AK} was published in 2000 we turn back to this work to unify
notation. Interchanging of the notation required great care and large time because we have got many complicated
formulae. We were not able to finish this work at that time because while doing it we digressed to some other
studies of the Painlev\'e equations. We came back to this work only in 2013. Since the large time past after
the paper were written and the mess with the notation we decided to check our results independently.
On this way appeared Sections~\ref{sec:spec-meromorphic-solution} and \ref{sec:numerics} where we considered some
special application of our results and undertook numerical studies, respectively. At this stage the paper took
the form close to its modern state, however, we again digressed to another studies and were not able to make the
final editing. Only after we decided to separate the 1997 draft into two parts we were able to finish the first part
in April 2019. The second part of this paper completes asymptotic  description of solutions as $t\to\imath\infty$.
It containes one type of asymptotics for general solutions and a few asymptotics for special one parameter families.

It is important to mention that during this time some interesting papers devoted to the study of
asymptotics of the fifth Painlev\'e functions has been published \cite{LNR,JR,ShSh-trunc,ShSh-cr}.

{\bf Acknowledgments.} One of the authors (FA) is grateful to
Andy Bennett and Lev Kapitanski, Kansas State University, for hospitality and support.
His work was partially supported by grants \#436--2978 and CMS--9813182.

\def\co{{\Cal O}}
\section{The Manifold of Monodromy Data}\label{sec:def-monodromy}
In this section, we define the monodromy data \cite{JM} for Equation~(\ref{mainl}).

Equation~(\ref{mainl}) has three singular points:
irregular one at the infinity and two regular singularities at $0$ and $1$.
We define the canonical solutions $Y_k=Y_k(\la)$ of Equation~(\ref{mainl}) by means
of their $|\la|\to \infty$ expansions,
\bee
Y_k(\la)\underset{|\la|\to \infty}=\left(I+
{\cal O}\left(\frac 1\la\right)\!\!\right)
\exp\left(\frac {\la t}2
\si_3-\frac{\Th_\infty}2\ln\la\si_3\right),
\label{stexp2}
\eee
in the corresponding sectors
\bee
-\frac\pi 2 +\pi (k-2)<\arg \la+\arg t<\frac{3\pi}2 +\pi (k-2),\quad
k=1,2,\dots.\label{sec2}
\eee
Henceforth, we fix the branch of $\ln \la $ in the usual way, i.e.,
$\ln \la=\ln|\la| +\imath \arg \la$, and consider  $\arg t$ as a given number.
For pure imaginary arguments $\arg t=\pm\frac \pi2$.

The canonical solutions are connected by the so called Stokes matrices,
\bee
Y_{k+1}(\la)=Y_k(\la)S_k.\label{stokes2}
\eee
Using \num{stexp2}, one easily proves that
\bee
Y_{k+2}(\la e^{2\pi \imath})
=Y_k (\la) e^{-\pii \tin \si_3}.\label{rel2}
\eee
Equations~\num{stokes2} and \num{rel2} give us
\bee
S_{k+2} =e^{\pi\imath
\Th_\infty \si_3}S_k e^{-\pi\imath \Th_\infty \si_3}.
\label{rec2}
\eee
Thus, one can determine all the Stokes matrices having two of them. We choose $S_1$
and $S_2$ to be these basic matrices.
It follows from Equations~(\ref{mainl}), \num{stexp2}, \num{sec2},
and \num{stokes2} that matrices $S_k$ are independent of $\la$ and of the
following structure:
\bee
S_{2k}=\bm
1&s_{2k}\\ 0&1\em,\quad
S_{2k+1}=\bm 1&0\\s_{2k+1}& 1\em,\quad k\in \Bbb Z.
\label{stst2}
\eee
The complex parameters $s_k$ $(k=1,2,\dots)$ are called
{\it Stokes multipliers}. The monodromy matrix at infinity, $M^\infty_k$,
for function $Y_k$ is defined by the following equation:
\bee
Y_k(\la e^{-2\pii})=Y_k (\la )M^\infty _k.\label{md2}
\eee
Using Equations~\num{stokes2} and \num{rel2}, one finds
\bee
M_k^\infty= S_k S_{k+1} e^{\pii \tin \si_3}.
\label{mi2}
\eee
In the following, we use only one of these matrices
$$
M_2 ^\infty\eq M^\infty= S_2 e^{\pii \tin \si_3}S_1.
$$
All other matrices $M^\infty_k$ can be expressed in terms of $M^\infty$
via the following recurrence relation
\bee
M_{k+1}^\infty =S_k^{-1} M^\infty_k S_k.
\label{recm2}
\eee

To deal with monodromy matrices at finite singularities at $0$ and $1$ we define a single-valued branch of $Y_2$.
It is convenient
to restrict $Y_2$ to the domain
$\bar {\Bbb C}\back ([0,1]\cup [1/2, \infty e^{\frac{-\imath \pi}2}])$.
In this domain
$Y_2(\la)$ is a single-valued analytical function (of $\la$) with
the following expansions at the regular
singularities $\la=p$, $p=0,1$:
\begin{eqnarray}
&Y_2(\la) \underset{\la\to p}=\sum^\infty_{n=0}\hat Y_n^p(\la-p)
^{n+\Th_p \si_3/2} E^p, \label{y22}&\\
&\hat Y_n^p= \hat Y_n^p (t,z, y,u, \Th_0,\Th_1,\Th_\infty),\quad
E^p=E^p(t,z, y,u, \Th_0,\Th_1,\Th_\infty), &\nonumber\\
&\det \hat Y_0^p
=\det E^p =1, &\nonumber\\
&(\hat Y_0^p)^{-1} A_p(t) \hat Y_0^p =\frac{\Th_p}2\si_3.\nonumber&
\end{eqnarray}
We assume here that $\Th_p\not \in \Bbb Z$. If $\Theta_p$ is an integer
the expansion is modified as it is written in \cite{AK1}.
The series in Equation~\num{y22} is convergent for $|\la-p|<1$.

Using expansions \num{y22}, we define the monodromy
matrices at the regular singular points as follows:
\bee
 \label{eq:mj2}
M^p =\big(E^p\big)^{-1} e^{\pi\imath\Th_p\si_3}E^p.
\eee
Even though expansion {\rm\num{y22}} does not define $E^p$ uniquely,
the matrices $M^p$ are defined properly.

The monodromy matrices are connected by
the following cyclic relation \cite{J}
\bee
M^\infty M^1 M^0=I.\label{crel2}
\eee
The matrix elements $m_{ij}^0$, $m_{ij}^1$
$(i,j=1,2)$, Stokes multipliers, $s_k$, $k=1,2$, and the parameters
$\Th_0,\Th_1,\Th_\infty\in \Bbb C$ are called the {\it monodromy data} of
Equation(\ref{mainl}). It is easy to check that the data satisfy the following relations:
\begin{eqnarray}
&\det M^0=1, \ \det M^1=1, &\nonumber\\
&\tr M^0 =2\cos \pi \Th_0,
\ \tr M^1
=2\cos \pi \Th_1,\ (M^1 M^0) _{11} =e^{-\pi \imath \Th_\infty}. &
\label{idm2}
\end{eqnarray}
System.~(\ref{idm2}) in $\Bbb C^8 \ni (m_{11}^0, m_{12}^0, m^0_{21}, m^0_{22},
m_{11}^1, m_{12}^1, m^1_{21}, m^1_{22})$,
for fixed $\Th_0,\Th_1,\tin \in \Bbb C$ define an algebraic variety
which is called the {\it manifold of monodromy data},
$\Cal M_5(\0,\1,\tin)$. In terms of
$m^p_{ij}$, Eqs.~\num{idm2} read as:
\begin{eqnarray}
&m_{11}^0 m_{22}^0 -m^0_{12} m_{21}^0=1, \quad
m_{11}^1 m_{22}^1 -m^1_{12} m_{21}^1=1, \label{eq:det-monmat1}&\\
&m_{11}^0+m_{22}^0=2\cos \pi \Th_0,\quad
m_{11}^1+m_{22}^1=2\cos \pi \Th_1,\label{id2}&\\
&m^1_{11} m_{11}^0 +m_{12}^1 m_{21}^0 =e^{-\pi \imath \Th_\infty}.\label{eq:M0M1}&
\end{eqnarray}
Note that $\op{dim}_{\Bbb C} \cm_5(\0,\1,\tin)=3$. Given a point in $\cm_5(\0,\1,\tin)$ (that is matrices $M^0$ and $M^1$),
one finds the Stokes multipliers via relations \num{mi2} and \num{crel2}.
Thus, $\cm_5(\0,\1 ,\tin)$ completely defines all the monodromy data.

The given point $(t,z,y,u,\0,\1,\tin)\in \Bbb C^7$
defines Matrices~\eqref{eq:mat0} and \eqref{eq:mat1} that, in its turn, Equation~(\ref{mainl}).
A solution to {\it direct monodromy problem} is a correspondence
$$
(t,z,y,u,\0,\1,\tin)\to \text{Equation~(\ref{mainl})}\to \cm\in
\cm_5(\0,\1,\tin).
$$
An inverse map
$$
\{\cm,t\}\to (t,z,y,u,\0,\1,\tin),
$$
where
$\cm\in\cm_5(\0,\1,\tin)$ and $t\in \Bbb C$,
is a solution to {\it inverse monodromy problem}.
If the inverse monodromy problem is solvable, its solution is unique.
If we demand that point $\cm\in \cm_5(\0,\1,\tin)$ does not move when we vary $t$, the corresponding solution of the inverse monodromy problem, i.e.,
functions $y=y(\cm,t)$, $z=z(\cm,t)$, and $u=u(\cm,t)$, are called the
{\it isomonodromy deformations}, which means that the monodromy data do not change. The main result is that in the case of isomonodromy deformations the functions $y=y(\cm,t)$, $z=z(\cm,t)$, and $u=u(\cm,t)$ satisfy Eqs.~(\ref{eq:ids1})--(\ref{eq:ids3}) \cite{JM}.

Only in exceptional cases the direct and/or inverse monodromy problems can
be solved explicitly. Therefore, we have to apply asymptotic methods.
In this paper we solve the direct monodromy problem asymptotically
as $t\to\infty$ with $\Re(t) =\mathcal{O}(1)$ and, also, for $t\to 0$. We find asymptotics, parameterized
by some complex numbers and present explicit formulas for the
monodromy data in terms of these parameters.

\section{Results}\label{sec:res}
The asymptotic results formulated in this section are valid in the cheese-like domains, $\mathcal D^1_\pm$ and
$\mathcal D^2_\pm$, in the complex $t$-plane. The subscripts $\pm$ means that the positive, respectively, negative
imaginary semiaxes, beginning with some finite point, belong to the corresponding domains. The domains with different
subscripts do not intersect, the different superscripts of the domains with the same subscripts means different
locations of the holes inside the domains. An important property of these domains is that the solution of
System~\eqref{eq:ids1}--\eqref{eq:ids3}, the set of functions $y(t)$, $z(t)$, and $u(t)$, as well as their asymptotics
restricted in the domains are singlevalued analytic functions. This fact is important for justification of asymptotics
as well as the study of distribution of zeroes and poles of the solutions along the imaginary axis.

The precise definition of the domains are given in Theorems~\ref{th:new1} and \ref{th:new2} below, before we specify
the branch of function $t^\nu$ with $\nu\in\mathbb C$, which we use in our domains to write our asymptotic formulae.
We define $t^\nu$ on the imaginary axis as $e^{\frac{\imath\pi\nu}2} |t|^\nu$ for $\Im(t)>0$ and
$e^{-\pi \imath \frac\nu2} |t|^\nu$ for $\Im(t)<0$ and extend it on the entire domains $\mathcal D^1$ and
$\mathcal D^2$  via the analytic continuation.

\begin{theorem}\label{th:new1}
Let $\varphi\in\mathbb C$ and
$\varphi\mp\tfrac{\Theta_0}2,\varphi-\tfrac{\Theta_\infty}2\mp\tfrac{\Theta_1}2\neq0,-1,-2,\ldots$,
and $\de,\hat u \in \Bbb C\back \{0\}$.
Denote
$$
R_1(t;\vartheta)=1-\dfrac{\varphi-\vartheta}{\delta t^{\nu_1}e^t},\quad\text{where}\quad
\nu_1=1+\tin -4 \vp
$$
and assume $-1/2<\Re(\nu_1)<1$. Then, for each value of the sign $\pm$
there exists the unique solution of System {\rm\eqref{eq:ids1}}--{\rm\eqref{eq:ids3}}
with the following asymptotic expansion
\begin{align*}
yt&=\delta t^{\nu_1}e^tR_1\Big(t;\dfrac{\Theta_0}2\Big)R_1\Big(t;\dfrac{\Theta_1+\Theta_\infty}2\Big)
+\mathcal O(t^{-3\nu_1-1}\ln t)+\mathcal O(t^{\nu_1-1}\ln t),&\\
z&=-\Theta_0-\Big(\varphi-\dfrac{\Theta_0}2\Big)R_1\Big(t;\dfrac{\Theta_1+\Theta_\infty}2\Big)
+\mathcal O(t^{-1})+\mathcal O(t^{-3\nu_1-1}\ln t)&\\
&=-\frac{\Theta_0+\Theta_1+\Theta_\infty}2
-\Big(\varphi-\dfrac{\Theta_1+\Theta_\infty}2\Big)R_1\Big(t;\dfrac{\Theta_0}2\Big)
+\mathcal O(t^{-1})+\mathcal O(t^{-3\nu_1-1}\ln t),&\\
u&=\dfrac{\hat u}\delta\cdot
\dfrac{t^{2\vp}\big(1+\mathcal O\big(t^{-1}\ln t\big)+\mathcal O\big(t^{-2\nu_1-1}\ln t\big)\big)}
{R_1\Big(t;\dfrac{\Theta_1+\Theta_\infty}2\Big)+\mathcal O(t^{-\nu_1-1}\ln t)+\mathcal O(t^{-3\nu_1-1}\ln t)},&
\end{align*}
as $t\to\infty$ and $\arg t\to\pm\frac\pi2$ with $t\in\mathcal D^1$, where
$$
\mathcal D^1:=
\left\{t\in\mathbb C, |\Re(t)|< r_1, -\pi<\arg t<\pi, |t-t^{0,1}_n|\geq r_2|t|^{\varepsilon-1},
0<\varepsilon\leq1,\forall\,r_1,r_2>0\right\},
$$
where the sequences, $t^0_n$ and $t^1_n$, are infinite series $(n=1,2,\ldots)$ of solutions (if they exist) of
the equations
\begin{equation}\label{eq:zeroes}
R_1\Big(t^0_n;\dfrac{\Theta_0}2\Big)=0,\qquad
R_1\Big(t^1_n;\dfrac{\Theta_1+\Theta_\infty}2\Big)=0,
\end{equation}
with $|\Re(t^{0,1}_n)|<r_1$.
\end{theorem}
\begin{remark}\label{rem:estimates}{\rm
The notation $t\to\infty$ and $\arg t\to\pm\frac\pi2$ with $t\in\mathcal D^1$ means that asymptotics hold for all
rather large $t\in\mathcal D^1$ with either $\Im t>0$ or, respectively, $\Im t<0$. We recall that the functions
$y(t)$, $z(t)$, $u(t)$ have the branching point at $t=0$ (cf. Section~\ref{sec:zero}) so that it is important to
specify the argument of $t$.
All error estimates depend on all the parameters: $\varphi$, $\delta$, $\hat u$, and $\Theta$'s, including those
characterizing the domains $\mathcal D^1$: $r_1$, $r_2$, and $\varepsilon$.
Exclusion of the negative integer values of $\varphi-\frac{\Theta_0}2$ and $\varphi-\frac{\Theta_1+\Theta_\infty}2$
does not follow from our derivation. This requirement is dictated by the justification scheme outlined in the
Introduction: in the case of negative integer values of these parameters one cannot uniquely specify the corresponding
solution by the monodromy data (cf. Theorem~\ref{th:allmon} and Equations~\eqref{eq:det-monmat1}--\eqref{eq:M0M1}).
It does not mean that solutions with the corresponding asymptotics do not exist: just our calculation requires
a minor modification. More specifically, in this case we have to calculate the other monodromy parameters than that
given in Equations~\eqref{eq:m1-12}--\eqref{eq:m0-21} below. It is important to mention that although
Theorems~\ref{th:new1} and \ref{th:new2} do not refer to the monodromy data, our way of proving them is based on the
monodromy correspondence established in Theorem~\ref{th:allmon}.}
\end{remark}
\begin{remark}\label{rem:zeroes}{\rm
It is clear that infinite series of solutions of Equations~\eqref{eq:zeroes} satisfying the condition
$|\Re(t^{0,1}_n)|<r_1$ exist only in the case $\Re(\nu_1)=0$. Thus, in the case $\Re(\nu_1)\neq0$ the domain
$\mathcal D^1$ is just a strip along the imaginary axis ("the cheese without holes")
incised along the segment $[-r_1,0]$.
One can prove that for all rather large $t$ in any circle with small enough radius (see definition of $\mathcal D^1$)
centred at zeroes of asymptotics, $t^{0,1}_n$, there exists one and only one zero of the Painlev\'e function $y(t)$.
Therefore solutions described in Theorem~\ref{th:new1} do not have poles in $\mathcal D^1$.}
\end{remark}
\begin{remark}\label{rem:mu-modifications}{\rm
Instead of taking the imaginary axis $(\Re(t)=0)$ as the axis of the domain $\mathcal D^1$
our derivation presented below with little modifications works for a
"logarithmic deformation" of the imaginary axis, namely, $\Re(t)=\mu_1\ln|t|$ for any $\mu_1\in\mathbb R$.
In this case instead of $\mathcal D^1$ we can formulate our result in $r_1$-neighborhood of the "deformed
imaginary axis", which can be denoted as $\mathcal D^1(\mu_1)$. In this case all asymptotics announced
in Theorem~\ref{th:new1} are valid and the condition on $\nu_1$ should be changed to
$-1/2<\mu_1+\Re(\nu_1)<1$. In the error estimates one also has to change $\nu_1\to\mu_1+\nu_1$.
It is interesting to note that points of the logarithmic curve satisfy the asymptotic condition
$t\to\pm\imath\infty$ in the sense that $\arg t\to\pm\pi/2$ as $|t|\to+\infty$.

Because the parameter $\mu_1$ is arbitrary we can consider slightly more complicated domains,
than $\mathcal D^1(\mu_1)$, where the asymptotics still hold, e.g., if $0<\mu_1<3/2$ we can
write our asymptotics in the $r_1$-neighborhood of the domain bounded on the right with the logarithmic curve and
on the left with the imaginary axis. In this case the parameter $\Re(\nu_1)$ is bounded as follows,
$-1/2<\Re(\nu_1)<1-\mu_1$. We can also consider domains which are bounded on the right and on the left with the
logarithmic curves with positive and negative values of the parameter $\mu_1$, or with the same sign of $\mu_1$ in
the latter case they would not contain the imaginary axis. These asymptotics in the "logarithmic" domains allow us
to establish existence of infinite sets of zeroes, $t^{0,1}_n\to\pm\imath\infty$ as $n\to+\infty$, for
$\Re(\nu_1)+\mu_1=0$ which logarithmically (with respect to $n$) moving away from the imaginary axis.}
\end{remark}
\begin{theorem}\label{th:new2}
Let $\varphi\in\mathbb C$ and
$\varphi\mp\tfrac{\Theta_0}2,\varphi-\tfrac{\Theta_\infty}2\mp\tfrac{\Theta_1}2\neq0,-1,-2,\ldots$,
and $\de,\hat u \in \Bbb C\back \{0\}$.
Denote
$$
R_2(t;\vartheta)=1-(\varphi+\vartheta)\delta t^{-\nu_2}e^t,\quad\text{where}\quad
\nu_2=1-\Theta_\infty+4\varphi
$$
and assume $-1/2<\Re(\nu_2)<1$. Then, for each value of the sign $\pm$
there exists the unique solution of System {\rm\eqref{eq:ids1}}--{\rm\eqref{eq:ids3}}
with the following asymptotic expansion
\begin{align*}
\frac ty&=\dfrac{R_2\Big(t;\dfrac{\Theta_0}2\Big)R_2\Big(t;\dfrac{\Theta_1-\Theta_\infty}2\Big)}{\delta t^{-\nu_2}e^t}
+\mathcal O(t^{-3\nu_2-1}\ln t)+\mathcal O(t^{\nu_2-1}\ln t),&\\
z&=-\Big(\varphi+\dfrac{\Theta_0}2\Big)R_2\Big(t;\dfrac{\Theta_1-\Theta_\infty}2\Big)
+\mathcal O(t^{-1})+\mathcal O(t^{-3\nu_2-1}\ln t)&\\
&=-\frac{\Theta_0-\Theta_1+\Theta_\infty}2
-\Big(\varphi+\dfrac{\Theta_1-\Theta_\infty}2\Big)R_2\Big(t;\dfrac{\Theta_0}2\Big)
+\mathcal O(t^{-1})+\mathcal O(t^{-3\nu_2-1}\ln t),&\\
u&=\dfrac{\hat u}\delta
t^{2\vp}\left(1+\mathcal O\big(t^{-1}\ln t\big)+\mathcal O\big(t^{-2\nu_2-1}\ln t\big)\!\right)\times\\
\times&\left(R_2\Big(\!t;\dfrac{\Theta_1-\Theta_\infty}2\!\Big)\!+\mathcal O\big(t^{-\nu_2-1}\ln t\big)+
\mathcal O\big(t^{-3\nu_2-1}\ln t\big)\!\right),&
\end{align*}
as $t\to\infty$ and $\arg t\to\pm\frac\pi2$ with $t\in\mathcal D^2$, where
$$
\mathcal D^2:=
\left\{t\in\mathbb C,|\Re(t)|\leq r_1, -\pi<\arg t<\pi, |t-t^{2,3}_n|\geq r_2|t|^{\varepsilon-1},
0<\varepsilon\leq1,\forall\,r_1,r_2>0\right\},
$$
where the sequences, $t^2_n$ and $t^3_n$, are infinite series $(n=1,2,\ldots)$ of solutions (if they exist) of
the equations
\begin{equation}\label{eq:poles}
R_2\Big(t^2_n;\dfrac{\Theta_0}2\Big)=0,\qquad
R_2\Big(t^3_n;\dfrac{\Theta_1-\Theta_\infty}2\Big)=0,
\end{equation}
with $|\Re(t^{2,3}_n)|<r_1$.
\end{theorem}
\begin{remark}\label{rem:th:new2-gen}{\rm
Remark~\ref{rem:estimates} holds with the change $\mathcal D^1\to\mathcal D^2$.
Remark~\ref{rem:zeroes} is also valid if one replaces: $t^{0,1}_n$ with $t^{2,3}_n$,
$\nu_1$ with $\nu_2$, Theorem~\ref{th:new1} with Theorem~\ref{th:new2},
and considers zeroes instead of poles. Remark~\ref{rem:mu-modifications} also can be reformulated
for the results stated in Theorem~\ref{th:new2} as existence of the infinite sequence of poles $t^{2,3}_n$,
$n=1,2,\ldots$, which diverges from the imaginary axis logarithmically. If we introduce the parameter
$\mu_2$ instead of $\mu_1$: $\Re(t)=\mu_2\ln|t|$, than the range of the validity of the asymptotics in
the $\mu$-deformed domain can be described as $-1/2<-\mu_2+\Re(\nu_2)<1$, and in the error estimates we must
change $\nu_2\to\nu_2-\mu_2$.}
\end{remark}

\begin{remark}\label{rem:interlace-th3.1-th3.2}{\rm
As mentioned at the end of Section~{\rm\ref{sec:derII}} and in Subsection~{\rm\ref{subsec:asympt-th3.1}}
the results reported in Theorems~{\rm\ref{th:new1}} and {\rm\ref{th:new2}} continue to describe the qualitative
behavior of function $y$ beyond the intervals of validity of the theorems, namely, for
$1\leq\Re(\nu_k)<2$, where $k=1,2$.
The error estimates in Theorem~{\rm 3.k} for the functions $y$ and $z$ equal $\mathcal O(t^{2\nu_k-2}\ln t)$ and
$\mathcal O(t^{\nu_k-2})$, respectively.
Since $\nu_1+\nu_2=2$, then for $1\leq\Re(\nu_k)<2$, we get $0<\Re(\nu_{3-k})\leq1$; thus Theorem~{\rm3.(3-k)} gives much
better approximation than Theorem~{\rm3.k}.
In the domain where at least one of the parameters $\nu_k$ satisfies the condition $1\leq\Re(\nu_k)<2$, which can be rewritten
in terms of $\varphi$ as $|\Re(4\vp-\tin)|<1$, the first {\rm(}largest{\rm)} terms of asymptotics given by
both Theorems~{\rm\ref{th:new1}} and {\rm\ref{th:new2}} coincide:
\begin{eqnarray*}
y&=&\de t^{-4\vp+\tin}e^t\left(1+\mathcal O\big(t^{|\Re(4\varphi-\Theta_\infty)|-1}\big)\right),\\
z&=&-\vp-\frac{\Th_0}2+\mathcal O\big(t^{|\Re(4\vp-\tin)|-1}\big),\\
u&=&\frac{\hat u}\de t^{2\vp}
(1+\mathcal O\big(t^{|\Re(4\vp-\tin)|-1}\big)+\mathcal O\big(t^{-1}\big)\ln(t))
\end{eqnarray*}
}\end{remark}
\begin{corollary}\label{cor:results-zeta}
For solutions defined in Theorems {\rm3.1} and {\rm3.2} the corresponding $\zeta$-function~\eqref{eq:zeta-def} has the
following asymptotics as $t\to\infty$, $\arg t\to\pm\frac \pi 2$, and $t\in\mathcal D^1\cup\mathcal D^2$:
\begin{equation}\label{eq:results-zeta}
\begin{split}
\ze&=\Big(\vp+\frac {\Th_0}2\Big)t-2\Big(\varphi+\frac{\Theta_0}2\Big)\Big(\varphi-\frac{\Theta_\infty+\Theta_0}2\Big)\\
&+\frac1t\left(-\de t^{-4\vp+\tin}e^{t}\Big(\varphi+\frac{\Theta_0}2\Big)\right.
\Big(\varphi-\frac{\Theta_\infty-\Theta_1}2\Big)\\
&\phantom{mmm}-2\Big(\varphi^2-\frac{\Theta_0^2}4\Big)\Big(\varphi-\frac{\Theta_\infty}2\Big)
-2\varphi\Big(\Big(\varphi-\frac{\Theta_\infty}2\Big)^2-\frac{\Theta_1^2}4\Big)\\
&\phantom{mmm}\left.+\de^{-1} t^{4\vp-\tin}e^{-t}\Big(\varphi-\frac{\Theta_0}2\Big)\Big(\varphi-\frac{\Theta_\infty+\Theta_1}2\Big)\right)
+\mathcal O\big(t^{|\Re(-4\vp+\tin)|-2}\big).
\end{split}
\end{equation}
\end{corollary}
\begin{remark}\label{rem:cor3.1}{\rm
The parameter $\varphi$ in asymptotics~\eqref{eq:results-zeta} satisfies the condition $|\Re(-4\vp+\tin)|<3/2$.
In the case $|\Re(-4\vp+\tin)|<1/2$ all explicitly written terms in \eqref{eq:results-zeta} larger than the
error estimate; if $1/2\leq|\Re(-4\vp+\tin)|<1$ or $1\leq|\Re(-4\vp+\tin)|<3/2$, then one or, respectively, two terms
in the right-hand side of Equation~\eqref{eq:results-zeta} are equal or smaller the error estimate and can be neglected.
}\end{remark}
\begin{theorem}\label{th:allmon}
Solutions of System~\eqref{eq:ids1}--\eqref{eq:ids3}
described in Theorems~{\rm3.1} and {\rm3.2} define
the isomonodromy deformations of Equation~\eqref{mainl}
with the following monodromy data:
\begin{eqnarray}\label{eq:m1-12}
m^1_{12}&=&
\frac{-2\pii \hat u}
{\Ga(1-\frac{\Th_1+\tin -2\vp}2)
\Ga(\frac{\Th_1-\tin+2\vp}2)}
,\\
m^0_{21}&=&
\frac{-2\pii \de e^{-\pi\imath \tin}}
{\hat u\Ga(1-\frac{\Th_0-2\vp}2)
\Ga(\frac{\Th_0+2\vp}2)}.\label{eq:m0-21}
\end{eqnarray}
If asymptotic expansions of solutions in these theorems are understood to be given for $\arg t\to\frac \pi 2$, then
\begin{equation}\label{mr1}
m^1_{11}=
e^{2\pii\vp -\pii \tin},
\end{equation}
if\;$\arg t\to-\frac \pi 2$, then
\begin{equation}\label{mr2}
m^0_{11}=
e^{-2\pii\vp}.
\end{equation}
\end{theorem}
\begin{corollary}\label{cor:real-th3.1}
Denote
$$
\imath\,\Theta_{01}\equiv\hat\gamma=1-\Theta_0-\Theta_1\quad\text{and}\quad
\imath\,\omega_1\equiv\nu_1=1+\Theta_\infty-4\varphi.
$$
Assume that $\Theta_{01}$ and $\omega_1$ as well as the coefficients of Equation~\eqref{eq:P5}, $\hat\alpha$ and $\hat\beta$,
are real and
$$
\left(\frac{\omega_1-\Theta_{01}}2\right)^2\!>2\hat\beta.
$$
Then the solution defined in Theorem~{\rm\ref{th:new1}} is real for the pure imaginary values of $t$, namely,
$y(t)\equiv\tilde y(\tau)$, $t=\imath\tau$, $\tilde y(\tau)\in\mathbb R$ for $\tau\in\mathbb R$, and its
asymptotics as $t\to\infty$ and $\arg t\to\epsilon\pi/2$ $(\epsilon=\pm1)$ can be rewritten as follows:
\begin{equation}\label{eq:theorem3.1-real}
\tilde y(\tau)\underset{\tau\to\epsilon\infty}=\frac1\tau\left(\frac{\omega_1-\Theta_{01}}2+
\sqrt{\left(\frac{\omega_1-\Theta_{01}}2\right)^2\!-2\hat\beta}\;\sin\left(\tau+\omega_1\ln|\tau|+\arg\delta\right)
+\mathcal O\left(\frac{\ln\tau}{\tau}\right)\right),
\end{equation}
where $\delta$ is the parameter introduced in Theorem~{\rm\ref{th:new1}}. For real solutions with
asymptotics~\eqref{eq:theorem3.1-real} $|\delta|$ is given by the relation
$$
2|\delta|\,e^{-\frac{\epsilon\pi\omega_1}2}=\sqrt{\left(\frac{\omega_1-\Theta_{01}}2\right)^2\!-2\hat\beta}\,>\,0.
$$
\end{corollary}
\begin{corollary}\label{cor:real-th3.2}
Denote
$$
\imath\,\Theta_{01}\equiv\hat\gamma=1-\Theta_0-\Theta_1\quad\text{and}\quad
\imath\,\omega_2\equiv\nu_2=1-\Theta_\infty+4\varphi.
$$
Assume that $\Theta_{01}$ and $\omega_2$ as well as the coefficients of Equation~\eqref{eq:P5}, $\hat\alpha$ and $\hat\beta$,
are real and
$$
\left(\frac{\omega_2-\Theta_{01}}2\right)^2\!>-2\hat\alpha.
$$
Then the solution defined in Theorem~{\rm\ref{th:new2}} is real for the pure imaginary values of $t$, namely,
$y(t)\equiv\tilde y(\tau)$, $t=\imath\tau$, $\tilde y(\tau)\in\mathbb R$ for $\tau\in\mathbb R$, and its
asymptotics as $t\to\infty$ and $\arg t\to\epsilon\pi/2$ $(\epsilon=\pm1)$ can be rewritten as follows:
\begin{equation}\label{eq:theorem3.2-real}
\frac1{\tilde y(\tau)}\underset{\tau\to\epsilon\infty}=-\frac1\tau\left(\frac{\omega_2-\Theta_{01}}2+
\sqrt{\left(\frac{\omega_2-\Theta_{01}}2\right)^2\!+2\hat\alpha}\:\sin\left(\tau-\omega_2\ln|\tau|+\arg\delta\right)
+\mathcal O\left(\frac{\ln\tau}{\tau}\right)\right),
\end{equation}
where $\delta$ is the parameter introduced in Theorem~{\rm\ref{th:new1}}. For real solutions with
asymptotics~\eqref{eq:theorem3.2-real} $|\delta|$ is given by the relation
$$
\frac{2e^{-\frac{\epsilon\pi\omega_2}2}}{|\delta|}=\sqrt{\left(\frac{\omega_2-\Theta_{01}}2\right)^2\!+2\hat\alpha}\,>\,0.
$$
\end{corollary}
\section{Derivation I}\label{sec:der1}
In this Section we asymptotically, as $t\to\infty$, solve the direct monodromy problem for Equation~\eqref{mainl}
by making several assumptions on asymptotical behavior of its coefficients. First of all, we do all
our calculations in the cheese-like strip domain along the imaginary axis. There are two real positive parameters
characterizing this domain: the half-width of the strip, $\delta_1>0$, and the radius of its holes, $\delta_2>0$.
These parameters are assumed to be fixed in the course of the calculations so that the error estimates depend on these
parameters. These holes are assumed to contain possible poles of $y$ and $z$, and zeroes of $y$, to avoid problems
with the estimates of coefficients of Equation~\eqref{mainl}. The exact location of the  centres of these holes
are unknown in the "first round" of our calculations, we assumed only the conditions on the functions $y$ and $z$
imposed below \eqref{zdem} and \eqref{ydem}. In the "second round" of calculations we put centres of the holes
exactly at zeroes and poles of the leading term of asymptotics which we find at the end of the first round.

The notation $t\to\pm i\infty$ more precisely means that we are taking a limit along any path in the
cheese-like domain: $|\Re t|<\delta_1$ and $\Im t\to\pm\infty$. The choice of the path is not important because of
the Painlev\'e property of System~\eqref{eq:ids1}--\eqref{eq:ids2}.
All our functions of $t$, say, $y(t)$, $z(t)$, $u(t)$ etc. are assumed to be analytic continuation from the positive
real axis. These functions have only two singular points at $0$ and $\infty$. After we make a cut along the
negative real semiaxis the analytic continuations mentioned above are correctly defined. The zeroes and poles of the
coefficients of Equation~\eqref{mainl}, if any, can be located only in the holes of the strip domain.

Our main assumptions on the coefficients in this Section are as follows:
\begin{align}
&|z|<\mathcal  O(t), \label{zdem}\\
\mathcal O\left(t^{-1}\right)<&|y| <\mathcal O(t).\label{ydem}
\end{align}
These asymptotic restrictions are assumed to be valid in the closure of the cheese-like domain. We use them in most
calculations in this section. Some further restrictions will appear in course of calculations and will be
clearly indicated in the corresponding places\footnote{\label{ftn:radius}
The radius of holes ($\delta_2$) is fixed as a positive parameter, however it is important to note that this
radius can be chosen even merging, $\delta_2=\mathcal O\big(t^{-\epsilon\big)},0<\epsilon<1/2$. We do not use
this fact here.}
Let us explain our convention for the use of the $o$ and $\mathcal O$ notation:
When we write $w=o(1)$ we actually mean that there exists some $C>0$ and $\epsilon>0$ such that
$|w|\leq C|t|^{-\epsilon}$, notation $\mathcal O(t^a)<|w|<\mathcal O(t^b)$ with $a<b$ means that
$C|t|^{a+\epsilon}\leq |w|\leq C|t|^{b-\epsilon}$.

In this section the direct monodromy problem is solved asymptotically for Equation~\eqref{mainl}
with coefficients in some classes of functions
analytic in the cheese-like strip domain and satisfying certain asymptotic conditions. Some of these conditions have
a form of simple systems of algebraic equations that can be uniquely resolved to explicitly give the leading terms of
asymptotics for the functions. One proves that thus obtained solution satisfies all the other conditions imposed
in the process of solving of the direct monodromy problem. Now when we have explicit formulae for asymptotics we can check that all our error estimates smoothly depends on the monodromy parameters. In particular, the estimates hold
under small local variations of the monodromy data.  Due to the way our asymptotics are obtained and because they are
parameterized with the monodromy data, we can say that they represent an asymptotic solution of the inverse monodromy
problem.

It is not immediately obvious that any asymptotic solution of the inverse monodromy problem represents an asymptotic
expansion of some solution of the system~\eqref{eq:ids1}--\eqref{eq:ids3}. However, there is a justification scheme~\cite{K} that allows one to prove (exact) solvability of the corresponding monodromy problem as long as its asymptotic solution is obtained via the method explained above.

To simplify the notation we perform the following gauge transformation,
\begin{equation}
Y_2=u^{\frac 12\si_3}
Yu^{-\frac12\si_3}.
\label{mainwu}
\end{equation}
Then we observe that function $Y$ satisfies equation (\ref{mainl}) but with $u=1$.
We will compute the monodromy data for function $Y$. For matrices $A_p$, $p=0,1$,
with $u=1$, we use the same notation, $A_p$.
The function $u$ will be restored in the final formulae for the monodromy data.

Another convention we follow is that in course of calculations we use notation
$\varphi$ for function $\varphi(t,\Theta_0,\Theta_1,\Theta_\infty)$. This function has the
following asymptotic evaluation as $t\to\infty$, $\varphi=\tilde\varphi+o(1)$, where
$\tilde\varphi\in\mathbb C$ is a parameter, i.e., independent of $t$ and $\Theta$'s variables.
In formulation of the results of solution of the inverse monodromy problem we use a simpler notation
$\varphi$ in the sense of $\tilde\varphi$.

The reader will find below two types of equalities: exact and asymptotic. All asymptotic equalities
with respect to $t$ in this section are understood as $t\to\infty$ in the cheese-like strip domain;
in case an asymptotic equality is understood in some other sense, say, with
respect to $\lambda$ the latter is explained. We also use notation $\approx$ to indicate asymptotic
equalities modulo lesser terms.
\subsection{WKB-method}
 \label{subsec:Der1WKB}
To obtain the monodromy data in terms of parameters of asymptotic expansion
for large pure imaginary $t$,
we apply the WKB-method.

Let us start from the exact formula
for $l^2(\la)=-(1/t^2) \det A$
\begin{eqnarray}
l^2(\la)&=&\frac1{4t^2\la^2(\la-
1)^2}\times\label{l}\\
&\times&\big(t^2\prod_{p=0,1}(\la-p)^2+
4\la(\la-1)(t\vp-\frac{t\lambda\tin}2)+
\sum_{p=0,1}\Th_{1-p}^2(\la-p)^2\big),
\nonumber
\end{eqnarray}
where
\begin{equation}\label{eq:varphi}
\vp=-z-\frac{\Th_0}2+\frac 1t\tr(A_1A_0).
\end{equation}
In addition to (\ref{zdem}) and (\ref{ydem}), we will also suppose that
\begin{equation}
\vp=\mathcal{O}(1).
\label{as1}
\end{equation}
Due to this, we may rewrite $l(\la)$ so that
\begin{equation}\label{eq:l-as-t}
l\eq l(\la)\underset{t\to\infty}{=}\frac 12+\frac1{\la(\la-1)t}
\left(\vp-\frac{\la\tin}2\right)
+\mathcal{O}\left(\frac1{t^2\lambda^2}\right)
+\mathcal{O}\left(\frac1{t^2(\lambda-1)^2}\right)
\end{equation}
The error estimate in Equation~\eqref{eq:l-as-t} is valid provided the following redefinition of $\varphi$ is made,
$$
\varphi\rightarrow
\varphi+2\varphi\left(2\varphi-\theta_\infty\right)/t.
$$
This redefinition is assumed below. It does not effect on the following calculations since we calculate $\varphi$
up to the order $o(1)$.

Then, we define
\begin{equation}\label{eq:def-F0}
F_0(\la)=t\int l(\la)\,d\la.
\end{equation}
Clearly this function is defined up to an arbitrary function of $t$, which does not play any role because in the
following we consider the definite integral (see Equation~\eqref{eq:def-WKB}). It is obvious that there exists
function $F_0(\lambda)$ with the following asymptotics as $t\to\infty$
\begin{equation}\label{eq:F0-t-estimate}
F_0(\la)=\frac{\la t}2+
\vp\ln(\la-1)-\vp\ln\la-\frac{\tin}2\ln
(\la-1)+o(1).
\end{equation}
provided $t(\lambda-p)\geq |t|^\epsilon$ for $p=0,1$ and $0<\epsilon<1$.
Here $\ln \la=\ln|\la|$ as
$\arg\la=0$ and $\ln(\la-1)\to\ln\la$
as $\la\to\infty$.
We need the following asymptotic expansions
of function $F_0(\la)$
\begin{eqnarray}\label{eq:F0-infty}
F_0(\la)&=&\frac{\la t}2-\frac{\tin}2\ln \la
+\mathcal{O}(\frac1\la)\mbox{ as }\la\to \infty,
\\
F_0(\la)&=&\frac{\la t}2
-\vp\ln \la+\pii \vp-
\frac{\tin}2\pii+\mathcal{O}(\la)
\mbox{ as }\la\to 0,
\label{eq:F0-0}\\
F_0(\la)&=&\frac t2+\frac{(\la-1)t}2+
(\vp-\frac{\tin}2)\ln(\la-1)
+\mathcal{O}(\la-1)
\mbox{ as }\la\to 1.
\label{eq:F0-1}
\end{eqnarray}
To simplify our notation we do not write in Equations~\eqref{eq:F0-infty}--\eqref{eq:F0-1} the $t$-estimate from
Equation~\eqref{eq:F0-t-estimate}, because it does not contribute to the final result, however we have to keep in
mind the domain on the $\lambda$-plane where it is valid.

We impose one more condition on the functions $y$ and $z$,
\begin{equation}\label{eq:cond-z-y-2}
t^{-2}z^2(y-\frac1y)=o\left(1/t^q\right)
\end{equation}
for some $q\in(0,1)$.
Condition~\eqref{eq:cond-z-y-2} does not neither follow nor contradict conditions~\eqref{zdem} and \eqref{ydem}:
it means just a special relation between asymptotic values of the functions $y$ and $z$. A posteriori
we know that $q=1/2$, however, at this stage we do not fix it.

Then, in the domain $t(\lambda-p)=O(t^\epsilon)$, where $\epsilon>1-q$ and $p=0,1$, the following estimate takes place
\begin{eqnarray}\label{eq:def-F1}
F_1(\lambda)\equiv-t^{-2}\frac12\left(\{A_1A_0\}_{11}-
\{A_1A_0\}_{22}\right)\int
(2l(l+A_{11}t^{-1})\la^2(\la-1)^2)^{-1}
d\la\approx\\
\approx -t^{-2} z^2(y-\frac 1y) \int
(2l(l+A_{11}t^{-1})\la^2(\la-1)^2)^{-1}
d\la=o(1).\label{eq:F1-estimate}
\end{eqnarray}
The definition of $F_1$ contains the same ambiguity as the one for $F_0$, so that the comment
to Estimate~\eqref{eq:F1-estimate} is analogous to the one after Equation~\eqref{eq:def-F0}.

We use the following representation of WKB-formula (\cite{AK})
\begin{eqnarray}\label{eq:def-WKB}
&\Psi_q(\la)=
T(\la)\exp\left\{\left(F_0(\la)+F_1(\la)-F_0(\la^*_q)-
F_1(\la^*_q)\right)\si_3\right\},\quad
\la^*_q\in\Ga_q,&\\
&T(\lambda)=\dfrac{\imath}{\sqrt{2l(l+A_{11}t^{-1})}}
\left(l\si_3+\frac At\right),&\nonumber
\end{eqnarray}
where $\Gamma_q$ is the so-called Stokes domain (see, e.g. \cite{AK}) and
$\la^*_q$ is an arbitrary fixed point from $\Gamma_q$. The paths of integration
in Equation~\eqref{eq:def-WKB} with $F_p(\lambda)$, $p=0,1$, defined in Equations~\eqref{eq:def-F0} and
\eqref{eq:def-F1}, should be taken in $\Gamma_q$. Here, however, we do not consider in detail the definition
of the Stokes graph, because the turning points in our case coalesce with the singular points and for our
purposes we can formulate the result in a simpler way (see Theorem~\eqref{th:wkb}).

We fix $T$ in such a way that $T\to\imath\si_3$
as $\la\to\infty$, $\arg \la = 0$.
Due to \eqref{eq:cond-z-y-2}, the term with $F_1$ is of
order $o(1)$ and can be ignored.

Let us write the asymptotic expansions we need
\begin{eqnarray*}
l(\la)&=&\frac12
+\mathcal{O}(\frac1\la)\mbox{ as }\la\to \infty,\\
l(\la)&=&\frac12+\mathcal{O}((\la t)^{-1})+\mathcal{O}(t^{-1})
\mbox{ as }\la\to 0,\\
l(\la)&=&\frac12+\mathcal{O}\left(((\la-1) t)^{-1}\right)+\mathcal{O}\left(t^{-1}\right)
\mbox{ as }\la\to 1.
\end{eqnarray*}
In the region $(\la-p)t=\mathcal{O}\left(t^\ep\right)$, $\ep>0$,
we have
\begin{eqnarray}\label{eq:WKB-T-0}
T&\sim&
\bm 1 & 0\\
\frac1{yt}(z+\frac{\Theta_0+\Theta_1+\Theta_\infty}2)
& -1\em + o(1) \mbox{ as }\la \to 0, \\
T&\sim&
y^{\frac12\si_3}\bm 1 & \frac{-z-\Theta_0}{yt}\\
0 & -1\em y^{-\frac12\si_3}
+o(1)
\mbox{ as }\la \to 1,\label{eq:WKB-T-1}
\end{eqnarray}
where we impose one more assumption,
\begin{equation}\label{eq:cond-z-y-3}
\frac{zy}t=o(1).
\end{equation}
For the large $\lambda$,
\begin{equation}\label{eq:WKB-T-infty}
T\sim
\si_3+o(1)
\mbox{ as }\la \to \infty.
\end{equation}
Hereafter, if two expressions are connected by the symbol ``$\sim$'', then they are equal up to a scalar nonzero multiplier.

Instead of defining the Stokes domains, we formulate the following theorem.
\begin{theorem}\label{th:wkb}
Assume that coefficients of  Equation~\eqref{mainwu} satisfy the following conditions:
\eqref{zdem}, \eqref{ydem}, \eqref{as1}, \eqref{eq:cond-z-y-2}, and \eqref{eq:cond-z-y-3}. Then, for any $j\in \Bbb Z$
there exists a solution $\Psi_j(\la)$ of Equation~\eqref{mainwu} with the following asymptotic expansion at large pure
imaginary $t$,
\begin{equation}\label{eq:WKBa}
\Psi_j(\la)=T(\la)
\exp\left\{\left(F_0(\la)-F_0(\la^*_j)\right)\sigma_3\right\}
\end{equation}
on the ray
$\arg(\la t)=\frac\pi2+\pi(j-2)$,
$|\la t|\geq |t|^{\ep_m}$,
$\ep_m=1-q$ where $q\in(0,1)$.
Point $\la^*_j$ lies on the same ray.
\end{theorem}
\subsection{A model equation for solutions near the singular points}
\label{subsec:Der1modeleq}
As usual, the WKB-asymptotic fails near the singular
points. In the neighborhood of these points,
we need another approximation.
To find it, we introduce the model functions
$Y_i(k,\si;x)$. Slightly modified, these functions
can be found in \cite{J}.
They satisfy the following linear differential
equation
\begin{equation}
\cd_xY_i(k,\si;x)=\left(\frac12\si_3 +
\frac 1{2x}
\bm -k &\si-k\\
\si+k &k\em\right)Y_i(k,\si;x).
\label{st}
\end{equation}
Asymptotic expansions of these solutions in the sectors
$$
-\frac \pi 2 +\pi (i-2) <\arg x<\frac {3\pi}2+\pi(i-
2)
$$
are as follows
$$
Y_i(k,\si;x)=\left(1+\mathcal{O}\left(x^{-1}\right)\right)\exp\left\{\frac x2-\frac k2 \ln x\right\}\si_3
$$
They define the functions $Y_i(k,\si;x)$ uniquely.
Following the standard method (see, for instance,
\cite{AK}), we define the monodromy parameters for
Equation~(\ref{st}).

The solutions $Y_i(k,\sigma;x)$ of Equation~(\ref{st})
are connected by Stokes matrices $S_i(k,\si)$
\begin{equation}
Y_{i+1}(k,\si;x)=
Y_i(k,\si;x) S_i(k,\si).
\label{stst}
\end{equation}
All the Stokes matrices can be found via $S_1(k,\si)$
and $S_2(k,\si)$ using the relation
$$
S_{i+2}(k,\si)=e^{\pii k\si_3}
S_{i}  (k,\si)e^{-\pii k\si_3}.
$$
For $S_1$ and $S_2$ we have
$$
S_1(k,\si)=\bm 1&0\\s_1(k,\si)&1\em,\quad
S_2(k,\si)=\bm 1&s_2(k,\si)\\0&1\em.
$$
The Stokes
multipliers $s_i(k,\si)$, $i=1,2$, are given by explicit formulae:
$$
s_1(k,\si) =
-\frac{2\pii}
{\Ga(1-\frac{\si-k}2)\Ga(\frac{\si+k}2)},\quad
s_2(k,\si) =-\frac{2\pii
e^{\pii k}}{\Ga(1-\frac{\si+k}2)\Ga(\frac{\si-k}2)}.
$$
We choose function $Y_2(k,\si;x)$ as the ``main'' basic
function.
The explicit construction for this function
is given by M.~Jimbo
\begin{eqnarray*}
&Y_2(k,\si;x) =&\\
&=\left(\begin{array}{cc}
e^{\pii (1-k)/2}
W_{(1-k)/2,\si/2}( e^{-\pii }x)&
-\frac 12 (\si-k)
W_{(-1+k)/2,\si/2}(x)\\
-\frac 12 (\si+k)
e^{ \pii (1-k)/2}
W_{-(1+k)/2,\si/2}(e^{-\pii }x)&
W_{(1+k)/2,\si/2}(x)\end{array}\right) x^{-\frac 12}.&
\end{eqnarray*}
Here
$W_{\,\cdot\,,\,\cdot\,}(\,\cdot\,)$ is the Whitteker function.
The asymptotic expansion as $x\to 0$ is given as follows
\begin{equation}
Y_2(k,\si;x)=
G_{k\si} (1+\mathcal{O}(x)) x^{\frac 12 \si \si_3} C_{k\si},
\label{const}
\end{equation}
where
\begin{eqnarray*}
G_{k\si} &=&\frac 1{\sqrt \si}
\bm\frac 12 (\si-k) &-1\\ \frac 12 (\si+k) &1\em,\\
C_{k\si} &=&\sqrt\si\bm
\frac{- \Ga(-\si)} {\Ga(1-\frac{ \si-k}2)}
e^{-\pii (\si+k)/2} &
\frac{- \Ga(-\si)} {\Ga(1-\frac{ \si+k}2)}\\
\frac{- \Ga(\si)} {\Ga(\frac{ \si+k}2)}
e^{\pii (\si-k)/2} &
\frac{ \Ga(\si)} {\Ga(\frac{ \si-k}2)}
\em.
\end{eqnarray*}
We introduce the following monodromy matrix
$$
M(k,\si)=C_{k\si}^{-1}e^{\pii\si\si_3}C_{k\si}
$$
and find
\begin{equation}
M(k,\si)=
\bm
e^{-\pii k}&
\frac{2\pii}
{\Ga(1-\frac{\si+k}2)\Ga(\frac{\si-k}2)}\\
\frac{2\pii e^{-\pii k}}
{\Ga(1-\frac{\si-k}2)\Ga(\frac{\si+k}2)}
&
2\cos\pi\si-e^{-\pii k}\em.
\label{monst}
\end{equation}
Also, we will need another monodromy matrix related to $Y_3(k,\si;x)$:
$$
M_3(k,\si)=
\bm
-e^{\pii k}+2\cos \pi\si&
m^3_{12}\\
\frac{2\pii e^{-\pii k}}
{\Ga(1-\frac{\si-k}2)\Ga(\frac{\si+k}2)}
&
e^{\pii k}\em.
$$
The entry $m^3_{12}$, which we do not use in the following text, can be computed with the help of
the relation $\det M_3(k,\si)=1$.
\subsection{Singular point $\la=1$}
 \label{subsec:Der1approx1}
In the neighborhood of the point $\la=1$,
we can rewrite Equation~(\ref{mainwu}) as
$$
\cd_\la Y(\la)=\bigg(\frac t2\si_3+
A_0+\frac{A_1}{\la -1}
+\mathcal{O}(\la-1)A_0\bigg)Y(\la).
$$

Divide this equation by $t$
and introduce
the variable $x=(\la-1)t$.
Then, function $Y^{(1)}(x)=Y(\la)$
satisfies the equation
$$
\cd_x Y^{(1)}(x)=\bigg(\frac 12\si_3+
\frac{A_0}t+\frac{A_1}x
+\mathcal{O}(x) \frac{A_0}{t^2}\bigg)Y^{(1)}(x).
$$
We introduce function
$Y^{(2)}=y^{-\frac12\si_3}Y^{(1)}$,
which satisfies the equation
$$
\cd_x Y^{(2)}(x)=\bigg(
\bm \frac 12 & \beta_1\\
o(1) & -\frac 12 \em
+\frac{\ti A_1}x
+\mathcal{O}(x)\frac{\ti A_0}{t^2}\bigg)
Y^{(2)}(x),
$$
where we used Assumption~\eqref{eq:cond-z-y-3}
and denoted $\beta_1=-\frac{ z+\Theta_0}{yt}$.

To obtain the basic equation (\ref{st}),
we transform the first matrix in the right-hand side of the previous
equation to the diagonal form substituting $Y^{(2)}=
T_1 Y^{(3)}=\bm 1 & - \beta_1 \\ 0 & 1\em Y^{(3)}$.
Then, $Y^{(3)}$ satisfies the following equation
$$
\cd _x Y^{(3)}(x)\approx \bigg(\frac 12\si_3
+\frac{\hat A_1}x
\bigg)
Y^{(3)}(x).
$$
Here,
$$
\hat A_1 =T_1^{-1}y^{-\frac 12\si_3}
A_1 y^{\frac 12\si_3}T_1.
$$
Since
\begin{equation}
\det \hat A_1=
\det A_1=-\frac{\Th_1^2}4,
\label{deteq}
\end{equation}
we see that parameter $\si$ in the model equation
should be taken equal to $\Th_1$. Then from the equation on $Y^{(3)}$,
namely from the diagonal
element in $\hat A_1$,
we can find parameter $k$ in the model equation:
\begin{align}
k\equiv k_1&=2\left(z+\dfrac{\Theta_0+\Theta_\infty}2+
\beta_1\left(\dfrac{\Theta_0+\Theta_1+\Theta_\infty}2\right)\right)=-2\vp_1+\tin,\nonumber\\
\varphi_1&=-z-\dfrac{\Theta_0}2+\dfrac{z+\Theta_0}{ty}\left(z+\dfrac{\Theta_0+\Theta_1+\Theta_\infty}2\right).
\label{eq:def-varphi1}
\end{align}
So, we know that $\sigma$ and $k$ in the model equation are $\Theta_1$ and $k_1=-2\vp_1+\tin$ correspondingly,
but we still do not have model equation, because $\{\hat A_1\}_{21} \neq \frac{\sigma+k}2$.
We transform the equation by a diagonal matrix: $Y^{(3)}=\rho_1^{\frac 12\si_3}Y^{(4)}$
and find that $Y^{(4)}$ satisfies the following equation:
$$
\cd _x Y^{(4)}(x)\approx \bigg(\frac 12\si_3
+\frac 1{2x}
\bm -k_1 & \Theta_1 - k_1\\
\Theta_1 +k_1 & k_1\em
\bigg)
Y^{(4)}(x),
$$
that is the model equation (\ref{st}). Parameter $\rho_1$ should be taken as follows:
\begin{equation}\label{eq:rho1}
\rho_1= -1 -\beta_1 = -1 +\frac{z+\Theta_0}{yt}.
\end{equation}
We also assume that $\varphi_1=\mathcal{O}(1)$, which is consistent with Theorem~\ref{th:new1}.

Now, the standard proof \cite{WW} allows us to formulate the following Theorem.
\begin{theorem}\label{exist1}
For any $j\in \Bbb Z$,
there exist a solution
$Y_{(j)}(1,\la)$ of
Equation~$(\ref{mainwu})$
with the following asymptotic
expansion as $|(\la-1)t|<|o(t)|$
\begin{equation}\label{eq:Y-near-1-main}
Y_{(j)}(1,\la)\approx
y^{\frac 12\si_3}
T_1 \rho_1^{\frac 12 \si_3}Y_j(\Th_1,k_1;(\la-1) t),
\end{equation}
\end{theorem}
\begin{corollary}\label{cor:large1}
In the region, $(\la-1)t=\mathcal{O}\left(t^{\ep_m}\right)$,
$0<\ep_m<1$, the asymptotic
expansion of function
$Y_{(j)}(1,\la)$ has the following form
\begin{equation}
Y_{(j)}(1,\la)=
y^{\frac 12\si_3}
T_1 \rho_1^{\frac 12 \si_3}
\left(1+\mathcal{O}\left(t^{-\ep_m}\right)\right)
e^{-\frac{k_1}2\ln(\la-1)\si_3}
t^{-\frac{k_1}2\si_3}
e^{\frac{(\la-1)t}2\si_3}
\label{big1}
\end{equation}
in the sector
$$
-\frac \pi 2 +\pi (j-2) <
\arg (\la-1)+
\arg t
<\frac {3\pi}2+\pi(j-2).
$$
\end{corollary}
\begin{corollary}\label{cor:can1}
In the region, $(\la-1)t=o(1)$,
the asymptotic
expansion of function
$Y_{(2)}(1,\la)$ is given by
\begin{equation}\label{eq:Y_i(1,lambda)-x-large}
Y_{(2)}(1,\la)=
y^{\frac 12\si_3}
T_1 \rho_1^{\frac 12 \si_3}
G_{k_1\Th_1}
\left(1+\mathcal O\big(\la-1)t\big)\right)\big((\la-1)t\big)^{\frac{\Th_1}2\si_3}
C_{k_1\Th_1}.
\end{equation}
\end{corollary}
In these statements
$$
t^{-\frac {k_1}2\si_3} \eq e^{-\frac{k_1}2\si_3\ln
t},\quad
\ln t = \ln|t| + \imath \arg t.
$$
Hereafter, we will understand all multivalued logarithmic
functions of $t$ in this way.
\subsection{Singular point $\la=0$}
  \label{subsec:Der1approx0}
The construction presented in this subsection is analogous to the one in Subsection~\ref{subsec:Der1approx1}.
In particular here we do not arrive at new restrictions on functions $y$ and $z$.

In the neighborhood of the point $\la=0$ we write down Equation~(\ref{mainwu}) as follows,
$$
\cd_\la Y(\la)=\left(\frac t2\si_3-A_1+\frac{A_0}\la+\mathcal{O}(\la)A_1\right)Y(\la).
$$
We divide this equation by $t$ and introduce the variable $x=\la t$.
Then, function $Y^{(1)}(x)=Y(\la)$ satisfies the equation
$$
\cd_x Y^{(1)}(x)=\bigg(
\bm \frac 12 & o(1)\\
\beta_0 & -\frac 12\em
+\frac{A_0}x
+\mathcal{O}(x)\frac{A_1}{t^2}\bigg)Y^{(1)}(x),
$$
where, $\beta_0=\frac{z+(\Theta_0+\Theta_1+\Theta_\infty)/2}{ty}$.

To obtain the model equation (\ref{st}),
we make the substitution
$$
Y^{(1)}= T_0Y^{(2)}=\bm 1 & 0 \\ \beta_0 & 1\em Y^{(2)}.
$$
Then, $Y^{(2)}$ satisfies the equation
$$
\cd_x Y^{(2)}(x)\approx \bigg(\frac 12\si_3
+\frac{\hat A_0}x
\bigg)
Y^{(2)}(x).
$$
Here,
$$
\hat A_0 =T_0^{-1} A_0 T_0.
$$
Note that
$$
\det \hat A_0=-\frac{\Th_0^2}4
$$
and we can find the parameters $\si$ and $k$ in the model
equation: $\si=\Th_0$ and
\begin{equation*}\label{eq:def-k0}
k\equiv k_0=-2 z-\Theta_0 +2 \beta_0 (z+\Theta_0)=2\vp_1,
\end{equation*}
with $\varphi_1$ defined in Equation~\eqref{eq:def-varphi1}.
Now we map equation on $Y^{(2)}$ to the model Equation~\eqref{st} by making
the following transformation
$Y^{(2)}=\rho_0^{\frac 12\si_3} Y^{(3)}$,
where
\begin{equation}\label{eq:rho0}
\dfrac1\rho_0=\beta_0-1=\frac{z+(\Theta_0+\Theta_1+\Theta_\infty)/2}{ty}-1.
\end{equation}
Then, we find that $Y^{(3)}$ satisfies the following equation:
$$
\cd _x Y^{(3)}(x)\approx \bigg(\frac 12\si_3
+\frac 1{2x}
\bm -k_0 & \Theta_0 - k_0\\
\Theta_0 +k_0 & k_0\em
\bigg)
Y^{(3)}(x),
$$
that is exactly the model equation (\ref{st}).

As in Subsection~\ref{subsec:Der1approx1} we arrive at the following results:
\begin{theorem}\label{exist0}
For any $i\in \Bbb Z$,
there exist a solution
$Y_{(i)}(0,\la)$ of Equation~$(\ref{mainwu})$
with the following asymptotic
expansion as $\la t=o(t)$
\begin{equation}\label{eq:Y-near-0-main}
Y_{(i)}(0,\la)=
T_0\rho_0^{\frac 12 \si_3} Y_i(\Th_0,k_0;\la t).
\end{equation}
\end{theorem}
\begin{corollary}\label{cor:large0}
In the region, $\la t=\mathcal{O}\left(t^{\ep_m}\right)$,
$0<\ep_m<1$, the asymptotic
expansion of function
$Y_{(i)}(0,\la)$ is given by
\begin{equation}\label{eq:Y_i(0,lambda)-x-large}
Y_{(i)}(0,\la)=
T_0\rho_0^{\frac 12 \si_3}
\left(1+\mathcal{O}\left(t^{-\ep_m}\right)\right)
e^{-\frac{k_0}2\ln\la\si_3}
t^{-\frac{k_0}2\si_3}
e^{\frac{\la t}2\si_3}.
\end{equation}
This asymptotic expansion is valid in the sector
$$
-\frac \pi 2 +\pi (i-2) <
\arg \la+
\arg t
<\frac {3\pi}2+\pi(i-2).
$$
\end{corollary}
\begin{corollary}\label{cor:can0}
In the region, $\la t=o(1)$,
the asymptotic
expansion of function
$Y_{(2)}(0,\la)$ is
\begin{equation}
Y_{(2)}(0,\la)=
T_0\rho_0^{\frac 12 \si_3}
G_{k_0\Th_0}
\big(1+\mathcal O(\la t)\big)(\la t)^{\frac{\Th_0}2\si_3}
C_{k_0\Th_0}.
\label{small0}
\end{equation}
\end{corollary}
\subsection{Matching}
 \label{subsec:Der1matching}
We have defined the following solutions of Equation~(\ref{mainwu}):
the canonical solutions,$Y_i(\la)$; the solutions with the WKB asymptotics, $\Psi_i(\la)$;
and the solutions in the proper neighbohoods of the singular points, $Y_{(i)}(p,\la)$, $i\in \Bbb Z$.
Since all of them are solutions of the same equation, the following matrices
\begin{equation}\label{eq:def-Li-Ci}
L_i(p)=Y^{-1}_{(i)}(p,\la)\Psi_i(\la),\qquad
C_i=\Psi^{-1}_i(\la)Y_i(\la)
\end{equation}
are independent of $\la$.
Due to Corollaries \ref{cor:can1} and \ref{cor:can0},
\begin{equation}\label{eq:def-E_{(i)}(p)}
Y_{(i)}(p,\la)\underset{\la\to p}
=G_i(p)(\la-p)^{\frac{\Th_p}2\si_3}
E_{(i)}(p),
\end{equation}
where the matrices $E_{(i)}(p)$ are also independent of $\la$.

Our immediate goal is to find asymptotics of these matrices as $t\to\infty$.
After that, we find asymptotics of the matrices $E^p_i$ (see Section~\ref{sec:def-monodromy}),
via the following relation:
\begin{equation}\label{eq:epi}
E^p_i=E_{(i)}(p)L_i(p)C_iu^{-\frac 12\si_3},
\end{equation}
where we restored function $u$ (see introductory part for this Section).

For any integer $i$ (note that the imaginary unit is denoted as $\imath$)
\begin{equation}\label{eq:ci}
C_i=-\imath \si_3(1+o(1)),
\end{equation}
where the error estimate is a diagonal matrix. This estimation is obtained by taking
asymptotics as $\lambda\to\infty$ along the corresponding (to $i$) Stokes line
in the second equation~\eqref{eq:def-Li-Ci}. The error estimate is a diagonal matrix,
because the off-diagonal entries would be $\lambda$--dependent.

To find asymptotics of matrices $L_i(p)$ we use their definition~\eqref{eq:def-Li-Ci},
where we take asymptotics as $x=(\lambda -p)t=\mathcal O\left(t^\ep\right)\to\infty$,
$0<\ep<1$, along Stokes lines of the functions $Y_i(p,\lambda)$. Again we have to take
the diagonal part of the asymptotics because the non-trivial off-diagonal part would be
$x$-dependent.

Using expansions~\eqref{eq:Y_i(1,lambda)-x-large} and \eqref{eq:WKBa}, with $F_0(\lambda)$
and $T(\lambda)$ given by \eqref{eq:F0-1} and \eqref{eq:WKB-T-1}, respectively, we want to
match these solutions in the leading term. In order for the leading term
the leading terms of the matrices $L_i(1)$ not to depend on $\lambda$,
the following condition should be valid
\begin{equation}\label{eq:phi-phi1-phi2}
\varphi_2\equiv\varphi-\varphi_1=
\dfrac{yz}t\left(z+\dfrac{\Theta_0-\Theta_1+\Theta_\infty}2\right)-
\dfrac2t\left(z+\dfrac{\Theta_0}2\right)\left(z+\dfrac{\Theta_0+\Theta_\infty}2\right)=o(1).
\end{equation}
Assuming the above condition is true, one finds,
\begin{equation}\label{eq:Li1}
L_i(1)= t^{\frac {k_1}2\si_3} (\rho_1 y)^{-\frac12 \si_3}
e^ {\frac t2\si_3}\sigma_3(1+o(1)),
\end{equation}
where, again, the error term $o(1)$ is a diagonal matrix.
We find asymptotics of $L_i(0)$ in the similar way, with the help of equations
\eqref{eq:Y_i(0,lambda)-x-large}, \eqref{eq:WKBa}, \eqref{eq:F0-1}, and \eqref{eq:WKB-T-1}.
The result reads
\begin{equation}\label{eq:Li0}
L_i(0)=t^{\frac {k_0}2\si_3} \rho_0^{-\frac12 \si_3}
e^ {\pi \imath (\vp-\frac{\tin}2)\si_3}\sigma_3(1+o(1)).
\end{equation}

Now we have enough information to calculate all the monodromy data introduced in Section~\ref{sec:def-monodromy}.
We, however, are going to find a minimal set of the data which completely characterize the domain of the monodromy
manifold corresponding to our assumptions on the coefficients of Equation~\eqref{mainl} which are made in the preceding
Subsections.

To simplify further relations, we need some preliminary notation. Denote $\Cal G_p$ the
upper (p=0) -- and lower (p=1) -- triangular subgroups of $SL(2,\mathbb C)$ with the unit diagonal,
$$
\Cal G_p=\{g\in SL(2,\Bbb C)\colon
g=\bm 1 & (1-p)*\\
p*&1\em\}.
$$
where the symbol $*$ stands for arbitrary complex number.
Subgroup $\Cal G_p$ acts on $SL(2,\Bbb C)$ via the right multiplication.
We use the same notation $\Cal G_p$ for this action, since it cannot cause any misunderstanding.
Note that element $X_{12}$ is not changed by transformation
$\Cal G_1$, while element $X_{21}$ is not changed by $\Cal G_0$.

Let $\Cal D$ be a diagonal subgroup of
$SL(2,\Bbb C)$. Then, for any $g\in\Cal G_p$
and $d\in \Cal D$
\begin{equation}
dg=gd \ \op{mod} \Cal G_p.
\label{dggd}
\end{equation}

By employing this notation, the matrices $E_{(i)}(p)$ (see Equation~\eqref{eq:def-E_{(i)}(p)})
due to Equations~\eqref{eq:Y-near-1-main}, \eqref{eq:Y-near-0-main}, \eqref{stst}, and \eqref{const},
can be written as follows:
\begin{eqnarray}
E_{(2)}(p)&=&C_{k_p\Th_p},\label{eipp2}\\
E_{(3)}(0)&=&C_{k_0\Th_0} \ \op{mod} \Cal G_0,
\label{eipp0}\\
E_{(1)}(1)&=&C_{k_1\Th_1} \ \op{mod} \Cal G_1.
\label{eipp1}
\end{eqnarray}

Relations \eqref{eq:ci}, \eqref{eq:Li1}, \eqref{eq:Li0},
\eqref{eipp2}, \eqref{eipp0}, and \eqref{eipp1} give us information sufficient
to find asymptotics of matrices $E^p$ (see \eqref{y22}).

First, consider the case $\arg t=\frac \pi2$.
Then, in the neighborhood of point $\la =1$, we have:
$\arg x=\arg (\la-1)t=\frac \pi 2$ for
real $\la>1$, $\arg \la=0$. Note that
asymptotic expansion of
$Y_2$ is fixed at $\arg \la =0$. Thus,
$$
 E^1\equiv E^1_2=C_{k_1\Th_1}
L_2(1)C_2
u^{-\frac 12\si_3}.
$$
In the neighborhood of point $\la =0$, we have:
$\arg x=\arg \la t=\frac {3\pi} 2$ for
real $\la<0$, $\arg \la =\pi$.
Note that: 1) solution $Y_3$ is fixed at
$\arg\la =\pi$;
2) $Y_2=Y_3 \  \op{mod} \Cal G_0$;
3) matrices $L_i(p)$ are diagonal ones,
$L_i(p)\in \Cal D$;
4) relations
(\ref{eipp0}) and (\ref{dggd}) take place. So,
$$
E^0\equiv E^0_2=E_3^0 \ \op{mod} \Cal G_0=C_{k_0\Th_0}
L_3(0)C_3
u^{-\frac 12\si_3}\ \op{mod} \Cal G_0.
$$

Now, consider the case $\arg t=-\frac \pi2$.
Then, in the neighborhood of point $\la =1$, we
have:
$\arg x=\arg (\la-1)t=-\frac \pi 2$ for
real $\la>1$, $\arg \la=0$. Note that
asymptotic expansion of
$Y_1$ is fixed at $\arg \la =0$.
Repeating the arguments above, we have
$$
E^1\equiv E^1_2=C_{k_1\Th_1}L_2(1)C_1
u^{-\frac 12\si_3}\ \op{mod} \Cal G_1.
$$
In the neighborhood of point $\la =0$, we have:
$\arg x=\arg \la t=\frac \pi 2$ for real $\la<0$, $\arg \la =\pi$.
Note that solution $Y_2$ is fixed at $\arg\la =\pi$.
Thus,
$$
E^0\equiv E^0_2=C_{k_0\Th_0}L_2(0)C_2u^{-\frac 12\si_3}.
$$

Combining all these facts together, we find
\begin{align}
\arg t=\frac \pi2:\qquad
&E^1=
C_{k_1\Th_1}
t^{\frac {k_1}2\si_3} (\rho_1 y)^{-\frac12 \si_3}
\exp\!{\Big(\frac t2\si_3\Big)}u^{-\frac 12\si_3} (1+o(1)),\label{e1+}\\
&E^0=
C_{k_0\Th_0}
t^{\frac {k_0}2\si_3} \rho_0^{-\frac12 \si_3}
\exp\!\Big(\pi \imath\Big(\vp-\frac{\tin}2\Big)\si_3\Big)u^{-\frac 12\si_3}(1+o(1))
\ \op{mod} \Cal G_0;
\label{e0+}\\
\arg t=-\frac \pi2:\qquad
&E^1=
C_{k_1\Th_1}
t^{\frac {k_1}2\si_3} (\rho_1 y)^{-\frac12 \si_3}
\exp\!{\Big(\frac t2\si_3\Big)}u^{-\frac 12\si_3} (1+o(1))
\ \op{mod} \Cal G_1,
\label{e1-}\\
&E^0=
C_{k_0\Th_0}
t^{\frac {k_0}2\si_3} \rho_0^{-\frac12 \si_3}
\exp\!\Big(\pi \imath\Big(\vp-\frac{\tin}2\Big)\si_3\Big)u^{-\frac 12\si_3}
(1+o(1)).
\label{e0-}
\end{align}
Note that in Equations~\eqref{e1+}--\eqref{e0-} the error estimates, $o(1)$, are diagonal matrices.

The manifold of monodromy data $\Cal M_5(\Th_0,\Th_1,\tin)$ is three-dimensional.
Thus, we are to obtain three parameters in monodromy data. When $\arg t= \frac \pi2$, we can find two parameters from
relation (\ref{e1+}) and one parameter from (\ref{e0+}). When $\arg t=-\frac \pi2$, we can find one parameter from
relation (\ref{e1-}) and two parameters from (\ref{e0-}).
In both cases, one of the matrices $E^p$ has the following structure: $E^p=C_{k_p\Th_p}d$, where $d$
is a diagonal matrix, $d\in \Cal D$.
So, the diagonal elements of the corresponding monodromy matrix $M^p$ are equal to those from (\ref{monst}).
Using this remark, we find for $\arg t= \frac \pi2$:
\begin{eqnarray}
m^1_{11}&=&
e^{2\pii\vp -\pii \tin}(1+o(1)),
\nonumber\\
m^1_{12}&=&
\dfrac {2\pi \imath t^{2\vp-\tin} e^{-t} \rho_1 y u(1+o(1))}
{\Ga(1-\frac{\Th_1+\tin -2\vp}2)
\Ga(\frac{\Th_1-\tin+2\vp}2)},
\label{eqs:mont+}\\
m^0_{21}&=&
\dfrac{2\pi \imath t^{2\vp} e^{-\pii\tin}(1+o(1))}
{u\rho_0\Ga(1-\frac{\Th_0-2\vp}2)
\Ga(\frac{\Th_0+2\vp}2)},
\nonumber
\end{eqnarray}
For $\arg t= -\frac \pi2$,
we have
\begin{eqnarray}
m^0_{11}&=&
e^{-2\pii \vp}(1+o(1)),\nonumber\\
m^1_{12}&=&
\frac {2\pi \imath t^{2\vp-\tin} e^{-t} \rho_1 y u(1+o(1))}
{\Ga(1-\frac{\Th_1+\tin -2\vp}2)
\Ga(\frac{\Th_1-\tin+2\vp}2)}
\label{eqs:mont-},\\
m^0_{21}&=&
\frac{2\pi \imath t^{2\vp} e^{-\pii\tin}(1+o(1))}
{u\rho_0\Ga(1-\frac{\Th_0-2\vp}2)
\Ga(\frac{\Th_0+2\vp}2)}.
\nonumber
\end{eqnarray}
Since the monodromy data~\eqref{eqs:mont+} and \eqref{eqs:mont-} do not depend on $t$, we arrive at
the following asymptotic expansions:
\begin{align}\label{eqs:delta-hatu-def}
y e^{-t} t^{4\vp-\tin}\frac{\rho_1}{\rho_0}=\de(1+o(1)),\qquad
-yue^{-t}t^{2\vp-\tin}\rho_1=\hat u(1+o(1)),\quad
\text{where}\quad\delta,\hat u\in\mathbb C\setminus0
\end{align}
are parameters (independent of $t$). At this point we recall our notational agreement (see the preamble
to Section~\ref{sec:der1}) and consider $\varphi$ as a complex parameter, rather than a function of $t$
with the behavior $\mathcal O(1)$ as $t\to\imath\infty$.

Now, substituting the first two conditions~\eqref{eqs:delta-hatu-def} into
Equations~\eqref{eqs:mont+} and \eqref{eqs:mont-}, and taking into account that the matrix elements $m^p_{ik}$
are independent of $t$, so that we can take the limit $t\to\imath\infty$, we arrive at the results announced in
Theorem~\ref{th:allmon}. Let us note that the expressions for $m^1_{21}$ and $m^0_{21}$ (as functions of $\vp$,
$\de$, and $\hat u$) remain the same regardless of sign of $\Im(t)$. Only $m^1_{11}$ and $m^0_{11}$ differ.
\subsection{Asymptotics of System~\eqref{eq:ids1}--\eqref{eq:ids3}}\label{subsec:asympt-th3.1}
To get the results announced in Theorem~\ref{th:new1}, we rewrite the first two Equations~\eqref{eqs:delta-hatu-def}
and the one for $\varphi$, \eqref{eq:phi-phi1-phi2}, in terms of $\rho_0, \rho_1$, and
$\varphi_1$, see Equations~\eqref{eq:rho0}, \eqref{eq:rho1}, and \eqref{eq:def-varphi1}, respectively:
\begin{align}
 \label{eq:ratio-y-z}
\dfrac{yt}{z+\Theta_0}&=1-\dfrac{\delta}{\varphi-\frac{\Theta_0}2}\,e^tt^{\Theta_\infty-4\varphi+1}(1+o(1)),\\
u&=\dfrac{\hat u(1+o(1))}{yt-\big(z+\Theta_0\big)}\,e^tt^{\Theta_\infty-2\varphi+1},\label{eq:hatu-y-z}\\
\varphi&=-z-\dfrac{\Theta_0}2+\dfrac{z+\Theta_0}{ty}\left(z+\dfrac{\Theta_0+\Theta_1+\Theta_\infty}2\right)+o(1),
\label{eq:varphi-y-z}
\end{align}
Substituting the ratio in the l.-h.s. of Equation~\eqref{eq:ratio-y-z} into Equation~\eqref{eq:varphi-y-z} we obtain
a linear algebraic equation for $z$. Solving it we successively obtain $yt$ and $u$ from Equations~\eqref{eq:ratio-y-z}
and ~\eqref{eq:hatu-y-z}, respectively:
\begin{align}\label{eq:yt-Der1}
yt&=\delta(1+o(1))t^{\nu_1}e^t-2\vp +\frac{\0 +\1 +\tin}2
+\dfrac{(\vp - \frac{\0}2)(\vp - \frac{\1 +\tin}2)}{\delta(1+o(1))t^{\nu_1}e^t},&\\
z&=-\vp-\frac{\0}2+\dfrac{(\vp - \frac{\0}2)(\vp - \frac{\1 +\tin}2)}{\delta(1+o(1))t^{\nu_1}e^t},&\label{eq:z-Der1}\\
u&=\dfrac{\hat u}\delta\cdot
\dfrac{t^{2\vp}(1+o(1))}{1-\Big(\varphi-\frac{\Theta_1+\Theta_\infty}2\Big)\delta^{-1}(1+o(1))t^{-\nu_1}e^{-t}},&
\label{eq:u-Der1}
\end{align}
where we denoted
\begin{equation}\label{eq:nu1}
\nu_1=\Theta_\infty-4\varphi+1.
\end{equation}
Thus we get explicit expressions for $y$, $z$, and $u$. The same expressions, presented however in a multiplicative
form, are given in Theorem~\ref{th:new1} as the leading terms of the asymptotic expansions. According to our
justification scheme outlined in Introduction to announce these formulae as asymptotics of the true solutions of
System~\eqref{eq:ids1}--\eqref{eq:ids3} we have to check that all our assumptions and error estimates made in this
section are valid. The latter, in fact, leads to some restrictions on our asymptotic parameters. To find them we
note that as $t\to\imath\infty$:
$$
|yt|=\mathcal O\left(|t|^{|\Re(\nu_1)|}\right),\qquad
|z|=\left\{
\begin{aligned}
\mathcal O(1)\qquad\Re(\nu_1)\geq0,\\
\mathcal O\left(|t|^{-\Re(\nu_1)}\right)\qquad\Re(\nu_1)\leq0
\end{aligned}
\right.
$$
Most of the calculations done in this section are valid for $0<\Re(\nu_1)<2$, as indicated in the preamble to this
Section. However, the matching requires, see Equation~\eqref{eq:phi-phi1-phi2}, the following restriction on $\nu_1$:
\begin{equation}\label{ineq:nu1}
-\dfrac12<\Re(\nu_1)<1.
\end{equation}
The left inequality in \eqref{ineq:nu1} follows from the fact that for negative $\Re(\nu_1)$ function $z$ is
growing . This growth is bounded by the second term in Equation~\eqref{eq:phi-phi1-phi2}.
Let us explain the right inequality in \eqref{ineq:nu1}. For positive $\Re(\nu_1)$ Equation~\eqref{eq:phi-phi1-phi2}
implies $\varphi_2=\mathcal O\big(t^{\nu_1-2}\big)+\mathcal O\big(t^{-1}\big)$. We demand that the second term in
asymptotics of $z$ (see Equation~\eqref{eq:z-Der1}), which has the order $t^{-\nu_1}$, should be greater than
$\varphi_2$. Otherwise the asymptotics of $z$ would consists of only one constant term which contains only one
parameter and therefore gives only very rough approximation for this function, in particular, such asymptotics does not
uniquely characterize function $z$. At the same time the condition $\Re(\nu_1)<1$ does not improve radically our
asymptotics for function $y$, because this asymptotics is mainly defined (see Equation~\eqref{eq:ratio-y-z}) by
multiplication of the constant term in asymptotics of $z+\Theta_0$ with the growing power term
$\mathcal O\big(t^{\nu_1}\big)$. Thus, one can continue to use asymptotics of $y$ announced in Theorem~\ref{th:new1}
in the region $1\leq\Re(\nu_1)<2$. Our numerical studies (see Section~\ref{sec:numerics}) confirms this observation.

Now we are ready to discuss the accuracy of approximation of function $y$ by the asymptotics given in
Theorem~\ref{th:new1}. This question is intimately related with the error estimate we introduced in
\eqref{eq:ratio-y-z} for function $\tilde\delta\equiv\delta(1+o(1))$. This estimate allows us
to confirm only the largest term in asymptotics of $y$ in case $\Re(\nu_1)\neq0$. However, as we see below,
we can assert that our result gives us up to three correct terms in asymptotics of $y$.

Note that the estimate (in $\tilde\delta$) defines the class of functions for which we calculate the monodromy data.
Our calculations are valid for any estimate $o(1)$ in $\tilde\delta$. Therefore, to formulate the best possible
asymptotic result for the functions $y$ and $z$ that follows from our derivation we have to demand that this $o(1)$
is as small as possible. It cannot be equal to $0$ because in this case Equation~\eqref{eq:ratio-y-z} would provide us
with the first integral for general solutions of the system~\eqref{eq:ids1}--\eqref{eq:ids2} which is not possible.
At the same time there is no sense to demand that the error estimate for $\tilde\delta$ is better than the one for
function $\varphi$, since function $\tilde\delta$ appeared for the first time in Equation~\eqref{eq:ratio-y-z} in
the ratio $\tilde\delta/\varphi$. Therefore, the order of the error estimate in Theorem~\ref{th:new1}, which follows
from our derivation coincides with the order of the estimate related with the transition from function $\varphi$
to the parameter $\varphi$ (see preamble Section~\ref{sec:der1}). This means that practically we can omit
$o(1)$-term in Equation~\eqref{eq:ratio-y-z}. Very similar reasoning leads to the conclusion that the error estimate
for function $u$ in Equations~\eqref{eq:hatu-y-z} and \eqref{eq:u-Der1} can be chosen coinciding with the one
for $\varphi$. The latter means that in all Equations~\eqref{eq:yt-Der1}--\eqref{eq:u-Der1} we can omit $o(1)$
terms and take into account the error that comes out from function $\varphi$. The best possible error for
transition from function $\varphi$ to the parameter $\varphi$ which comes from our calculation coincides with
$\mathcal O(\varphi_2)$. The following analysis is based on this fact.

First consider positive values of $\Re(\nu_1)$. The error estimate for function $\varphi_2$ in this case is
$\mathcal O\big(1/t\big)$, which gives the following error estimate for function $yt$ in
Theorem~\ref{th:new1},
\begin{equation}\label{eq:error-thnew1}
\mathcal O\Big(t^{\nu_1+\mathcal O(1/t)}-t^{\nu_1}+t^{\nu_1}\mathcal O\left(1/t\right)\Big)=
\mathcal O\left(t^{\nu_1-1}\ln t\right).
\end{equation}
Thus, because of the inequality $\Re(\nu_1)-1<-\Re(\nu_1)$ which holds for $0\leq\Re(\nu_1)<1/2$ we see that all
three terms of asymptotics of $y$ (Equation~\eqref{eq:yt-Der1} without $o(1)$ terms) are larger than the error
estimate~\eqref{eq:error-thnew1}. In the case $1/2<\Re(\nu_1)<1$ only the two first terms of asymptotics of $y$
are larger than the error estimate; we note that in the whole interval $0\leq\Re(\nu_1)<1$ two terms of asymptotics
are larger than the error estimate for both functions $y$ and $z$ (see Equations~\eqref{eq:yt-Der1} and
\eqref{eq:z-Der1} with the omitted $o(1)$ terms).

The conclusion made in the above paragraph is consistent with the complete asymptotic expansions for functions
$y$ and $z$ developed in Appendix~\ref{sec:allterms}, namely, one can improve approximation of function $yt$
by adding up the following correction terms:
\begin{align*}
&y_{10}\delta e^tt^{\nu_1-1}
&\text{for}\quad
\dfrac12\leq\Re(\nu_1)<\dfrac23\qquad&
\text{and one more term}\\
&y_{11}\left(\delta e^tt^{\nu_1-1}\right)^2
&\text{for}\quad
\dfrac23\leq\Re(\nu_1)<1,\qquad&
\end{align*}
where $y_{10}$ and $y_{11}$ are defined in Appendix~\ref{sec:allterms}.

It is mentioned above that one can continue to use asymptotics of $y$ given in Theorem~\ref{th:new1} in the
region $1<\Re(\nu_1)<2$. It is worth noting that for the latter values of $\nu_1$ the error estimate for function
$y$ is $\mathcal O(t^{2\nu_1-2}\ln t)$, i.e. is growing, and for $z$ it is still vanishing, $\mathcal O(t^{\nu_1-2})$.
It is easy to observe that Theorem~\ref{th:new2} deliver much better approximation of $y$ and, surely, $z$
for these values of $\nu_1$. At the boundary value $\Re(\nu_1)=1$, the leading (growing) term of asymptotics of
function $y$ and the leading (constant) term of asymptotics of $z$ given by both
Theorems ~\ref{th:new1} and~\ref{th:new2} coincide. Either result can be used for approximation of these functions:
the accuracy (which one is better?) depends on the particular solution. One has to use the correction terms given
in Appendix~\ref{sec:allterms}, especially for approximation of function $z$, to achieve a "reasonably" good
asymptotic description of these functions. The reader will find a numeric example in
Subsection~\ref{subsec:numeric-extra-terms}.

Consider now negative values of $\Re(\nu_1)$ for general solutions:
$$
-1/2<\Re(\nu_1)<0,\qquad\varphi-\Theta_0/2,\;\varphi-(\Theta_1+\Theta_\infty)/2\neq0.
$$
In this case $\varphi_2=\mathcal O\big(t^{-2\nu_1-1}\big)$, therefore, the error estimate in both formulae
for $yt$ and $z$ in Theorem~\ref{th:new1} is of the order $t^{-3\nu_1-1}\ln t$. Both function $yt$ and $z$ have
exactly the same leading term of asymptotics (here we again refer to Equations~\eqref{eq:yt-Der1} and
\eqref{eq:z-Der1}) proportional to $t^{-\nu_1}$. Therefore, in the region
$-1/4<\Re(\nu_1)\leq0$ all three terms of asymptotics for function $yt$ and two terms for function $z$ are larger
than the error estimate. In the case $-1/3<\Re(\nu_1)\leq-1/4$ two terms of asymptotics
for functions $yt$ and $z$ are larger than the error estimate. Finally, for $-1/2<\Re(\nu_1)\leq-1/3$ only the
largest terms of asymptotics for $yt$ and $z$ are larger than the corresponding error estimates.
\section{Derivation II}\label{sec:derII}
In this section we outline some basic steps which lead to Theorem~\ref{th:new2}. The scheme of the proof and major
steps of calculations are the same as in the previous section. Therefore here we outline only modifications that are needed for the
case under consideration.

The major assumptions on the coefficients of Equation~\eqref{mainl} are as follows:
\begin{equation}\label{eqs:Der2z-y-conditions}
|z|<\mathcal O(t),\qquad
\mathcal O\left(t^{-1}\right) <1/|y| <\mathcal O\left(t\right).
\end{equation}
In fact, these conditions are equivalent to those given by Equations~\eqref{zdem} and \eqref{ydem}
in Section~\ref{sec:der1}. In this section we continue to use the conventions about symbols $\mathcal{O}(\cdot)$
and $o(\cdot)$ made in the paragraph below Equations~\eqref{zdem} and \eqref{ydem}.

Therefore, we do not need to change anything in the WKB-method, except for asymptotics of matrix $T(\lambda)$:
\begin{eqnarray*}
T&\sim&
\bm 1 & \frac yt(-z+\frac{-\Theta_0+\Theta_1-\Theta_\infty}2) \\
0& -1\em + o(1) \mbox{ as }\la \to 0, \\
T&\sim&
y^{\frac12\si_3}\bm 1 & 0\\
\frac{yz}{t} & -1\em y^{-\frac12\si_3}
+o(1)
\mbox{ as }\la \to 1,\\
T&\sim&
\si_3+o(1)
\mbox{ as }\la \to \infty.
\end{eqnarray*}
The reason for this change is the following assumption on the functions $z$ and $y$,
$$
\frac{z}{yt}=o(1),
$$
which we use now instead of Assumption~\eqref{eq:cond-z-y-3}.
For solution in the neighborhood of $\la =1$, we have
the same result as above (see Theorem \ref{exist1} and Corollaries \ref{cor:large1},
and \ref{cor:can1}), but now:
\begin{equation}\label{eq:Der2rho1}
T_1 = \bm 1 & 0 \\ \beta_1 & 1\em,\quad
\beta_1=\frac{zy}t,\quad\rm{and}\quad
\rho_1 = 1/(\beta_1-1).
\end{equation}
For solution in the neighborhood of the point $\la =0$, we have
the same result as above (see Theorem \ref{exist0} and Corollaries \ref{cor:large0},
\ref{cor:can0}), but now:
\begin{equation}\label{eq:Der2rho0}
T_0 = \bm 1 & -\beta_0 \\ 0& 1\em,\quad
\beta_0=\frac yt(-z+(\Theta_1-\Theta_0-\tin)/2),\quad\rm{and}\quad
\rho_0 = -1 -\beta_0.
\end{equation}

The matching goes exactly as before. In particular, for matrices $L_i(1)$ and $L_i(0)$ we get
exactly the same expressions \eqref{eq:Li1} and \eqref{eq:Li0}, respectively, but with $\rho_1$ and $\rho_0$
defined  in \eqref{eq:Der2rho1} and \eqref{eq:Der2rho0}.
Proceeding exactly as in Section~\ref{sec:der1}, we arrive at formulas (\ref{eqs:mont+}) and (\ref{eqs:mont-})
for the monodromy data.

Next, we introduce asymptotic parameters, $\delta$ and $\hat u$ by formulae~\eqref{eqs:delta-hatu-def},
with the corresponding parameters $\rho_1$ and $\rho_0$.

The asymptotic parameter $\vp$ is defined in Equation~\eqref{eq:varphi}. However, due to the
conditions~\eqref{eqs:Der2z-y-conditions}, Equation~\eqref{eq:phi-phi1-phi2} should be changed to
\begin{equation}\label{eq:Der2-phi-phi1-phi2}
\varphi_2\equiv\varphi-\varphi_1=
\dfrac{z+\Theta_0}{ty}\left(z+\dfrac{\Theta_0+\Theta_1+\Theta_\infty}2\right)-
\dfrac2t\left(z+\dfrac{\Theta_0}2\right)\left(z+\dfrac{\Theta_0+\Theta_\infty}2\right)=o(1).
\end{equation}
Now, using definitions for $\rho_1$ and $\rho_0$, \eqref{eq:Der2rho1} and \eqref{eq:Der2rho0}, and condition
\eqref{eq:Der2-phi-phi1-phi2}, we can write an analog of system~\eqref{eq:ratio-y-z}--\eqref{eq:varphi-y-z}:
\begin{align}\label{eq:Der2-delta}
\dfrac{e^{-t}t^{\nu_2}}{\delta(1+o(1))}&=\dfrac yt\left(z-\dfrac ty\right)
\left(z-\dfrac ty+\dfrac{\Theta_0-\Theta_1+\Theta_\infty}2\right),\\
-ue^{-t}t^{\nu_2-2\varphi}&=\hat u\left(z-\dfrac ty\right)(1+o(1)),\label{eq:Der2-hatu}\\
\vp &=-z -\frac{\Theta_0}2 +\frac {yz}t \left(z+\frac{\Theta_0-\Theta_1+\Theta_\infty}2\right)+\mathcal{O}\left(\dfrac{z^2}t\right)+
\mathcal{O}\left(\dfrac{z^2}{ty}\right),
\label{eq:Der2-phi-z-y}
\end{align}
where
$$
\nu_2=4\vp-\tin+1.
$$
Opening the parenthesis in Equations~\eqref{eq:Der2-delta} and  ~\eqref{eq:Der2-phi-z-y} and substituting the term $yz^2/t$
in Equation~\eqref{eq:Der2-phi-z-y} by its expression obtained from Equation~\eqref{eq:Der2-delta} one finds
\begin{equation}\label{eq:z-t/y}
z-\dfrac ty=\varphi+\dfrac{\Theta_1-\Theta_\infty}2-\dfrac{e^{-t}t^{\nu_2}}{\delta(1+o(1))}.
\end{equation}
Here we include the error estimate from Equation~\eqref{eq:Der2-phi-z-y} into the notation $\varphi$ as agreed in
preamble of Section~\ref{sec:der1}. Substituting $z-t/y$ given by Equation~\eqref{eq:z-t/y} into
Equation~\eqref{eq:Der2-hatu} we get the leading term of asymptotics for function $u$ in Theorem~\ref{th:new2}.
Making the same substitution into Equation~\eqref{eq:Der2-delta} we obtain
\begin{equation}\label{eq:DER2-t/y}
\dfrac ty=\dfrac{\delta(1+o(1))}{e^{-t}t^{\nu_2}}
\left(\dfrac{e^{-t}t^{\nu_2}}{\delta(1+o(1))}-\varphi-\dfrac{\Theta_1-\Theta_\infty}2\right)
\left(\dfrac{e^{-t}t^{\nu_2}}{\delta(1+o(1))}-\varphi-\dfrac{\Theta_0}2\right).
\end{equation}
Now factoring out from the parentheses in Equation~\eqref{eq:DER2-t/y} the term
$\tfrac{e^{-t}t^{\nu_2}}{\delta(1+o(1))}$
we get the leading term of asymptotics for $t/y$ presented in Theorem~\ref{th:new2}. Substituting the
latter asymptotics into Equation~\eqref{eq:z-t/y} we obtain asymptotics for $z$. Finally, asymptotics
for $u$ immeditely follows from Equations~\eqref{eq:Der2-hatu} and \eqref{eq:z-t/y}. Thus the analog of
System~\eqref{eq:yt-Der1}--\eqref{eq:u-Der1} reads:
\begin{align*}
\dfrac ty&=
\dfrac{e^{-t}t^{\nu_2}}{\delta(1+o(1))} -2 \vp -\frac{\0 +\1-\tin}2+
\left(\vp + \frac{\1-\tin}2\right) \left(\vp +\frac{\0}2\right)\delta(1+o(1))e^t t^{-\nu_2}\\
z&=-\vp - \frac{\0}2 +\left(\vp + \frac{\1-\tin}2\right)\left(\vp + \frac{\0}2\right)\delta(1+o(1))e^t t^{-\nu_2}\\
u&=-\hat u(1+o(1))t^{-\nu_2+2\vp} e^t \left(\vp + \frac{\1-\tin}2-\dfrac{t^{\nu_2}e^{-t}}{\delta(1+o(1))}\right)
\end{align*}
Reasoning similar to the one presented in Subsection~\ref{subsec:asympt-th3.1} shows that to find the
best error estimate that comes from our derivation we can put all $o(1)$-estimates in the above formulae
to be of the same order as $\varphi_2$ (see \eqref{eq:Der2-phi-phi1-phi2}). Therefore, we have to come back to the
error estimate hidden in $\varphi$. The analysis, which is very similar to the one at the end of
Section~\ref{sec:der1} for the parameter $\nu_1$, implies similar restriction for the parameter $\nu_2$,
$$
-\dfrac12<\Re(\nu_2)<1.
$$
We can make a comment analogous to the one at the end of Section~\ref{sec:der1}: since asymptotics of $y$ is growing
in the region $1\leq\Re(\nu_2)<2$ the leading term of asymptotics for $y$ is still satisfactory, although
the asymptotics for $y$ given in Theorem~\ref{th:new1} works better. The situation is worse for function
$z$ because only the constant term of asymptotics remains larger than the error estimate for $1\leq\Re(\nu_2)<2$.
So in the domain  $1<\Re(\nu_2)<2$, one has to use the result given in Theorem~\ref{th:new1}. In case,
$\Re(\nu_2)=1$ either Theorem can be used but to get a good approximation one has to employ the correction
terms (see Appendix~\ref{sec:allterms}).
\section{Comparison with the results by McCoy and Tang}\label{sec:mccoy}
\def\ttt{{\frac\Th2}}
In this section, we compare our results with the ones obtained in paper
\cite{MT2}. The authors of \cite{MT2}  considered the case $\Th_\infty =2n\in\mathbb Z_+$, $\Th_0 =
\Th_1 =\Th$. We discuss here only the principal case $\Theta_\infty=0$, since the case $n>0$ can be
treated as application of the B\"acklund transformations (see Theorem~\ref{dress}) to the principal one.
In particular, the monodromy matrices for $n>0$ coincide with those for $n=0$. McCoy and Tang
obtained the following asymptotic expansion as $t \to \infty$ and $\arg t =-\frac \pi2$
($\hat x= \imath \frac t4$):
\begin{gather}
A^0_{11} =- \frac \imath 4 \frac{d \si_0}{d\hat{x}} = z +\frac{\Th}2,
\nonumber\\
z = -\imath k - \frac{\Th}2,\quad \si_0(\hat{x}) = 4 k\hat{x} +O(1)
\label{eq:mccoy2}\\
y = \frac{2k + \imath \Th}{2k-\imath \Th} e^{-4 \imath s}, \quad s =\hat{x}+\tilde{x}_0 + k \ln\hat{x},
\label{eq:mccoy3}
\end{gather}
In paper \cite{MT2} two parameters $k$ and $A=4k$ are used. In \cite{MT1, MT2} parameter $\delta=-8$,
whereas we fix $\delta=-1/2$ (see Equation~\eqref{eq:P5}). We use notation $\hat x$, instead of $x$ in \cite{MT1, MT2},
because in Section~\ref{sec:numerics} we denote $x=\imath t$ for $\Im t<0$. There is an obvious relation $\hat x=x/4$,
which should be used in the comparison of our results with those obtained in \cite{MT1,MT2}.
In Equations~\eqref{eq:mccoy2} and \eqref{eq:mccoy3}
we simplify the notation by using only the parameter $k$. These results agree  with our asymptotic expansions providing
the asymptotic parameters are related as
\begin{equation}\label{eq:parameters-correspondence}
\imath k = \vp, \quad \de = \frac{2 \vp - \Th}{2\vp +\Th}
e^{-4 \imath\tilde{x}_0} \left(\frac \imath 4 \right) ^{-4 \vp},
\end{equation}
where $(\imath /4 )^{-4\varphi}\equiv e^{-2\pi \imath \varphi} 4^{4\varphi}$.

Now we turn to parametrization of the asymptotics by the monodromy data. In \cite{MT2} McCoy and Tang expressed
the monodromy data in terms of the parameters $I^p$, which in our notation are defined as follows,
\begin{equation}\label{eq:Ip-McCoy-original}
I^p = \frac{ m^p_{11}(Y) -e^{-\pi \imath \Th}}{m^p_{11}(Y) - e^{\pi \imath
\Th}}, \quad p =0,1,
\end{equation}
where $Y$ is a canonical solution of Equation~\eqref{mainl} and $m^p_{11}(Y)$, is the $(11)$-element of the monodromy matrix,
$M^p(Y)$, corresponding to the singular point $\lambda=p$.
While analyzing their paper \cite{MT3} devoted to the
connection formulae for asymptotics of solutions on the real axis, we observed in \cite{AK} that the monodromy
parameter $I^0$ was calculated for solution $Y=Y_3$ (in our notation), while the parameter $I^1$ was given for $Y=Y_1$.
The same remark concerns the imaginary case considered in \cite{MT2}. By using the formulae presented in
Section~\ref{sec:def-monodromy} one finds that in terms of our monodromy data the parameters $I^p$ are given by
the following expression:
\begin{equation}\label{eq:Ip-McCoy-adjusted}
I^p = \frac{m^{1-p}_{11}(Y) - e^{\pi \imath \Th}}{m^{1-p}_{11}(Y) - e^{-\pi\imath \Th}},
\qquad
p=0,1.
\end{equation}
The parameters $I^p$ as calculated by McCoy and Tang (see Equations (2.53) and (2.68) in \cite{MT2}) are:
\begin{align}
I^1 &= e^{\pi \imath \Th} \frac{\sin \pi (\imath k +\frac{\Th}2)}{\sin\pi(\imath k -\frac{\Th}2)}
\label{I1MT2}\\
I^0 &= \frac{\bar C_- +4^{-\imath k}e^{-2 \pi k}e^{\frac{\pii\Th}2}\frac{\Ga(\imath k +
\frac{\Th}2)}{\Ga(-\imath k +\frac{\Th}2)}}{\bar C_- -4^{-\imath k}e^{-2 \pi k}e^{-\frac{\pii\Th}2}
\frac{\Ga(1+\imath k - \frac{\Th}2)}{\Ga(1-\imath k -\frac{\Th}2)}},
\label{I0MT2}\\
\bar C_- &= - \frac{2 \pii e^{-4 \imath\tilde{x}_0} 4^{3 \imath k}e^{-\pi k}}
{\Ga(1+\imath k+ \frac \Th2)\Ga(\imath k - \frac{\Th}2)}.
\label{Cbar}
\end{align}

Substituting into Equation~\eqref{eq:Ip-McCoy-adjusted} (for $p=1$) $m^0_{11}$ given by Equation
\eqref{mr2} and taking into account the first equation \eqref{eq:parameters-correspondence}
we see the complete agreement of our results with equation~\eqref{I1MT2}.

The parameter $I^0$ is more complicated: By making use of Equation~\eqref{eq:M0M1} and the results
for the monodromy data presented in Theorem~\ref{th:allmon}, we rewrite our monodromy parameter
$m^1_{11}$ as follows:
\begin{align*}
m^1_{11}& = \frac{1- m^0_{21}m^1_{12}}{m^0_{11}}= e^{2\pii\vp}(1-X),\\
X &\eq m^0_{21} m^1_{12} = - \frac {4 \pi^2\de}{\Ga^2(1+\vp -\ttt)\Ga^2(\vp +\ttt)},
\end{align*}
Now, Equation~\eqref{eq:Ip-McCoy-adjusted} (for $p=0$) allows us to present the monodromy parameter $I^0$
in the following way:
\begin{equation}\label{eq:ourI0}
I^0 = \frac{X + e^{\pii \Th} e^{-2\pii \vp} -1}{X +e^{-\pii \Th}e^{-
2\pii \vp} -1}.
\end{equation}
In its turn, Equations~\eqref{I0MT2} and \eqref{Cbar} obtained for $I^0$ by McCoy and Tang ([MT]) can be rewritten
with the help of relations~\eqref{eq:parameters-correspondence} and the well-known identities for the Gamma-function
in terms of $X$:
$$
I^0[MT]= \frac{-X+e^{\pii \Th}e^{-2\pii \vp} -1}{-X +
e^{-\pii \Th}e^{-2\pii \vp} -1}.
$$
So, for $I^0$ we have an agreement only up to the sign of $X$ or, equivalently, the sign of the parameter $\de$ or
function $y$. If we are to keep the sign of $\delta$ unchanged, then to get $I^0=I^0[MT]$ we can alternatively
demand that either $I^0=1$, or $X=0$, which is equivalent
(we recall that $\delta\neq0$) to one of the following conditions:
\begin{align}
I^0=1:&\;\; \Theta=0,\pm1,\pm2,\ldots,
\label{eq:X=0}\\
X=0:&\;\;
\varphi+\frac\Theta2=0,-1,-2,\ldots,\quad\text{or}\quad
\varphi-\frac\Theta2=-1,-2,\ldots.
\label{eq:I0=1}
\end{align}
Thus, contrary to the case of real argument $t$, where our parametrization of the quantities $I^0$ and $I^1$, after
being associated to the canonical solutions $Y_3$ and $Y_1$, respectively, coincides with the parametrization obtained
by McCoy and Tang (see \cite{AK}), for pure imaginary $t$ these parameterizations coincide only up to the sign of $X$ in
$I^0$.



Finally, we comment on the connection formulae for the asymptotics. To get the connection formulae McCoy and Tang
employ asymptotics of the fifth Painlev\'e transcendent as $t\to0$ obtained by Jimbo in \cite{J}. The latter asymptotics
were parameterized by the monodromy data of the canonical solution $Y_2$ in our notation
(see Section~\ref{sec:def-monodromy}). Thus asymptotics as $t\to\imath\infty$ and $t\to\imath0$ in \cite{MT2} appear
to be parameterized by the monodromy data of $Y_1$, $Y_3$ and $Y_2$, respectively. Hence, the connection
formulae obtained by McCoy and Tang could be correct only in a special situation when all three canonical solutions
coincide, $Y_1=Y_2=Y_3$, or, in other words, the Stokes multipliers vanish, $s_1=s_2=0$. Since $\Theta_\infty=0$ it means that
the monodromy matrix $M_\infty=I$ and the corresponding monodromy group of Equation~\eqref{mainl} is commutative.
For their connection formulae on the pure imaginary axis to be correct one should additionally demand one of
the conditions~\eqref{eq:X=0} or \eqref{eq:I0=1}.

In Subsection~\ref{subsec:numeric-2-MT} we consider a numerical solution of
IDS~\eqref{eq:ids1}-\eqref{eq:ids3} corresponding to nontrivial Stokes multipliers, $s_1$, $s_2$, and observe
a good agreement with our connection results, while the connection formulae by McCoy and Tang
do not show the correct asymptotic behavior.
\section{Asymptotic expansions for $t\to 0$}\label{sec:zero}
In this section $\arg t$ is fixed in the standard way, in particular, $\arg t=0$ for $t>0$. Moreover, $\arg t$ is assumed
to be bounded as $t\to0$. Let $\si$ be a complex number. It will be convenient to use the following notations:
\begin{equation}\label{eqs:abcd-definition}
b(\si)=\frac{\Th_1+\Th_0+\si}2,\
c(\si)=\frac{\Th_1-\Th_0+\si}2,\
d(\si)=\frac{\tin+\si}2,\
a(\si)=b(\si)c(\si).
\end{equation}
\begin{theorem}\label{th1}
Let $\si, s^2,r \in\Bbb C\back\{0\}$ and $\Re \si \in[0,1)$.
Let also $\0,\1\not \in \Bbb Z$.
Then there exists the unique solution of System~\eqref{eq:ids1}--\eqref{eq:ids2} with the following
asymptotic expansion as $t\to 0$:
\begin{eqnarray}
y&=&\frac
{(a(\si)+s^2 d(-\si)a(-\si)t^\si)(b(\si)+s^2d( \si)b(-\si)t^\si)}
{(a(\si)+s^2 d( \si)a(-\si)t^\si)(b(\si)+s^2d(-\si)b(-\si)t^\si)}+\nonumber\\
{}&+&
\frac{1-\1-\0}{(1-\si)^2}t+\mathcal O(t^{1+\Re\si})+\mathcal O(t^{2-\Re\si}),\label{eq:y->0}\\
z&=&\frac{
(c(\si)+s^2d( \si)c(-\si)t^\si)
(b(\si)+s^2d(-\si)b(-\si)t^\si)}{\si^2s^2t^\si}+\nonumber\\
{}&+&2\left(\frac{b(\si)c(\si)}
{\si^2(1-\si) s^2t^\si}\right)^2 t+\mathcal O(t^{1-\Re\si})+\mathcal O(t^{2-3\Re\si}),\label{eq:z->0}\\
u&=&-rt^{\tin}\left(\frac
{b(\si)+s^2d(-\si)b(-\si)t^\si}
{b(\si)+s^2d( \si)b(-\si)t^\si}\right)(1+\mathcal{O}(t)).\label{eq:u->0}
\end{eqnarray}
\end{theorem}
\begin{corollary}\label{cor:zeta-0}
Function $\ze(t)$, corresponding to the solution of System~\eqref{eq:ids1}--\eqref{eq:ids3}
defined in Theorem~{\rm\ref{th1}}, has the following asymptotic  expansion as $t\to 0$:
\begin{equation}\label{zeta->0}
\begin{aligned}
\ze(t)&=
\frac 14 (\si^2-\Th_1^2 +\Th_0^2+2\Th_0\tin)
-\frac{(\Th_1+\si)^2 -\Th_0^2 }
{4s^2\si^2 (1-\si)} t^{1-\si}
+\frac t{4\si^2} ( 2\Th_0\si^2 +\\
&+\tin (\Theta_0^2 -\Theta_1^2 +\si^2))
-\frac{s^2 (\Theta_\infty^2-\si^2)((\Th_1-\si)^2-\Theta_0^2)}
{16\si^2 (1+\si)}t^{1+\si}
+\mathcal{O}\left(|t|^{2-2\Re \si}\right).
\end{aligned}
\end{equation}
\end{corollary}
\begin{remark} If $0<\Re\si<1$, then the asymptotic expansion for function $y=y(t)$
can be rewritten as follows:
$$
y\underset{t\to0}=1+\frac{\si^2s^2b(-\si)}{a(\si)}t^\si+
\frac{1-\1-\0}{(1-\si)^2}t+\mathcal{O}(t^{2\Re\si})+\mathcal O(t^{2-\Re\si}).
$$
\end{remark}
\begin{remark}\label{rem:tau-u-gen-sigma}
{\rm
Small $t$-expansion of the $\tau$-function related to $\zeta$ as,
\begin{equation}\label{eq:tau}
\zeta(t)\equiv t\frac{d}{dt}\log\tau(t)+(\Theta_0+\Theta_\infty)t/2+((\Theta_0+\Theta_\infty)^2-\Theta_1^2)/4,
\end{equation}
has been obtained by Jimbo~{\rm\cite{J}}. We independently derived our results by a similar but
slightly different method and presented it in terms of the functions $y$, $z$, $u$, and $\zeta$ in
{\rm\cite{AK}}. The latter result, together with the asymptotics at infinity, allows one to find
the connection formulae for function $u$, that cannot be obtained from Jimbo's result. The
present form of small t-asymptotics (Theorem~{\rm\ref{th1}}) we announced in paper {\rm\cite{AK97ZNS}}, however
the error estimates in that paper were correctly written only for $0<\Re\sigma\leq1/2$. In this paper
we have added additional error estimates (see the last terms in Equations~\eqref{eq:y->0} and \eqref{eq:z->0})
which work for the interval $1/2<\Re\sigma<1$. Now the estimates cover the whole semi-open interval, $0\leq\Re\sigma<1$.
The origin of this mistake is not related with the isomonodromy deformation method which gives only the leading
terms of the asymptotics together with the error estimates in the form $\mathcal{O}\big(t^{1+\delta}\big)$, $\delta>0$,
without specification of the dependence of $\delta$ on $\sigma$. Formally, these estimates looks similar to the
first error estimates in Equations~\eqref{eq:y->0} and \eqref{eq:z->0}. The explicit form (in terms of $\sigma$) of these
"isomonodromy error estimates" are obtained via substitution of the corresponding asymptotic expansions into the isomonodromy
deformation system ~\eqref{eq:ids1} and \eqref{eq:ids2} where the terms important for the parameter $\sigma$ in the interval
$(1/2,1)$, the last estimates in Equations~\eqref{eq:y->0} and \eqref{eq:z->0}, were just overlooked although they have the
required form $\mathcal{O}\big(t^{1+\delta}\big)$.
}
\end{remark}
\begin{theorem}\label{th2}
Assume $\Theta_0,\Theta_1,b(\sigma),b(-\sigma),c(\sigma),c(-\sigma),d(\sigma),\mathrm{and}\;d(-\sigma)\notin\mathbb Z$.
Then the solution of System \eqref{eq:ids1}--\eqref{eq:ids3} defined in Theorem~{\rm\ref{th1}} generates an
isomonodromy deformation of Equation~\eqref{mainl} with the following monodromy data:
\begin{align*}
m^p_{11}&=\frac1{\hat d}\left(
\frac {e^{-\p 2\si}\hat m^p_{11}}\pi \sin \pi d(-\si) -
\frac {e^{\p 2\si} \hat m^p_{22}}\pi \sin \pi d( \si)-\right.\\&-
\left.\frac
{e^{-\p 2\si}\hat m^p_{21}}
{\hat s_2 \Ga (1+d(-\si))\Ga (1-d(\si))}-
\frac{\hat s_2 e^{\p 2\si}\hat m^p_{12}}
{\Ga (d(\si))\Ga (-d(-\si))}\right),\\
m^p_{22}&=\frac1{\hat d}\left(
-\frac {e^{ \p 2\si}\hat m^p_{11}}\pi \sin \pi d(\si)+
\frac {e^{-\p 2\si}\hat m^p_{22}}\pi \sin \pi d(-\si)+\right.\\&
\left.+\frac {e^{-\p 2\si}\hat m^p_{21}}{\hat s_2\Ga (1-d(\si))
\Ga (1+d(-\si))}+
\frac {\hat s_2 e^{\p 2\si}\hat m^p_{12}}
{\Ga (-d(-\si))\Ga (d(\si))}\right),
\end{align*}
\begin{align}
&m^p_{12} \!=\!\frac {re^{\p 2\tin}}{\hat d} \left(
\frac{\hat m^p_{22}-\hat m^p_{11}}
{\Ga(-d(-\si))\Ga (1-d(\si))}
\!-\!
\frac{\hat m^p_{21}}{\hat s_2\Ga^2 (1-d(\si))}
\!+\!\frac{\hat s_2\hat m^p_{12}}{\Ga^2(-d(-\si))}\right),\nonumber\\
&m^p_{21} \!=\!\frac {e^{-\p 2\tin}}{r\hat d}
\left(\frac{\hat m^p_{22}-\hat m^p_{11}}
{\Ga(d(\si))\Ga (1+d(-\si))}
\!+\!
\frac{\hat m^p_{21}e^{-\pi \imath \si}}
{ \hat s_2\Ga^2(1+d(-\si))}
\!-\!\frac{\hat s_2 \hat m^p_{12} e^{\pi \imath \si}}
{\Ga^2(d(\si))}\right),\nonumber\\
&s_2 =-
\frac {2\pi \imath r e^{\pii\tin}}
{\Ga(1-d(\si))\Ga(-d(-\si))},\qquad
s_1 =-
\frac {2\pi \imath r^{-1}}
{\Ga(1+d(-\si))\Ga(d(\si))}.\label{eqs:s1-s2at0}
\end{align}
In the previous formulas:
\begin{align*}
\hat d&= -\frac {e^{\p2 \tin}}\pi \sin \pi \si,\qquad
\hat m^p_{ij}=\{\hat M^p\}_{ij},\quad p=0,1,\qquad
\hat s_2 =s^2 \frac{\Ga(\si)}{\Ga(-\si)},\\
\hat M^0&=-\frac \imath {\sin \pi \si}
\bm -\cos \pi \1 +e^{\pii \si}  \cos \pi \0&
-e^{-\pii \si} \hat p\\ -e^{\pii \si} \hat q&
\cos \pi \1 -e^{-\pii \si}  \cos \pi \0\em,\\
\hat M^1&=-\frac \imath {\sin \pi \si}
\bm -\cos \pi \0+e^{\pii \si}  \cos \pi \1&
\hat p\\ \hat q&
\cos \pi \0 -e^{-\pii \si}  \cos \pi \1\em,\\
\hat p&=\frac{ \Ga(1+\si)}{\Ga(1-\si)}
\frac {2\pi ^2}{\Ga (1\!+\!b(\si)) \Ga (-b(-\si))
\Ga(1\!+\!c(\si)) \Ga (-c(-\si))},&\\
\hat q&=-\frac{ \Ga(1-\si)}{\Ga(1+\si)}
\frac {2\pi ^2 }{ \Ga (1\!+\!b(-\si)) \Ga (-b(\si))
\Ga(1\!+\!c(-\si)) \Ga (-c(\si))}.
\end{align*}
\end{theorem}
\begin{remark}{\rm
Multiplying expressions for the Stokes multipliers $s_2$ and $s_1$ in \eqref{eqs:s1-s2at0} one finds
\begin{equation}\label{eq:sigma-s1-s2}
2\cos\,\pi\sigma=2\cos\,\pi\Theta_\infty+s_1s_2e^{-\pi\Theta_\infty}.
\end{equation}
This equation means that to define parameter $\sigma$ for all monodromy data we have to allow $\Re\sigma\in[0,1]$.

The case $\Re\sigma\in(0,1)$ and $\Re\sigma=0$ with $\Im\sigma\neq0$, modulo some restrictions on the parameters
$\Theta_k$ where $k=0,1,t$, and $\infty$  is served by Theorems~{\rm\ref{th1}} and {\rm\ref{th2}}.

The case $\Re\sigma=1$ and the restrictions mentioned in the previous sentence are studied in part II of this work.

We finish this section by considering the case $\sigma=0$ with certain restrictions on the parameters $\Theta_0$ and $\Theta_1$.
}\end{remark}
\begin{theorem}\label{th4}
Let $\0,\1\not \in \Bbb Z$.
Let also $\hat s_1 \in \Bbb C$ and $r \in \Bbb C\back\{0\}$.
Then there exists a solution of system \eqref{eq:ids1}--\eqref{eq:ids2}
with the following asymptotic expansion as $t\to0$:
\begin{eqnarray*}
z&=&-dbc(\ln t+\hat s_1)^2+(bc+(b+c)d)(\ln t+\hat s_1)-d-b
+\ep,\\
u&=&-rt^{\tin}
\frac{db(\ln t+\hat s_1)-b-d+\ep}{db(\ln t+\hat s_1)-\frac 12+\ep},\qquad
\ep=\mathcal{O}\left(t\ln^4t\right),\\
y&=&\frac{(-dbc(\ln t+\hat s_1)+db+bc+cd+\ep)(b(\ln t+\hat s_1)-1+\ep)}
{(-bc(\ln t+\hat s_1)+b+c+\ep)(db(\ln t+\hat s_1)-b-d+\ep)}.
\end{eqnarray*}
Here $b=b(0)=\frac{\1+\0}2$, $c=c(0)=\frac{\1-\0}2$, and $d=d(0)=\frac{\tin}2$.
\end{theorem}
\begin{corollary}\label{cor:th5}
Function $\ze(t)$, corresponding to the solution of system \eqref{eq:ids1}--\eqref{eq:ids3}
defined in the Theorem~{\rm\ref{th4}}, has the following asymptotic  expansion as $t\to 0$:
\begin{align*}
\ze(t)=&-bc+(b-c)d+t(\ln t+\hat s_1)(dbc(\ln t+\hat s_1)-db-bc-cd-2bcd)+\\
&t(b+d+db+bc+cd+2dbc)+\mathcal{O}\left(t^2\ln^4t\right).
\end{align*}
\end{corollary}
\begin{theorem}\label{th6}
The solution of system~\eqref{eq:ids1}--\eqref{eq:ids3} described in Theorem~{\rm\ref{th4}}
defines isomonodromy deformation of Equation~\eqref{mainl} with the following monodromy data:
$$
s_1 =-\frac{2\pii }{r\Ga\left(1+\frac{\tin}2\right)\Ga\left(\frac{\tin}2\right)},
\qquad
s_2 =-\frac{2\pii re^{\pii \tin}}{\Ga\left(1-\frac{\tin}2\right)\Ga\left(-\frac{\tin}2\right)}.
$$
\begin{eqnarray*}
&m^1_{11}=\cos \pi \Th_1
+2\imath  e^{-\imath \hat d}
\left(
\left(\om\sin \hat d-e^{-\imath  \hat d}\right)w_1
-
\frac 12 e^{-\imath \hat d}\sin (\hat b+\hat c)
-\imath \cos \hat d\, \sin \hat b\, \sin \hat c\right),&\\
&m^1_{22}=\cos \pi \Th_1
-2\imath  e^{-\imath \hat d}
\left(
\left(\om\sin \hat d-e^{-\imath \hat d}\right)w_1
-
\frac 12 e^{-\imath \hat d}\sin (\hat b+\hat c)
-\imath \cos \hat d\, \sin \hat b\, \sin \hat c\right),&\\
&m^0_{11}=\cos \pi \Th_0
-2\imath e^{-\imath \hat d}
\left(\left(\om\sin \hat d
-e^{\imath \hat d}\right)w_1
-
\frac 12 e^{\imath \hat d} \sin (\hat b+\hat c)
+\sin \hat d\, \cos \hat b\, \cos \hat c\right),&\\
&m^0_{22}=\cos \pi \Th_0
+2\imath e^{-\imath \hat d}
\left(\left(\om\sin \hat d-
e^{\imath \hat d}\right)w_1
-
\frac 12 e^{\imath \hat d} \sin (\hat b+\hat c)
+\sin \hat d\, \cos \hat b\, \cos \hat c\right),&\\
&m^1_{12}=\frac{2\imath e^{-\imath \hat d}}{d_1^2}
\left(\om \sin \hat d-\cos \hat d\right)
\left(
w_1
\sin \hat d
-\cos \hat d \sin \hat b \sin \hat c\right),&\\
&m^0_{12}=
-\frac{2\imath e^{-\imath \hat d}}{d_1^2}
\left(\sin \hat d\left(\om \sin \hat d-e^{\imath \hat d}\right)w_1
-e^{\imath \hat d}\sin \hat c\sin \hat b
\left(\om \sin \hat d-e^{\imath \hat d}\right)
+\cos \hat c \cos \hat b \sin ^2\hat d\right),&\\
&m^0_{21}=
2\imath e^{-\imath \hat d}d_1^2
\left(
\om
w_1
+\cos \hat c\cos \hat b
\right),\qquad
m^1_{21}=-2\imath e^{-\imath \hat d}d_1^2
\left(\om +\imath\right)
\left(
w_1 +\imath \sin \hat b \sin \hat c\right).&
\end{eqnarray*}
Here
\begin{gather}
\hat b=\pi b(0)=\pi\frac{\Th_1+\Th_0}2,\;
\hat c=\pi c(0)=\pi \frac{\Th_1-\Th_0}2,\;
\hat d=\pi d(0)=\pi \frac{\tin}2,\;
d_1^2=-\frac{2\pi}{r\tin}\frac{e^{-\p2\tin}}{\Ga^2\left(\frac{\tin}2\right)},\label{eq:def-hat-abcd-d1}
\\
\om =\frac 1\pi \left(\hat s_1-\psi(b+1)-\psi(c+1) -\psi(d+1)+4\psi(1)\right),\quad
\psi(z)\equiv\frac{d}{dz}\ln\Gamma(z),\label{eq:def-omega-psi}\\
w_1=\om \sin \hat b\,\sin\hat c-\sin(\hat b+\hat c).\nonumber
\end{gather}
\end{theorem}
\begin{remark}\label{rem:lim-sigma0}
{\rm
The results stated in the last two theorems can be obtained by the repetition of the calculation scheme
outlined in our previous paper~{\rm\cite{AK}}, but with the asymptotics of the special functions involved there
corresponding to the case $\sigma=0$. Instead, we, following Jimbo (see {\rm\cite{J}}), consider the limit
$\sigma\to0$ by making the following substitution, $\hat s =1+\hat s_1 \si$ where $\hat s =-s^2(\tin+\si)/2$,
in the results stated in Theorems~{\rm\ref{th1}}, Corollary~{\rm\ref{cor:zeta-0}}, and {\rm\ref{th2}}, respectively.
Strictly speaking, to make the results obtained in Theorem~{\rm\ref{th4}} and Corollary~{\rm\ref{cor:th5}}
via this limiting procedure rigorous, we have to study the dependence of the error estimates in
Theorem~{\rm\ref{th1}} and Corollary~{\rm\ref{cor:zeta-0}} not only as $t\to0$ but also as $\sigma\to0$.
Concerning the latter estimates no details are given neither in Jimbo's paper, nor in our work~{\rm\cite{AK}}.
Our result is based on the conjecture that the $\sigma\to0$ limit of the $k$-th term in the small $t$-expansion of the
$\zeta$-function, $\zeta\underset{t\to0}=\sum_{k=0}^\infty t^k\sum_{|l|\leq k} a_{kl}t^{l\sigma}$ (the inner sum), can be
estimated as $\mathcal O\left(\ln^{2k}t\right)$. In principle, we do not need to prove this conjecture in case we
use the derivation of the results stated in Theorems~{\rm\ref{th4}}, Corollary~{\rm\ref{cor:th5}} via the
"first principles". The limiting procedure is simpler in the sense of derivation, but requires a proof of the
additional nontrivial result. The limiting procedure in the monodromy data of Theorem~{\rm\ref{th2}} is the
straightforward application of the l'Hopitale rule which is formulated as Theorem~{\rm\ref{th6}}.
}\end{remark}
\begin{remark}\label{rem:res-sigma0}
{\rm
Jimbo in {\rm\cite{J}}, presented small-$t$ asymptotics of $P_5$ only in terms of the $\tau$-function
(see Equation~\eqref{eq:tau}) our asymptotic formulae in terms of the functions $y(t)$, $z(t)$, $u(t)$, and $\zeta(t)$
are equivalent (with the same comment about $u(t)$ as in Remark~{\rm\ref{rem:tau-u-gen-sigma}}) to Jimbo's one. To see
this we have to make one more calculation, because Jimbo didn't write explicitly the monodromy matrices, like we do in
Theorems~{\rm\ref{th6}} and {\rm\ref{th2}}, instead he presented the analogs of our matrices $E^p$
(see Section~{\rm\ref{sec:def-monodromy}}) modulo left and right diagonal multipliers. So, below we give some details
which explain how one can get Jimbo's monodromy for $\sigma=0$.
}\end{remark}
To obtain this monodromy data (corresponding to the case $\si=0$), we use
formulas for $E^p$ from our previous paper (see \cite{AK}, Section 10, page 1834):
$$
E^p= \hat E^p s^{-\si_3}C_{\tin \si}R,
$$
where $R$ is a diagonal matrix independent of $\sigma$.
Let us note that $\lim_{\si\to0} E^p=0$ (up to scalar multiplier).
To apply the l'Hopitale rule, we need to compute the first derivative of $E^p$ with
respect to $\sigma$
$$
(E^p)'=({\hat E^p})' s^{-\si_3}C_{\tin 0}R
-\frac{s'}s \hat E^p s^{-\si_3}\si_3C_{\tin 0}R+
\hat E^p s^{-\si_3}C_{\tin \si}' R.
$$
Here and in the list of formulae below, the prime denotes the derivative with respect to
$\sigma$ taken at $\sigma=0$. Moreover, all objects that are functions of $\sigma$ are
assumed taken at $\sigma=0$, e.g.: $d=d(0)=\Theta_\infty/2$, $s^2=s^2(0)=-2/\Theta_\infty$
(see Remark~\ref{rem:lim-sigma0}), etc:
\begin{eqnarray*}
&\frac{s'}s =-\frac 1{4d}+\frac{\hat s_1}2,&\\
&\hat E^p=\diag\{*,*\}\bm 1&-1\\1&-1\em,\qquad
(\hat E^p)'=\diag\{l_{11}^p,l^p_{22}\} \si_3\hat E^p\si_3,&\\
&C'_{\tin\si}=\si_3C_{\tin 0}\si_3\diag\{-\p2 +\frac{\psi(d)}2,
-\frac{\psi(-d)}2\}-&\\
&-2\psi(1) \diag\{1,0\}C_{\tin 0}
+\frac1{2d}\si_3\diag\{1,0\}C_{\tin 0}\si_3,&\\
&C_{\tin 0}=\frac1{\sqrt d}d^{-\frac 12\si_3}\bm 1&-1\\
-1&1\em \diag\{\frac{e^{-\pii d}}{\Ga(d)},\frac 1{\Ga(-d)}\},&
\end{eqnarray*}
where function $\psi$ is defined in the second equation~\eqref{eq:def-omega-psi} and
\begin{align*}
&l^0_{11}=\p2 -\psi(1)+\frac{\psi(1+c)}2 +\frac{\psi(-b)}2,&
&l_{22}^0 =-\p2+2\psi(1)-\frac{\psi(1+b)}2 -\frac{\psi(-c)}2,\\
&l^1_{11}=-\psi(1)+\frac{\psi(-c)}2 +\frac{\psi(-b)}2,&
&l^1_{22}= \psi(1)-\frac{\psi(1+c)}2 -\frac{\psi(1+b)}2.
\end{align*}
Using the well-known identities for function $\psi(x)$ (see \cite{BE}):
$$
\psi(1+x)=\psi(x)+\frac1x, \qquad \psi(1+x)=\psi(-x)-\pi\cot\pi x,
$$
we arrive at the following expressions for the matrices $E^p$ at $\sigma=0$:
\begin{align*}
&E^0=\bm
-\om+\cot \hat b & \om -\imath-\cot \hat b-\cot \hat d\\
-\om +\cot\hat c & \om -\imath -\cot \hat c-\cot \hat d
\em
\diag\left\{\frac{e^{-\pii d}}{\Ga(d)},\frac{r}{\Ga(-d)}\right\},&\\
&E^1=\bm
-\om -\imath +\cot \hat b +\cot \hat c&
\om -\cot \hat b-\cot \hat c-\cot \hat d\\
-\om -\imath &
\om -\cot \hat d
\em
\diag\left\{\frac{e^{-\pii d}}{\Ga(d)},\frac{r}{\Ga(-d)}\right\}.&
\end{align*}
Now, writing $\cot(\cdot)$ as $\cos(\cdot)/\sin(\cdot)$ and putting the matrix elements of $E^0$ and $E^1$ to common
denominators, then getting rid of these denominators by factorizing $E^p$ $p=0,1$ with the help of left and right
diagonal matrices such that the numerators do not change, and, finally, omitting the left diagonal matrices
(as they do not effect the monodromy matrices) one arrives at the following expressions for the corresponding
matrix elements:
\begin{align*}
&\tilde E^0_{11}=d_1\left(s_1 \sin\frac\pi2 (\1+\0)-\pi \cos\frac\pi2 (\1+\0)\right),&
\\
&\tilde E^0_{12}=\frac 1{d_1}\left((s_1-\pii)\, \sin \frac \pi2 (\1+\0)\, \sin \frac\pi2
\tin -\sin\frac\pi2 (\1+\0+\tin)\right),&\\
&\tilde E^0_{21}=
d_1\left(s_1 \sin \frac \pi2 (\1-\0)-
\pi  \cos\frac\pi2 (\1-\0)\right),&\\
&\tilde E^0_{22}=\frac1{d_1}\left((s_1-\pii)\,\sin\frac\pi2(\1-\0)\,\sin\frac\pi2 \tin
-\sin\frac\pi2 (\1-\0+\tin)\right).&
\end{align*}
\begin{align*}
&\tilde E^1_{11}=d_1\left((s_1+\pii)\,\sin\frac\pi2(\1+\0)\,\sin\frac\pi2(\1-\0)-\pi\sin\pi\1\right),&\\
&\tilde E^1_{12}\!=\!\frac1{d_1}\left(\left(s_1\sin\frac\pi2\tin-\pi\cos \frac\pi2\tin\right)
\sin\frac\pi2(\1+\0)\,\sin\frac\pi2(\1-\0)-\pi\sin\pi\1\,\sin\frac\pi2\tin\right),&\\
&\tilde E^1_{21}=d_1\left(s_1+\pii\right),\qquad
\tilde E^1_{22}= \frac1{d_1}\left(s_1\sin\frac\pi2\tin-\pi\cos\frac\pi2\tin\right),&\\
\end{align*}
here matrices $\tilde E^p$ coincide with $E^p$ for $p=0,1$, respectively, modulo left diagonal factors.
Parameter $d_1$ is given in the last equation of \eqref{eq:def-hat-abcd-d1}, the sign of $d_1$ is not important,
since the monodromy matrices (see Theorem~\ref{th6}) depend on $d_1^2$. The formulae for $\tilde E^p$ exactly
coincide with the corresponding matrices obtained by Jimbo~\cite{J} modulo the scalar multiplier $d_1$, which is not
given in the paper \cite{J}.

\section{Degeneration in the general formulas for asymptotics as $t\to0$} \label{sec:deg0}

Here we show how short-$t$ asymptotics of some known particular solutions can be obtained with the help of the results
presented in Section~\ref{sec:zero}. We discuss in this section some limiting procedures when two or more parameters
simultaneously and consistently tends to some singular points of the formulae presented in the previous section.
The parameters which limits we consider parameterize both the monodromy data and the asymptotics of the corresponding
solutions. The formulae for the monodromy data are explicit so that the limiting procedures are chosen such that the
limiting set of the monodromy data exists. This set of the monodromy data define a solution of
IDS~\eqref{eq:ids1}--\eqref{eq:ids3}. This statement follows from the justification scheme for asymptotics obtained by the
isomonodromy deformation method \cite{K}, which can be applied for the large-$t$ asymptotics. Then these solutions can be
analytically continued into the neighborhood of $t=0$ and so we can discuss their asymptotics as $t\to0$. In our derivation
of the results in Section~\ref{sec:zero} we didn't study the dependence of the asymptotic estimates as the functions of the
parameters which we are going to send to some limiting values. Therefore, there appears a question concerning their behavior
in these limits. In fact, it is clear that these estimates cannot be unbounded because it would mean existence of a
singularity for all rather small $t$ which cannot be the case because of the Painlev\'e property. The only problem that
may occur in the limiting procedure with the small-$t$ asymptotics is that the error estimates may become equal or larger
than the leading term. In the last case, we cannot get a definite asymptotics directly
from our results. In this section we consider only the situations when the leading term of asymptotics after the limiting
procedure is larger than the error estimate. In following Section~\ref{sec:spec-meromorphic-solution} we consider a case when
in the limiting procedure the order of the error estimate in the limit coincide with the order of the leading term.

The parameters which are not involved in the limiting procedure are assumed to take general values, i.e., such values that
all functions and expressions with these parameters are properly defined at these values.

Note that asymptotic expansions of functions and monodromy data in the theorems of Section~\ref{sec:zero}
are not changed under the formal substitution
\begin{equation}
\si' = -\si,\,(s^2)' =\frac4{\Th_\infty^2-\si^2}\frac
1{s^2},\,r'=\frac{\tin-\si}{\tin+\si}r.
\label{tr1}
\end{equation}
Due to this invariance, we introduce parameters better suited for the change $\sigma \to -\sigma$:
$$
\hat s=-\frac{\tin+\si}2s^2,\hat r=\frac r{d(\si)}.
$$
In these notations, the substitution (\ref{tr1}) can be written as:
$$
\si'=-\si,\,\hat s'=\frac1{\hat s},\, \hat r'=\hat r.
$$

When we formulated our main monodromy results for $t\to 0$ we excluded the cases when $b(\pm \sigma)$,
$c(\pm \sigma)$, or $d(\pm \sigma)\in\mathbb Z$.
In this section we outline how one could overcome this difficulty and cover
the part of the manifold of monodromy data where these assumptions fail.

Our approach is to: a) explain the cases when functions $b$, $c$ or $d$ become zeroes and then b) use the results
on the Schlesinger transformations to reduce the cases of integer $b$, $c$ or $d$ to the case when one of these
linear combinations is zero.
Let us look at asymptotic expansions of functions $y(t)$, $z(t)$, and $u(t)$
in Theorem \ref{th1}. These expansions can become degenerate in the following cases:
\begin{enumerate}
\item $b(\si)=0$,
\item $c(\si)=0$,
\item $d(\si)=0$.
\end{enumerate}
Degeneration with $b(-\si)=0$, $c(-\si)=0$, or $d(-\si)=0$
can be reduced to the previous case, using transformation (\ref{tr1}).
Now we will specify what we mean by degeneration.  We define the three types of {\it complete degeneration}:
\begin{description}
\item[1c] $b(\si)\to0$, $s^2\to 0$, $b(\si)/s^2\to0$
\item[2c] $c(\si)\to0$, $s^2\to 0$, $c(\si)/s^2\to0$,
\item[3c] $d(\si)\to0$, $s^2\to \infty$, $d(\si)s^2\to0$
\end{description}
and the three types of {\it partial degeneration}:
\begin{description}
\item[1p] $b(\si)\to0$, $s^2\to 0$, $\lim b(\si)/s^2\neq0$,
\item[2p] $c(\si)\to0$, $s^2\to 0$, $\lim c(\si)/s^2\neq0$,
\item[3p] $d(\si)\to0$, $s^2\to \infty$, $\lim d(\si)s^2\neq0$.
\end{description}
In the partial degenerations we arrange the limiting procedure such that the finite limits of the rations
in the items above exist.
\begin{theorem}\label{th41}
The following monodromy data correspond to complete degeneration {\rm\bf{1c}} and {\rm\bf{2c}}:
\begin{eqnarray*}
&m^p_{11} =-\frac{e^{-\p2 \tin}}{\sin^2\pi \si}
\{e^{-\p2 \tin}(-\cos\pi\Th_{1-p}+\cos\pi \Th_p\cos \pi \si)
+&\\&+
e^{\p2 \tin}(- \cos \pi \Th_p +\cos \pi \Th_{1-p}\cos \pi \si)\},&\\
&m^p_{22} =-\frac{e^{-\p2 \tin}}{\sin^2\pi \si}
\{e^{-\p2 \tin}(\cos\pi\Th_{1-p}-\cos\pi \Th_p\cos \pi \si)
+&\\&+
e^{\p2 \tin}(- \cos \pi \Th_{1-p}\cos\pi\si
+\cos \pi \Th_p\cos 2\pi \si)\},&\\
&m^p_{12}=\frac{2\imath r}{\sin^2 \pi\si}
\frac{\cos \pi \Th_{1-p}-\cos \pi \Th_p\cos \pi \si}
{\Ga(-d(-\si))\Ga(1-d(\si))},
m^p_{21}=\frac{2\imath e^{-\pii \tin}}{r\sin^2 \pi\si}
\frac{\cos \pi \Th_{1-p}-\cos \pi \Th_p\cos \pi \si}
{\Ga(d(\si))\Ga(1+d(-\si))}&
\end{eqnarray*}
Parameter $\sigma$ is defined by equation $b(\sigma)=0$ or $c(\sigma)=0$ depending on degeneration scheme {\rm\bf{1c}} or
{\rm\bf{2c}}, respectively.
\end{theorem}
\begin{corollary}\label{th42}
Let $r \in \Bbb C\back\{0\}$.
There exists a solution of system {\rm (\ref{eq:ids1})--(\ref{eq:ids3})}
with the following asymptotics as $t\to0$:
\begin{align*}
&y=1+\mathcal{O}(t),\\
&z=-\frac{\Th_0(\tin+\0+\1)}{2(\0+\1)}+\mathcal{O}(t),\\
&u=-rt^{\tin}\frac{\tin-\si}{\tin+\si}(1+o(1))
\end{align*}
The monodromy data from Theorem \ref{th41} with $\si=-\Th_0-\Th_1$, $b(\si)=0$ degeneration {\rm\bf{1c}},
correspond to this solution.
\end{corollary}
\begin{corollary}\label{th43}
Let $r \in \Bbb C\back\{0\}$.
There exists a solution of system {\rm (\ref{eq:ids1})--(\ref{eq:ids3})}
with the following asymptotics as $t\to0$:
\begin{align*}
&y=\frac{\tin-\Th_0+\Th_1}{\tin+\Th_0-\Th_1}+\mathcal{O}(t),\\
&z=\frac{\Th_0(\tin+\Th_0-\Th_1)}{2(\Th_1-\Th_0)}+\mathcal{O}(t), \\
&u=-rt^{\tin}(1+o(1)).
\end{align*}
The monodromy data from (\ref{th41})
with $\si=\Th_0-\Th_1$, $c(\si)=0$ degeneration {\rm\bf{2c}}, correspond to this solution.
\end{corollary}
Before concidering degeneration {\bf3c} we find partial degeneration {\bf3p}.
\begin{corollary}\label{th44}
There exists a solution of system {\rm (\ref{eq:ids1})--(\ref{eq:ids2})}
with the following asymptotic expansion as $t\to0$:
\begin{eqnarray*}
&y=c(-\si)\frac{b(\si)+s_f b(-\si) t^\si}
{a(\si)+s_f a(-\si)t^\si}+\mathcal{O}(t),&\\
&z=\frac 1{\si^2}
d(-\si)b(-\si)c(\si)+ \frac{s_f t^\si}{\si^2}
c(-\si)d(-\si)b(-\si)+\mathcal{O}(t)=&\\
&=\frac1{4\tin}(\Th_1^2-(\Th_0+\tin)^2) +
s_f\frac{(\Th_1+\tin)^2-\Th_0^2}{4\tin}t^\si+\mathcal{O}(t),&\\
&u=-\hat rt^{\tin}\left(\frac{s_fd(-\si)b(-\si)t^\si}
{b(\si)+s_fb(-\si)t^\si}+\mathcal{O}(t)\right).&
\end{eqnarray*}
The following monodromy data corresponds to this solution
\begin{eqnarray*}
&m^p_{11} =-\frac{e^{-\p2 \tin}}{\sin^2\pi \si}
\{e^{-\p2 \tin}(-\cos\pi\Th_{1-p}+\cos\pi \Th_p\cos \pi \si)
+&\\
&+
e^{\p2 \tin}(- \cos \pi \Th_p +\cos \pi \Th_{1-p}\cos \pi \si)\}
-\frac 1{\hat d}\frac{\ti s_2e^{\p2\si}\hat m^p_{12}}{\Ga(-\tin)},&\\
&m^p_{22} =-\frac{e^{-\p2 \tin}}{\sin^2\pi \si}
\{e^{-\p2 \tin}(\cos\pi\Th_{1-p}-\cos\pi \Th_p\cos \pi \si)
+&\\
&+e^{\p2 \tin}(- \cos \pi \Th_{1-p}\cos\pi\si
+\cos \pi \Th_p\cos 2\pi \si)\}
+
\frac 1{\hat d}\frac{\ti s_2e^{\p2\si}\hat m^p_{12}}{\Ga(-\tin)}
,&\\
&m^p_{12}=
\hat r e^{\p2\tin}\ti s_2\hat m^p_{12}\hat d^{-1}\Ga^{-2}(-\tin),&\\
&m^p_{21} \!=\!\frac {e^{-\p 2\tin}}{\hat r\hat d}
\big(\frac{\hat m^p_{22}-\hat m^p_{11}}
{\Ga (1+\tin)}
+
\frac{\hat m^p_{21}e^{-\pi \imath \si}}
{ \ti s_2\Ga^2(1+\tin)}
\!-\!\ti s_2 \hat m^p_{12} e^{\pi \imath \si}\big).&
\end{eqnarray*}
Here $\ti s_2=s_f\frac{\Ga(\si)}{\Ga(-\si)}$ and $\hat d$ is the same as in Theorem \ref{th2}.
The Stokes multipliers are as follows:
$$
s_2=0,\qquad
s_1=\frac{2\pi s_f}{\hat{r}\Gamma(1+\Theta_\infty)}.
$$
\end{corollary}
The previous formulas were obtained in the following way.
Let us denote
$$
s^2d(\si)=s_f.
$$
Make limit transition as
$d(\si)\to0$
in formulas for functions $y$ and $z$.
Replace $r$ with the adjusted parameter $\hat r$ introduced in the beginning of the section.
Then, we obtain formulae for function $u$.
Make the same substitution in the monodromy data and perform the limit transition.
As a result, we arrive to the formulae above.
\begin{remark}
The partial degeneration $d(-\sigma)=0$ gives a similar result but with $s_1=0$ and $s_2\neq0$.
\end{remark}
\begin{remark}{\rm
Partial degenerations {\bf1p} and {\bf2p} is easy to obtain just by looking at formulae given in Theorems \ref{th1}
and \ref{th2}. Introduce parameters: $\hat{b}=b(\si)/s^2$ or $\hat{c}=c(\si)/s^2$, respectively, and make a limit $s^2\to0$.
We leave this analysis as a simple exercise for the readers.
}\end{remark}

Now to get complete degeneration case {\bf3c} we have to make the following limit in the results
presented in Corollary~\ref{th44}
$$
s_f\to0,\qquad
\hat{r}s_f\to\tilde{r}\in\mathbb{C}.
$$
\begin{corollary}\label{th45}
There exists a solution of system {\rm (\ref{eq:ids1})--(\ref{eq:ids2})}
with the following asymptotic expansion as $t\to0$:
\begin{align*}
&y=\frac{\Theta_1-\Theta_0+\Theta_\infty}{\Theta_1-\Theta_0-\Theta_\infty}+\mathcal{O}(t),\\
&z=\frac1{4\tin}(\Th_1^2-(\Th_0+\tin)^2)+\mathcal{O}(t),\\
&u=-\tilde{r}\Theta_\infty\left(\frac{\Theta_0+\Theta_1+\Theta_\infty}{\Theta_0+\Theta_1-\Theta_\infty}
+\mathcal{O}(t^{1+\Theta_\infty})\right).
\end{align*}
The following monodromy data corresponds to this solution
\begin{eqnarray*}
&m^p_{11} =-\frac{e^{-\p2 \tin}}{\sin^2\pi \Theta_\infty}
\{e^{-\p2 \tin}(-\cos\pi\Th_{1-p}+\cos\pi \Th_p\cos \pi \Theta_\infty)
+&\\
&+
e^{\p2 \tin}(- \cos \pi \Th_p +\cos \pi \Th_{1-p}\cos \pi \Theta_\infty)\},&\\
&m^p_{22} =-\frac{e^{-\p2 \tin}}{\sin^2\pi \si}
\{e^{-\p2 \tin}(\cos\pi\Th_{1-p}-\cos\pi \Th_p\cos \pi \si)
+&\\
&+e^{\p2 \tin}(- \cos \pi \Th_{1-p}\cos\pi\si
+\cos \pi \Th_p\cos 2\pi \si)\},&\\
&m^p_{12}=-\frac{\tilde{r}\Theta_\infty\hat{m}^p_{12}}{\pi\hat{d}}\,e^{\p2\tin}\sin(\pi\tin),&\\
&m^p_{21}=-\frac{\hat{m}^p_{21}}{\tilde{r}\Theta_\infty\pi\hat{d}}\,e^{\p2\tin}\sin(\pi\tin),&\\
\end{eqnarray*}
Here $\hat d$ is the same as in Theorem {\rm\ref{th2}}. Both Stokes multipliers in this case vanish, $s_1=s_2=0$.
\end{corollary}
\begin{remark}\label{rem:3p-extension}{\rm
We recall that in basic Theorems \ref{th1} and \ref{th2} the real part of parameter $\sigma$ belongs to the semi-interval
$[0,1)$. Therefore, in the resultas obtained above we have the corresponding restrictions on the parameters $\Theta$'s.
Say, in Corollary~\ref{th45} put $\sigma=-\Theta_\infty$, thus $\Re\Theta_\infty\in(-1,0]$. If we consider another complete
degeneration, $d(-\sigma)\to0$, then we find the same formulae but $\Re\Theta_\infty\in[0,1)$.
In Section~\ref{sec:spec-meromorphic-solution} we consider in detail the case $\Theta_\infty=-1$, $\Theta_0=\Theta_1=1/2$.
For the general values of
$\Theta_\infty$ the result of Corollary~\ref{th45} can be extended with the help of the B\"acklund trnasformations
considered in Appendix~\ref{sec:trans}. Analogous comments apply to Corollaries~\ref{th42} and \ref{th43}.
}\end{remark}
\begin{remark}{\rm
Functions $y$ and $z$ obtained in Corollaries~~\ref{th42}, \ref{th43}, and \ref{th45} are regular at $t=0$ and the
correction term is evaluated as $\mathcal{O}(t)$. In fact all these solutions are holomorphic at $t=0$, so that the
corresponding solutions of System~\eqref{eq:ids1}--\eqref{eq:ids2} are meromorphic functions. We omit a detailed proof
of this statement because these solutions were studied in paper~\cite{KO}.
}\end{remark}

Now we outline the second step of our approach to finding the asymptotic behavior of solutions for parameters
satisfying the conditions: $b(\pm\sigma)\in\mathbb{Z}$, $c(\pm\sigma)\in\mathbb{Z}$, and
$d(\pm\sigma)\in\mathbb{Z}$. These special cases can be characterized in terms of the Stokes multipliers:
\begin{enumerate}
\item $s_1s_2=2e^{\pii\tin}(-\cos\pi\tin+\cos\pi(\Th_0+\Th_1))$
$\Rightarrow$
$b(\pm\si)\in \Bbb Z$,
\item $s_1s_2=2e^{\pii\tin}(-\cos\pi\tin+\cos\pi(\Th_0-\Th_1))$
$\Rightarrow$ $c(\pm\si)\in \Bbb Z$,
\item $s_1s_2=0$ $\Rightarrow$ $d(\pm\si)\in \Bbb Z$.
\end{enumerate}

Above we considered the cases $b(\pm\sigma)=0$, $c(\pm\sigma)=0$, and
$d(\pm\sigma)=0$. It is already mentioned in Remark~\ref{rem:3p-extension} that the case $d(\pm\sigma)\in\mathbb{Z}$
can be treated with the help of application to asymptotics for $d(\pm\sigma)=0$ the
B\"aclund transformations considered in Appendix~\ref{sec:trans}. These transformations
shift $\Theta_\infty$ by even integers which means an arbitrary integer shift of $d(\pm\sigma)$. It is well-known that
analogous transformation is easy to write for $\Theta_0$ and $\Theta_1$, which generate an integer shift of
$b(\pm\sigma)=0$ and $c(\pm\sigma)=0$ together with the corresponding asymptotics.
\begin{subsection}
{Case $\Re\si=1$ and $\si=1$}\label{subsec:Re-sigma-1}
\end{subsection}
\begin{theorem}\label{th51}
Let $\de, h,r \in\Bbb C\back\{0\}$ and $\Re \de \in[0,1)$,
There exists a solution of system {\rm (\ref{eq:ids1})--(\ref{eq:ids3})} with the following asymptotic
expansion as
$t\to 0$
$$
y=1+y_1t+\mathcal{O}(t^{2-2\Re\de}),\,z=\frac {w+\mathcal{O}(t^{1-2\Re\de})}{2y_1^2}\frac 1t.
$$
Here
\begin{eqnarray*}
&y_1=\frac 1h\frac \vp{\de^2}t^{-\de}+\frac \vp{\de^2}
+\frac h4\big(1-\frac{\de^2}{\vp^2}\big) \frac \vp{\de^2}t^\de,&\\
&w=\frac{\de-\vp}h \frac \vp{\de^2}t^{-\de}
+1-\frac{\vp^2}{\de^2}-(\de +\vp)\frac h4\big(
1-\frac{\de^2}{\vp^2}\big)\frac \vp{\de^2}t^\de,&\\
&\vp=1-\Th_0-\Th_1.&
\end{eqnarray*}
Asymptotic expansion of function $u=u(t)$ as $t\to 0$ is given by the following formula:
\begin{eqnarray*}
&u=-rt^{\tin}\big(
1+\frac{t^{1-\de}\vp}{h\de^2(1\!-\!\de)}(\frac{\de\!-\!\vp}2-\Th_1)
+&\\
&+\frac t2 -\frac{t\vp}{\de^2}(\frac\vp2 \!+\!\Th_1)
-\frac{t^{1+\de}h\vp}{4\de^2 (1\!+\!\de)}(1-\frac{\de^2}{\vp^2})
(\frac{\de\!+\!\vp}2+\Th_1)+\mathcal{O}(t^{2-2\de})\big).&
\end{eqnarray*}
The corresponding monodromy data are given by formulas from Theorem \ref{th2},
where parameters are
$$
\de =1-\si,\quad
h=\left(\frac{1-\Th_0-\Th_1}{2(1-\si)^2\si^2s^2}\right)
\left(\frac
{(\Th_1+\si)^2-\Th_0^2}{\Th_1+\Th_0- \si}\right).
$$
\end{theorem}
\begin{remark}
Function $\ze(t)$, corresponding to solution
of system (\ref{eq:ids1})--(\ref{eq:ids3}), described in Theorem \ref{th51},
has the following asymptotics as $t\to 0$
$$
\ze(t)=\de+\frac14((1-\de)^2+\Th_0^2-\Th_1^2+2\Th_0\tin)+
\frac{2\vp\de}{ht^\de(\de-\vp)-2\vp}+\mathcal{O}(t).
$$
\end{remark}
\begin{theorem}\label{th52}
Let $h_1 \in\Bbb C$.
Then, there exists a solution of system {\rm (\ref{eq:ids1})--(\ref{eq:ids2})} with the following asymptotics as
$t\to 0$
$$
y=1+y_1t+\mathcal{O}(t^2),\,z=\frac w{2y_1^2}\frac 1t+\mathcal{O}(1).
$$
where
$$
y_1=\frac1{2\vp}-\frac\vp 2(\ln t+h_1)^2,
w=\frac12(1+\vp (\ln t+h_1))^2.
$$
Asymptotic expansion of function $u=u(t)$ as $t\to 0$ is given by the following formula:
\begin{eqnarray*}
&u=-rt^{\tin}\big\{1+\frac{\Th_1}2(\vp-\frac1\vp)t
+\frac{\Th_1\vp t}2((\ln t+h_1)-1)^2
+&\\&
+\frac t4\big\{\vp(\ln t+h_1)(\vp(\ln t+h_1)+2-2\vp)+1-2\vp +2\vp^2\big\}
+\mathcal{O}(t^2\ln^m t)
\big\}.&
\end{eqnarray*}
\end{theorem}
\begin{remark}
Function $\ze(t)$, corresponding to solution
of system (\ref{eq:ids1})--(\ref{eq:ids3}), described in Theorem \ref{th51},
has the following asymptotics as $t\to 0$
$$
\ze(t)=\frac14(1+\Th_0^2-\Th_1^2+2\Th_0\tin)+
\frac{\vp}{\vp (\ln t+h_1)-1}+\mathcal{O}(t).
$$
\end{remark}

Let us briefly describe the derivation of Theorems \ref{th51} and \ref{th52}.
Theorem \ref{th51} gives us, in particular, case  $\Re\si=1$, $\si\neq1$.
Theorem \ref{th52} corresponds to the case  $\si=1$.
It is easy to see that
$y=1+ y_1t(1+\mathcal{O}(t))$, $z=\frac{z_{-1}}t(1+\mathcal{O}(t))$, with
$\mathcal{O}(t^\ep)<|y_1|,|z_{-1}|<\mathcal{O}(t^{-\ep})$
and sufficiently small $\ep>0$ give valid asymptotic expansion for system (\ref{eq:ids1})--(\ref{eq:ids2}).
For $y_1$ and $z_{-1}$ we have the following system:
\begin{eqnarray*}
t\dot y_1 &=&1-\vp y_1 -2z_{-1}y_1^2,\\
t\dot z_{-1} &=& \vp z_{-1}+2z_{-1}^2 y_1.
\end{eqnarray*}
There are three types of solutions of this system:
\begin{enumerate}
\item a rational function of $t^\de$, where $\Re\de =0$;
\item a rational function of $\ln t$;
\item a fixed point, that is $y_1,z_{-1}=\const$.
\end{enumerate}
Theorem \ref{th51} corresponds to case 1), Theorem \ref{th52} corresponds to case 2),
and Theorem \ref{th54} (see below) corresponds to case 3).

The monodromy data were obtained as follows.
It is easy to see that formulas from Theorem \ref{th2}
are not degenerate when $\Re\si=1$, $\si\neq 1$.
So we can use them for solutions described in Theorems \ref{th51} and \ref{th52}.
The only thing to find is parameter $s^2$.
However, it can be easily seen that formulas from Theorem \ref{th51} become formulas
from Theorem \ref{th2},
if we put $\de=1-\si$, $\Re\de>0$. Therefore,
formulas from Theorem \ref{th51} are valid in interval
$0\leq \Re\de<1$. Then, to achieve the correspondence with function $y(t)$ from
Theorem \ref{th1},
parameters $s^2$ and $h$ should be connected as it is stated in Theorem \ref{th51}.

To write the monodromy data for $\si=1$ ($\de=0$), we need the following notations
\begin{eqnarray*}
&\hat E^0_f(\de)=\bm
\frac{e^{\p2(1-\de)}}{\Ga(1+c(\de-1))\Ga(-b(1-\de))}&
-\frac{e^{-\p2(1-\de)}}{\Ga(1+c(1-\de))\Ga(-b(\de-1))}\\
\frac{e^{\p2(1-\de)}}{\Ga(-c(1-\de))\Ga(1+b(\de-1))}&
-\frac{e^{-\p2(1-\de)}}{\Ga(c(\de-1))\Ga(1+b(1-\de))}
\em,&\\
&\hat E^1_f(\de)=\bm
-\frac1{\Ga(-c(1-\de))\Ga(-b(1-\de))}&
\frac1{\Ga(-c(\de-1))\Ga(-b(\de-1))}\\
-\frac1{\Ga(1+c(\de-1))\Ga(1+b(\de-1))}&
\frac1{\Ga(1+c(1-\de))\Ga(1+b(1-\de))}
\em,&\\
&C_f(\de)=
\bm
-\frac{e^{-\pii d(1-\de)}}{\Ga(1+d(\de-1))}&
-\frac1{\Ga(1-d(1-\de))}\\
-\frac{e^{-\pii d(\de-1)}}{\Ga(d(1-\de))}&
\frac1{\Ga(-d(\de-1))}
\em.&
\end{eqnarray*}
\begin{theorem}
Solution of system (\ref{eq:ids1})--(\ref{eq:ids3}), defined in Theorem \ref{th52}, generates the following
monodromy data
$$
E^p=\frac d{d\de}\big(\hat E^p_f(\de)
\diag\{\frac1{s_0(\de-1)}(1-\frac {s_1}2\de),
s_0(1+\frac{s_1}2\de)\} C_f(\de)\big)\big|_{\de=0}\diag\{1,r\}.
$$
Here
$$
s_0=\frac 12\sqrt{(\Th_1+1)^2-\Th_0^2},\
h_1=-\frac{2(1+\Th_1)}{(\Th_1+1)^2-\Th_0^2}+2-s_1-\frac1{\Th_1+\Th_0-1}.
$$
\end{theorem}
So, the matrices $E^p$ are given. The monodromy matrices $M^p$ can be found from their definition:
$$
M^p=(E^p)^{-1} e^{\pi \imath \si_3 \Theta_p} E^p.
$$
\begin{theorem}\label{th54}
There exists solution of system {\rm (\ref{eq:ids1})--(\ref{eq:ids2})} with the following asymptotics as
$t\to 0$
$$
y=1+\frac t{1-\Th_1-\Th_0}+\mathcal{O}(t^2),\,
z=\mathcal{O}(1).
$$
Monodromy data from Theorem \ref{th2}, with $\si=2-\Th_0-\Th_1$, correspond to this solution.
\end{theorem}
Theorem \ref{th54} can be derived as follows.
Let us put $\de=-1+\Th_0+\Th_1=-\vp$ and tend $h\to \infty$
in formulas from Theorem \ref{th51}.
Let us note that this solution was already described among the special solutions above.
\section{Special Meromorphic Solution}\label{sec:spec-meromorphic-solution}

As an illustration of how one can use the results for monodromy data to derive the connection formula,
we consider an example. The example will also clarify how the degeneration procedure, discussed in the previous
section, can be performed in particular cases.

In this section we assume that the coefficients of Equation~\eqref{eq:P5} are fixed as follows:
\begin{equation}\label{sys:coeff-gamma=0}
\alpha=\frac{l^2}8,\quad
\beta=-\frac{l^2}8,\quad
\gamma=0,\quad
\delta=-\frac12,\qquad
l\in\mathbb{Z}_+
\end{equation}
There is a quadratic auto transformation of Equation~\eqref{eq:P5} which maps these coefficients
into the following ones:
$$
\alpha=\frac{l^2}2,\quad
\beta=0,\quad
\gamma=4,\quad
\delta=0,\qquad
$$
This equation can be also mapped into the complete third Painlev\'e equation (see, e.g., \cite{K2004}).
\begin{proposition}\label{prop:general-l-specsolat0}
For any $l\in\mathbb{Z}_+$  and $a_l^l\in\mathbb{C}$ there exists the unique solution of Equation~\eqref{eq:P5}
with the following asymptotics as $t\to0$,
\begin{equation}\label{eq:general-l-specsol-y-at0}
y(t)=-1+\sum\limits_l^\infty a_k^lt^k
\end{equation}
\end{proposition}
In fact the solution is holomorphic in some neighborhood of $t=0$, so that Expansion~\eqref{eq:general-l-specsol-y-at0}
is nothing but the Taylor series. The proof can be done in a straightforward way: substitution of the
Expansion~\eqref{eq:general-l-specsol-y-at0} into Equation~\eqref{eq:P5} to  obtain the recurrence relations for the
coefficients $a_k^l$. Then one observes that the recurrence allows one to uniquely present $a_k^l$ as the polynomials
of the first coefficient $a_l^l$. Then the Wasow theorem provide us with the existence of the solution. The further
analysis of the recurrence relation allows one to prove convergence of Expansion~\eqref{eq:general-l-specsol-y-at0},
which also implies the uniqueness.

Let us mention some properties of Solution~\eqref{eq:general-l-specsol-y-at0}. We also consider the corresponding
Taylor expansion of function $z(t)$:
\begin{equation}\label{eq:general-l-specsol-z-at0}
z(t)=\frac{l-1}4-\frac{t}8+\sum\limits_{k=l}^\infty c_{k}^lt^{k},\qquad
c_l^l=-\frac{l+1}8\,a_l^l,\quad
\end{equation}
the coefficients $c_{k}^l$  for $k\geq l$ are polynomials of $a_l^l$.
\begin{remark}
We do not write explicitly expansion for function $\zeta(t)$, see Equation~\eqref{eq:zeta-def}, because
it can be easily obtained via Equation~\eqref{eq:zeta-derivativ} by integration of Expansion~\eqref{eq:general-l-specsol-z-at0}.
The constant of integration is $l(l-1)/4$.
\end{remark}

Properties of Expansions~\eqref{eq:general-l-specsol-y-at0} and \eqref{eq:general-l-specsol-z-at0} depend on
parity of the number $l$:

If $l$ is even, then $y(t)$ is even (as the function of $t$), so that $a_{2m+1}^l=0$ for all $m\in\mathbb{Z}_+$.
The function $z(t)$ is neither odd nor even, however, the first odd coefficients also vanish
$c_{2m+1}^l=0$ for $l/2\leq m\leq l-1$. Note that $c_{2l+1}^l=\tfrac{\left(a_l^l\right)^2}{32}$.

If $l$ is odd, then $a_{2m}^l=0$ for $\tfrac{l+1}{2}\leq m\leq l-1$, while $a_{2l}=-\tfrac{(a_l^l)^2}{2}$.
The corresponding function $z(t)+\tfrac{1-l}{4}$ is an odd function of $t$, i.e., $c_{2m}^l=0$ for all $m\in\mathbb{Z}_+$.

If we put $a_l^l=0$, then $y(t)=-1$. It is the simplest rational solution of Equation~\eqref{eq:P5}. It is known
\cite{KLM} that all rational solutions with the asymptotics $-1+\mathcal{O}\left(1/t\right)$ as $t\to\infty$ can be
obtained from this elementary solution via application of the B\"acklund transformations. Some of these solutions
has at $t=0$ the Taylor expansion of the form \eqref{eq:general-l-specsol-y-at0}, however, in these cases the coefficients
of Equation~\eqref{eq:P5} differ from \eqref{sys:coeff-gamma=0}.
As follows from \cite{KLM} $y(t)=-1$ is the only case when Solution~\eqref{eq:general-l-specsol-y-at0} is rational.

An interesting question is whether Solution~\eqref{eq:general-l-specsol-y-at0} is a truly transcendental Painlev\'e
function in the sense of Umemura (see \cite{U}). For Equation~\eqref{eq:P5} this question was analysed in \cite{WH}.
However, these results are not formulated in terms of the Expansion~\eqref{eq:general-l-specsol-y-at0} and therefore require
additional investigation.

There are a few different ways of how to choose parameters $\Theta_k$ corresponding to coefficients
\eqref{sys:coeff-gamma=0}, see  Equations~\eqref{eq:coeffP5}. Our choice is
\begin{equation}\label{eqs:theta-special}
\Theta_0=\frac12,\qquad
\Theta_1=\frac12,\qquad
\Theta_\infty=-l.
\end{equation}
As follows from \cite{WH} for even $l$ there are no classical solutions, apart from the trivial one $y(t)=-1$.
So, we consider the odd values of $l$. If $l=1$, then $\Theta_0+\Theta_1+\Theta_\infty=0$; so that we can put
in System~\eqref{eq:ids1}, \eqref{eq:ids2} $z(t)=0$ and find $y(t)$ as a solution of the Riccati equation. After that
we have to check whether there is a solution of the Riccati equation with Expansion~\eqref{eq:general-l-specsol-y-at0}
at $t=0$. Omitting elementary details we find that such solution $y(t)\equiv y_1(t)$ exists and is unique,
\begin{equation}\label{eq:specBessel}
y_1(t)=\frac{I_1\left(t/2\right)+I_0\left(t/2\right)}{I_1\left(t/2\right)-I_0\left(t/2\right)},\quad
z_1(t)=0,
\end{equation}
here $I_n(\cdot)$ for $n=0,1$ is the modified Bessel function of the first kind~\cite{BE2}. Solution~\eqref{eq:specBessel}
has the following Taylor expansion at $t=0$,
$$
y_1(t)=-1-\frac{t}2-\frac{t^2}{8}-\frac{t^3}{64}+\frac{t^5}{3072}+\mathcal{O}\left(t^6\right).
$$
Comparing it with Expansion~\eqref{eq:general-l-specsol-y-at0} we find $a_1^1=-1/2$.

To proceed further we have to apply to Solution~\eqref{eq:specBessel} the second B\"acklund transformation given in
Theorem~\ref{dress} (see Appendix~\ref{sec:trans}). In this way, by mathematical induction, we find that for all odd
$l$ there exists only one classical solution $y_l(t)$ $z_l(t)$ of System~\eqref{eq:ids1}, \eqref{eq:ids2}, which
can be presented explicitly in terms of the modified Bessel functions with the Taylor
Expansion~\eqref{eq:general-l-specsol-y-at0}. In particular,
\begin{align*}
y_3(t)&=-1-\frac{t^3}{192}+\frac{t^5}{12288}+\mathcal{O}\left(t^6\right),&\quad
z_3(t)&=\frac12-\frac{t}{8}+\frac{t^3}{384}-\frac{t^5}{16384}+\mathcal{O}\left(t^7\right),\\
y_5(t)&=-1-\frac{t^5}{61440}+\frac{t^7}{5898240}+\mathcal{O}\left(t^8\right),&
z_5(t)&=1-\frac{t}{8}+\frac{t^5}{81920}-\frac{t^7}{5898240}+\mathcal{O}\left(t^9\right).
\end{align*}
Corresponding values of the parameters $a_3^3=-\tfrac1{192}=-\tfrac1{3\cdot2^6}$,
$a_5^5=-\tfrac1{61440}=-\tfrac1{3\cdot5\cdot2^{12}}$.
\begin{conjecture}
Solution~\eqref{eq:general-l-specsol-y-at0} with odd $l$ is classical iff $a_l^l=-\frac1{l!2^{2l-1}}$.
\end{conjecture}
The corresponding monodromy matrices are
upper triangular and can be calculated with the help of Section~\ref{sec:deg0}.

Because of the restriction $0\leq\Re\sigma\leq1$, see Sections~\ref{sec:zero} and Subsection~\ref{subsec:Re-sigma-1},
we cannot directly apply our results to Solution~\eqref{eq:general-l-specsol-y-at0} for general $l\in\mathbb{Z}_+$.
However, as we show below, the case $l=1$ is tractable with the help of Sections~\ref{sec:zero}, \ref{sec:deg0}
and Appendix~\ref{sec:allterms}.
The general case of odd $l$ can be further treated with the help of the B\"acklund transformations,
see Appendix~\ref{sec:trans}. As for the even values of $l$ they also can be studied within the framework presented
in this paper but with the help of a different "initial solution". The case of $l>1$ is interesting in view of
applications, however, it goes far beyond our goals in this paper and will be considered elsewhere.

In fact, even in the case $l=1$ solutions show different behavior as $t\to\infty$ depending on the value of
$a_1^1$. For the illustrative purposes we choose the case $a_1^1=-2$. The reader can find numeric illustration of the
theoretical results obtained in this section in Subsection~\ref{subsec:numeric-extra-terms}.
With the help of Maple Code we can be more specific about the asymptotic expansion at $t=0$ of the main object of
our study in this section:
\begin{equation}\label{eq:y-spec-meromorphic-a-2}
y(t)=-1-2t-2t^2-\frac{31}{16}t^3-\frac{15}8t^4-\frac{2833}{1536}t^5-\frac{2789}{1536}t^6+O(t^7).
\end{equation}
One can notice that the first terms in this expansion are negative. So it rises a natural question: whether all
coefficients of this expansion are negative? The following Statement can be confirmed by mathematical induction
applied to the recurrence relation for coefficients $a_1^1$ obtained via substitution of
Expansion~\eqref{eq:general-l-specsol-y-at0} into Equation~\eqref{eq:P5}.
\begin{statement}\label{stm:asympt-ak}
All polynomials $a_k^1(a_1^1)$ have real coefficients. For large values of $a_1^1$ these polynomials have the
following asymptotics
as $a_1^1\to\infty$:
$$
a_k^1(a_1^1)=\left(-\frac12\right)^{k-1}(a_1^1)^k+\mathcal{O}\left((a_1^1)^{k-2}\right).
$$
\end{statement}
Consider
$$
a_5^1(a_1^1)=\frac{a_1^1}{16}\left((a_1^1)^4-\frac5{16}(a_1^1)^2+\frac1{192}\right).
$$
The largest root of this polynomial is
$$
R=\frac1{24}\sqrt{90+6\sqrt{177}}=0.5429868659\ldots.
$$
Numerical studies suggest the following
\begin{conjecture}
All zeroes of all polynomials $a_k^1(a_1^1)$ lie inside the circle of radius $R$ centered at the origin of the complex plain.
\end{conjecture}
This conjecture together with Statement~\ref{stm:asympt-ak} implies that for real $a_1^1<-R$ all coefficients
of Expansion~\eqref{eq:general-l-specsol-y-at0} are negative, while for $a_1^1>R$ their signs alternate.
\begin{remark}
Define a mapping of the sequence $\{{\rm sign}(a_k^1)\}$ into the set $\{0,1\}$, say, the minus sign goes to $0$
and the plus sign to $1$. Denote this sequence of zeroes and ones as $\epsilon_k$. Then our result can be formulated
as follows: for $a_1^1>R$ $\epsilon_k=0$, for $a_1^1>R$ $\epsilon_{2n+1}=1$ and $\epsilon_{2n}=0$, $n\in\mathbb{Z}_+$.
It is an interesting problem to study the sequences $\{\epsilon_k\}$ for $a_1^1\in[-R,R]$. For example, in the
"integrable" case, $a_1^1=-1/2$, the corresponding sequence $\{\epsilon_k\}$ is periodic with the minimal period $54$.
\end{remark}
\begin{subsection}
{Monodromy data}
\end{subsection}\label{subsec:monodromy data}
Here we calculate the monodromy data for $y(t)$ defined by Equation~\eqref{eq:y-spec-meromorphic-a-2}.
Comparing the general asymptotic behavior of $y(t)$ as $t\to0$ \eqref{eq:y->0} with
Expansion~\eqref{eq:general-l-specsol-y-at0} we see that the only possibility to get it is to consider a limit
$\sigma\to1$. The functions defining Asymptotics~\eqref{eq:y->0} for our choice
of $\Theta$-parameters \eqref{eqs:theta-special} read:
\begin{equation}\label{eqs:bcda}
b(\sigma)=\frac{1+\sigma}2,\quad
c(\sigma)=\frac{\sigma}2,\quad
d(\sigma)=-\frac{1-\sigma}2,\quad
a(\sigma)=\frac{\sigma(1+\sigma)}4.
\end{equation}
Therefore, we see that $b(\pm\sigma)\in \Bbb Z$, $d(\pm\sigma)\in \Bbb Z$. These conditions show that here we meet
much "deeper" degeneration procedure rather than that considered in Section~\ref{sec:deg0}. It suggests that it is
more reliable to do the direct degeneration of the results presented in the main Theorems~\ref{th1} and \ref{th2},
rather than further degeneration of the results of Section~\ref{sec:deg0}. In the intermediate degenerations
considered in this section it was assumed that the conditions  $b(\pm\sigma)\in \Bbb Z$, $d(\pm\sigma)\in \Bbb Z$,
and $\sigma=1$ do not hold simultaneously.

Before making the limit procedure we recall that first we have to:
1) freeze $t$ in Equations~\eqref{eq:y->0}--\eqref{eq:z->0};
2) choose $\Theta$-parameters according to Equations~\eqref{eqs:theta-special}; and
3) make the limit transition $\sigma\to1$.

Since $y(0)=-1$ we see that the only way we can achieve it is to send parameter $s^2\to\infty$. More carefully
examining the limit, we find that in fact we have to assume that
$$
s^2=\frac{s_0^2}{(1-\sigma)^2},\qquad
s_0^2\in\mathbb{C}\setminus\{0\},\quad
\sigma\to1,
$$
where the parameter $s_0^2\neq0$ is a complex number. After this assumption Equations~\eqref{eq:y->0}--\eqref{eq:z->0}
imply:
\begin{equation}\label{eqs:yzu-spec0-l-1}
y(t)=-1+\mathcal{O}(t),\quad
z(t)=o\left(t^{-1}\right),\quad
u(t)=\frac{\hat{r}s_0^2}{2}\left(1+\mathcal{O}(t)\right),\quad
\hat{r}=\frac{r}{1-\sigma},\quad
\hat{r}\in\mathbb{C}\setminus\{0\}.
\end{equation}
We put $o\left(t^{-1}\right)$ as the estimate for $z$ because
we know small-$t$ expansion of the $\tau$-function obtained by Jimbo \cite{J}. The estimate
$\mathcal{O}\left(t^{3-2\Re\sigma}\right)$, can be traced from the corresponding term
$\mathcal{O}\left(t^{3-3\sigma}\right)$ in the expansion of the $\tau$-function, see Equations~\eqref{eq:tau}
and \eqref{eq:zeta-derivativ}. This term contains the factor $s^2(1-\sigma)^2$ in the denominator.
After two differentiations, this term gains an additional factor $(1-\sigma)^2$, so that the factor $1/s^2$
kills this term in the limit. This mechanism is analogous to disappearance of the explicitly written term
(see Equation~\eqref{eq:z->0}) of the order $\mathcal{O}\left(t^{1-2\sigma}\right)$.
Since we know that Solution~\eqref{eq:general-l-specsol-y-at0} is the only
solution with the property $y(0)=-1+\mathcal{O}(t)$, then in fact $z(t)$ is given by \eqref{eq:general-l-specsol-z-at0}
with $l=1$.

Now we perform the limit in the formulas for the monodromy data given in Theorem~\ref{th2}. Performing the
limit $\sigma\to1$ in Equations~\eqref{eqs:s1-s2at0} with the help of relations
for $d(\sigma)$ given in \eqref{eqs:bcda} and $\hat{r}$ from \eqref{eqs:yzu-spec0-l-1}, we see that the
Stokes multipliers vanish,
\begin{equation}\label{eq:s1=s2=0}
s_1=s_2=0.
\end{equation}
Further we calculate the limits for the monodromy matrices. For the intermediate matrices $\hat M^p$ we find
$\hat M^0=\hat M^1$. This, together with Conditions~\eqref{eq:s1=s2=0} imply:
\begin{equation}\label{eqs:monodromy-special}
M^0=M^1,\qquad
\left(M^0\right)^2=-I,\qquad
M^\infty=-I.
\end{equation}
The matrix elements
\begin{equation}\label{eq:matrix-elements}
m_{11}^p=-m_{22}^p=-\frac{\imath}4\left(4s_0^2+\frac{1}{s_0^2}\right),\quad
m_{12}^p=\frac{\imath\hat{r}}2\left(4s_0^2-\frac1{s_0^2}\right),\quad
m_{21}^p=-\frac{\imath}{8\hat{r}}\left(4s_0^2-\frac1{s_0^2}\right),
\end{equation}
where $p=0,1$.
This completes calculation of the monodromy data.

Now we have to find how the monodromy data relate to the Taylor expansion of $y(t)$ at $t=0$.
Equations~\eqref{eqs:yzu-spec0-l-1} do not allow us to find it. We can perform the limit more
carefully and find that:
$$
y\underset{t\to0}{=} -1+\frac{s_0^2}{2}t^\sigma+\mathcal{O}(\sigma-1)
+\mathcal{O}\left(t^{1+\sigma}\right)+\mathcal O\left(t^{2-\Re\si}\right).
$$
All the estimates above depend on $\sigma$ and $t$. Our purpose is to put $\sigma=1$.
Why would the last two estimates remain finite at this limit? If they (one of them) would blow up, then it means
that the monodromy data \eqref{eq:matrix-elements} does not correspond to any solution of P5. On the other
hand, from the derivation of asymptotics at the point of infinity we know that there is a solution of P5
corresponding to these monodromy data.

The last estimate looks "dangerous", since in the limit it may contribute to the second term of asymptotics.
So, in fact, the $\sigma$-dependence of the last estimate requires further investigations. This
estimate can be obtained explicitly in two ways, either directly from Equation~\eqref{eq:P5}
or via Equation~\eqref{eq:ids2} and with the help of asymptotics for the function $z$. Another
way is the direct calculation of the monodromy data for this solution via Equation~\eqref{mainl}.

Since all these calculations looks cumbersome, we probably present them somewhere, here we announce
the correct result and suggest an indirect proof.
\begin{proposition}
\begin{equation}\label{eq:a11-s0}
a_1^1=\frac{s_0^2}2+\frac1{8s_0^2}=\frac{i}2m_{11}^p,
\end{equation}
\end{proposition}
\begin{proof}
Note that $m_{12}^p$ and $m_{21}^p$ depend on the parameter $\hat r$ whilst functions $y(t)$ and $z(t)$ do not.
So, $a_1^1$ may depend only on the quadratic combination, $m_{21}^pm_{12}^p=-1+\left(m_{11}^p\right)^2$.
Thus $a_1^1$ should be an entire function of $m_{11}^p$. On the other hand, the inverse function $m_{11}^p(a_1^1)$
also should be an entire function. It means (see Theorem 4.3 of \cite{M2}), that $a_1^1=C_1m_{11}^p+C_2$, where
$C_1$ and $C_2$ are some constants. For $a_1^1=0$ we have $m_{11}^p=0$, see Corollary 3 and Proposition 3 (item (1))
of \cite{AK97MRL}. Thus $C_2=0$.  In the case $a_1^1=-1/2$, the monodromy matrices are diagonal which implies
$s_0^2=\pm1/2$, see the last two Equations~\eqref{eq:matrix-elements}. Our case corresponds to the lower triangular
case: $z=0$, $\Theta_0+\Theta_1+\Theta_\infty=0$, of Equation~\eqref{mainl}, the other one corresponds to the upper
triangular: $z=-\Theta_0$, $-\Theta_0-\Theta_1+\Theta_\infty=0$. The monodromy element, say, $m_{11}^0$ in our case
is defined by the monodromy of the function $\lambda^{\Theta_0/2}$ and equals $e^{2\pi\imath/4}=\imath$, in the
upper triangular case $m_{11}^0$ is defined by the monodromy of the function $\lambda^{-\Theta_0/2}$, which is
$e^{-2\pi\imath/4}=-\imath$. Using this fact we arrive at Equation~\eqref{eq:a11-s0}.
\end{proof}
In Subsection~\ref{subsec:numeric-extra-terms} we provide a numeric evidence of Equation~\eqref{eq:a11-s0}. For
for Solution~\eqref{eq:y-spec-meromorphic-a-2} with $a_1^1=-2$. The key monodromy data for this solution are as follows:
$$
m^1_{11}=m^0_{11}=-m^1_{22}=-m^0_{22}=4\imath.
$$
\subsection{Asymptotics as $t\to+\infty$}\label{subsec:asympt-spec-sol+infinity}

Since we have explicit formulae for solution for parameter $a_1^1=-1/2$ (see Equation~\eqref{eq:specBessel}),
we begin with the asymptotics for this case:
\begin{equation}\label{eq:y1-asymptotics+infty}
y_1(t)\underset{t\to+\infty}=-2t+2+\sum\limits_{n=1}^\infty\frac{a_n}{(2t)^n}+\mathcal{O}\left(t^{-\infty}\right),
\end{equation}
where
$a_1=3$, $a_2=18$, $a_3=153$, $a_4=1638$. In general $a_n$ is a sequence of positive integers counting the number of
Feynman diagrams in a problem in quantum electrodynamics (see sequence A005412 of \cite{OEIS}). The series is divergent,
$$
\frac{a_n}{2^n}=\frac{4}{\pi}\cdot\frac{n!}{n-1/2}\left(1+\mathcal{O}\left(\frac1{n^2}\right)\right).
$$
The complete asymptotic expansion can be written as a transseries with exponentially small terms. The estimate
above can be obtained with the help of this expansion.

Here, to give an instructive example we consider only the case of $a_1^1<0$. Obviously, for the real values of $a_1^1$
our solution is real for real $t$.
In this subsection we apply the results
obtained in paper \cite{AK1}. In that paper asymptotics as $t\to+\infty$ of solutions were parameterized via auxiliary
parameters $\hat v$ and $\beta_0$. We use general formulae, Equations (3.5) of \cite{AK1}, and substitute there
our particular monodromy data obtained in Subsection~\ref{subsec:monodromy data}:
\begin{align}
\beta_0&=\frac1{2\pi\imath}\ln\left(1+(m_{11}^p)^2\right)
=\frac1{2\pi i}\ln\left(1-(4a_1^1)^2\right),\label{eq:beta0}\\
\hat v&=\frac{\sqrt{2\pi}}{2a_1^1}\,\frac{2^\beta_0\exp\left(\pi\beta_0/2\right)}{\Gamma(\beta_0)}
=\frac{\sqrt{\pi\imath}\,2^\beta_0}{\Gamma(\beta_0)\sqrt{\sin(\pi\beta_0)}}.\label{eq:vh}
\end{align}
We see that if $-1/2<a_1^1<0$ parameters $\beta_0$ and $\hat v$ are defined uniquely in terms of the
initial data $a_1^1$, where  the branch of $\ln$ is fixed in a natural way: $\Re\beta_0=0$. In terms
of $\beta_0$, parameter $\hat v$ is fixed up to a sign which causes no problem because asymptotics are
invariant under this change.
In this case we have to use the results reported in Theorems~3.1 and 3.2 of \cite{AK1}.
In particular, it means that our solution is regular on the real positive semiaxis and the asymptotics of
functions $y(t)$ and $z(t)$ the reader can find in Theorem~3.1. Theorem~3.2 gives us the asymptotics of
the corresponding function $\zeta(t)$.

Consider the case $a_1^1<-1/2$ in a more detail. First we have to choose the real part of $\beta_0$ so that
one of the Theorems~3.1 or 4.1 would be applicable. It fixes uniquely parameter $\beta_0$ as follows,
$$
\beta_0=\frac12-\frac{\imath}{2\pi}\ln(4(a_1^1)^2-1).
$$
This equation implies that we are in the situation described by Theorem~4.1 of \cite{AK1}. It means that
in this case our solution has infinite number of poles on the positive semiaxis.
According to Theorem 4.1 function $y(t)$ has the following asymptotics in the proper cheese-like domain,
$$
y=\frac{\cos^2\ti x}{\sin ^2\ti x}+\mathcal{O}\left(\frac 1t\right), \quad
\ti x= \frac t4 + \gamma \ln t+\psi.
$$
where
\begin{equation}\label{eqs:gamma-psi}
\gamma=\frac1{4\pi}\ln\left(4(a_1^1)^2-1\right),\quad
\psi=-\frac1{2\imath}\ln\left(-\frac{\hat v}{\sqrt{2}}\,e^{\pi\imath/4}\right)
= \frac{\pi}4 +\gamma \log 2 +\frac 12 \op{arg} \Ga(\beta_0).
\end{equation}
The last equality holds because the quantity under the logarithm is unimodular. The reader can find the
corresponding asymptotics of function $z(t)$ and $\zeta(t)$ in Theorem~4.1 and Corollary~4.1 of \cite{AK1},
respectively.

The numeric values for the parameters $\gamma$ and $\psi$ corresponding to Solution~\eqref{eq:y-spec-meromorphic-a-2}
(with $a_1^1=-2$) are as follows:
\begin{equation}\label{eqs:gamma-beta-numeric}
\gamma=\frac{\ln(15)}{4\pi}=0.21549978\ldots,\qquad
\psi=1.27729163\ldots.
\end{equation}
\subsection{Transform $t\to -t$}
\label{subsec:t-to-minus-t}
Since general solutions of System~\eqref{eq:ids1}--\eqref{eq:ids3} are not singlevalued, transformation $t\to-t$
should be considered for analytic continuation of functions $y(t)$, $z(t)$, and $u(t)$ for
$\arg t\to\arg t+\pi(2k-1)$, for all $k\in\mathbb{Z}$. However, the solution we are interested in is a very special
one, and is in fact singlevalued, we consider only the case $k=1$. Here we present formulae which are valid for
the general solution of System~\eqref{eq:ids1}--\eqref{eq:ids3}. The purpose of this transformation is that the
analytic continuation of any solution at the point $-t$ can be presented in terms of another solution of this
system at the original point $t$. To uniquely define the latter solution we consider action of this transformation
not only on the space of the solutions but also on the manifold of monodromy data.

We refer to the definition of the canonical solutions $Y_k(\lambda)$ given in Section~\ref{sec:def-monodromy}, but
here we use an extended notation to reflect their dependence on the coefficients of Equation~\eqref{mainl}.
As one can easily see, transform $t\to -t$ induces the following transformation of the canonical solutions:
$$
Y_k(\lambda;t,z,y,u,\0,\1,\tin)=
\sigma_1 \tilde{Y}_{k+1}(\lambda;\ti t,\ti z,\ti y,\ti u,\ti \0,\ti \1,\ti \tin)\sigma_1\qquad
k\in\mathbb{Z},
$$
where $\tilde y=\tilde y(\tilde t)$, $\tilde z=\tilde z(\tilde t)$, $\tilde u=\tilde u(\tilde t)$ and
\begin{eqnarray*}
&\ti t=-t, \quad \arg \ti t =\pi +\arg t, \quad \ti \Theta_p=\Theta_p, \quad \ti\Theta_\infty =-\tin, &\\
& \ti z(\ti t)=-z (t)-\0, \quad \ti u(\ti t) =\frac 1{u(t)},\quad
\ti y(\ti t) = \frac1{y(t)}, &\\
& \ti \zeta(\ti t) = \zeta (t) -\0 t - \0 \tin.&
\end{eqnarray*}
The monodromy data are transformed as follows:
\begin{equation}\label{eq:monodromy-data-t-to-minus-t}
\begin{gathered}
\ti s_1 = e^{-2\pi \imath \tin}s_2,\quad \ti s_2 = s_1,\\
\ti M^p = \sigma_1 S_1 M^p S_1^{-1} \sigma_1,\qquad
p=0,1,\infty
\end{gathered}
\end{equation}
There is another transformation $t\to-t$ for IDS~\eqref{eq:ids1} -\eqref{eq:ids3}:
\begin{equation}\label{eqs:second-trans-t-to-minus-t}
\begin{gathered}
\ti t=-t, \quad \arg \ti t =\pi +\arg t, \quad \ti \Theta_p=\Theta_{1-p}, \quad \ti\Theta_\infty =\tin, \\
\ti z(\ti t)=-z (t)-\frac12(\Theta_0+\Theta_1+\Theta_\infty), \quad
\ti u(\ti t) =y(t)u(t)e^{-t+\pi\imath\Theta_\infty},\quad
\ti y(\ti t) = \frac1{y(t)}, \\
\ti \zeta(\ti t) = \zeta (t) -\frac12(\Theta_0+\Theta_1+\Theta_\infty)t
+\frac12(\Theta_1-\Theta_0)(\Theta_0+\Theta_1+\Theta_\infty).
\end{gathered}
\end{equation}
The corresponding transformation of the monodromy data can be found with the help of the following transformation
for the canonical solutions:
\begin{equation}\label{eq:second-trans-t-to-minus-t-Y}
\tilde{Y}_k(\tilde\lambda;\tilde t,\ti z,\ti y,\ti u,\ti{\Theta}_0,\ti{\Theta}_1,\ti{\Theta}_\infty)=
\exp\left(-\frac{t}2\sigma_3+\frac{\pi\imath\Theta_\infty}2\sigma_3\right)Y_k(\lambda; t,z,y,u,\0,\1,\tin).
\end{equation}
Here we assumed that
$$
\tilde\lambda=1-\lambda,\qquad
\arg\tilde\lambda\underset{\lambda\to\infty}{\to}\arg\lambda-\pi.
$$
The other tilde-variables are defined in Equations~\eqref{eqs:second-trans-t-to-minus-t}.
Equation~\eqref{eq:second-trans-t-to-minus-t-Y} imply the following relation for the monodromy data:
$$
\tilde{M}^\infty=M^\infty,\quad
\tilde{M}^0=M^1,\quad
\tilde{M}^1=M^1M^0(M^1)^{-1}.
$$
For our solution (see Equations~\eqref{eqs:monodromy-special}) we have $M^0=M^1$ so that $\tilde{M}^p=M^p$ for
$p=0,1,\infty$. This is consistent with the fact that Solution~\eqref{eq:general-l-specsol-y-at0}
$$
\tilde{y}(\tilde{t})=y(\tilde{t}),\quad
\tilde{z}(\tilde{t})=z(\tilde{t}).
$$
The last equality holds because of the relation $\Theta_0+\Theta_1+\Theta_\infty=0$. Thus functions $\tilde{y}$ and
$\tilde{z}$ are just the analytic continuation of $y(t)$ and $z(t)$ and, therefore have the same monodromy data.
\subsection{Asymptotics as $t\to-\infty$}\label{asympt-spec-sol--infinity}
Here we apply the first transformation considered in the previous subsection to find asymptotics of solution
\eqref{eq:general-l-specsol-y-at0} as $t\to-\infty$.
The monodromy data for our solution are given in Equations~\eqref{eq:s1=s2=0} and \eqref{eq:matrix-elements}.
Using Equations~\eqref{eq:monodromy-data-t-to-minus-t} we find that
$$
\tilde{s}_1=\tilde{s}_2=0,\quad
\ti M^0=\ti M^1=\sigma_1M^0\sigma_1.
$$
In particular we have $\hat{m}_{11}^p=-m_{11}^0$ for $p=0,1$.
On the other hand
$$
\tilde{y}(\tilde{t})=\frac1{y(t)}=y(-t)=y(\tilde{t})
$$
So, the functions $y$ and $\tilde{y}$ coincide for all arguments, but have different monodromy data! It is
explained by the fact that $\tilde{\Theta}_\infty=-\Theta_\infty=1$, therefore the corresponding functions $\tilde{z}$
and $z$ are different. It is solutions of IDS \eqref{eq:ids1}--\eqref{eq:ids3}, rather than solutions of
Equation~\eqref{eq:P5}, which are characterized uniquely by the monodromy data.

Now consider $y_1(t)$. Using Asymptotics~\eqref{eq:y1-asymptotics+infty} we find,
\begin{gather*}
y_1(t)\underset{t\to-\infty}=\frac1{y_1(-t)}=-\sum\limits_{n=1}^\infty\frac{b_n}{(-2t)^n}
+\mathcal{O}\left(t^{-\infty}\right),\\
b_1=1,\; b_2=2,\; b_3=7,\; b_4=38,\; b_5=286,\; b_6=2756,\;\ldots
\end{gather*}
where the asymptotic expansion is of the same type as \eqref{eq:y1-asymptotics+infty}. It is given by a
divergent series, all numbers $b_n$, are positive integers representing sequence A094664 of \cite{OEIS} with
the following asymptotics,
$$
\frac{b_n}{2^n}=\frac1{\pi}\cdot\frac{n!}{n-1/2}\left(1+\mathcal{O}\left(\frac1n\right)^2\right).
$$
The complete expansion can be presented as transeries with exponentially small terms. In fact, our
calculation represent an explicit relation between sequences A005412 A094664 of \cite{OEIS} which possibly
was not observed earlier.

Turning to the general value of the initial data $a_1^1$, we note that
$$
\tilde\beta_0=\beta_0,\qquad
\tilde{\hat v}=-\hat v,
$$
where $\tilde\beta_0$ and $\tilde{\hat v}$ are the parameters defined for the solution $\tilde y$, $\tilde z$, and
$\tilde u$ via Equations~\eqref{eq:beta0} and \eqref{eq:vh} with the corresponding monodromy data. Since these
parameters define asymptotic behaviour of our solution we conclude that qualitatively on the negative semiaxis
the solution behaves similar to on the positive one.

Explicit asymptotics for $-1/2<a_1^1<0$ can be found with the help of  Theorems~3.1 and 3.2 of \cite{AK1}.
In the case $a_1^1<-1/2$ the asymptotics is given in Theorem~4.1 of \cite{AK1}. For example, in the last case,
$$
y(t)\underset{t\to\infty}=\frac{\sin^2(\frac t4 -\gamma \ln (-t)-\psi)}{\cos ^2(\frac t4 -\gamma \ln (-t)-\psi)}
+\mathcal{O}\left(\frac 1t\right),
$$
where $\gamma$ and $\psi$ are given by Equations~\eqref{eqs:gamma-psi}. For the numerical values of $\beta_0$
and $\gamma$ for $a_1^1=-2$ see \eqref{eqs:gamma-beta-numeric}.
\subsection{Asymptotics for pure imaginary $t$}
\label{subsec:asympt-special-imaginary}
With the help of explicit formula \eqref{eq:specBessel} one finds:
\begin{align}
y_1(t)&\underset{t\to+\imath\infty}=\imath e^{t}-\frac{\imath e^{t}}{2t}\left(1-2\imath\sinh t\right)
+\mathcal{O}\left(\frac1{t^2}\right),\label{eq:y1-imaginary+}\\
y_1(t)&\underset{t\to-\imath\infty}=-\imath e^{t}+\frac{\imath e^{t}}{2t}\left(1+2\imath\sinh t\right)
+\mathcal{O}\left(\frac1{t^2}\right)\label{eq:y1-imaginary-}
\end{align}

Consider now the asymptotics as $t\to+\imath\infty$ for the general initial data $a_1^1<0$.
We see from \eqref{eq:a11-s0} that $\arg m_{11}^p=\frac{\pi}2$.
Theorem~\ref{th:allmon} implies:
\begin{align}
\varphi&=\frac1{2\pi\imath}\ln\left(-m_{11}^1\right)=-\frac14+\imath\kappa,\quad
\kappa=-\frac1{2\pi}\ln(-2a_1^1),\label{eq:phi-kappa}\\
\delta&=-\frac1{4\pi^2}\left(1-e^{\pi\imath(4\varphi+1)}\right)\Gamma\left(\varphi+\frac54\right)
\Gamma^2\left(\varphi+\frac34\right)\Gamma\left(\varphi+\frac14\right).\label{eq:delta}
\end{align}
With the help of the duplication formula for the gamma function we can rewrite $\delta$ as follows,
$$
\delta=\imath2^{-4\imath\kappa}e^{-2\pi\kappa}\frac{\sinh(2\pi\kappa)}{2\pi\kappa}\Gamma^2(1+2\imath\kappa).
$$

Both parameters $\nu_1$ and $\nu_2$ introduced in Theorems~\ref{th:new1} and \ref{th:new2} satisfy the
condition $\Re\nu_1=\Re\nu_2=1$, which is the boundary value for their applicability. It was mentioned
in Remark~\ref{rem:interlace-th3.1-th3.2} that in fact these theorems are valid beyond these boundary values.
In this case we can apply either Theorem for construction of the asymptotics. Suppose we consider
Theorem~\ref{th:new1}. Numerics presented in Subsection~\ref{subsec:numeric-extra-terms} show that the
leading term $\delta t^{\nu_1-1}e^t$ (see Theorem~\ref{th:new1}) delivers quite satisfactory approximation of the
function $y(t)$. So the terms $\mathcal{O}\left(1/t\right)$ are not needed, at least, for description of
the qualitative behavior of $y(t)$. At the same time the leading term of $z(t)$ in this case is just a constant,
therefore even for understanding of the qualitative behavior of $z(t)$ we need the the terms $\mathcal{O}\left(1/t\right)$.

In our case ($\Re\nu_1=1$) asymptotics presented in Theorem~\ref{th:new1} does not allow one to get
$\mathcal{O}\left(1/t\right)$-terms correctly for both functions $y(t)$ and $z(t)$. The result
presented there contains some $\mathcal{O}\left(1/t\right)$-terms but not all the terms of this order. To get these
terms we included Appendix~\ref{sec:allterms}, where we presented the complete asymptotic expansions for
$y(t)$, $z(t)$, and $\zeta(t)$ for large pure imaginary $t$ corresponding to both Theorems~~\ref{th:new1} and
~\ref{th:new2}. Following a notation introduced in Appendix~\ref{sec:allterms} we denote,
\begin{equation}\label{eq:alpha-in-delta-phi}
\alpha=\delta t^{\nu_1-1}e^t\equiv \delta t^{-(4\varphi+1)}e^t.
\end{equation}
For $\arg t=\pi/2$ we have $|\alpha|=|\delta|e^{2\pi\kappa}=1$. Putting $t=\imath|t|$ we rewrite $\alpha$ in terms
of $\kappa$,
$$
\alpha=\imath e^{\imath\omega},\qquad
\omega=|t|-4\kappa\ln(2|t|)+2\arg\Gamma(1+2\imath\kappa).
$$
Substituting the values for $\Theta$-parameters into
Equations~\eqref{eqs:asymptseries-y-1} and \eqref{eqs:asymptseries-z-1} we arrive at the following asymptotics:
\begin{align}
&y(t)\underset{t\to+\imath\infty}=\alpha\left(1+\frac{1}{t}\left((2\varphi+1)\alpha+
\left(\frac34(4\varphi+1)^2-\frac12\right)-\frac{2\varphi}{\alpha}\right)+\mathcal{O}\left(\frac1{t^2}\right)\right),
\label{eq:y-spec-imaginary+phi}\\
&z(t)\underset{t\to+\imath\infty}=\left(\varphi+\frac14\right)\left(-1+
\frac1t\left(\left(\varphi+\frac34\right)\alpha+\frac{\left(\varphi-\frac14\right)}\alpha\right)+
\mathcal{O}\left(\frac1{t^2}\right)\right).\label{eq:z-spec-imaginary+phi}
\end{align}
We also present these asymptotics in terms of $\{\kappa,\omega\}$-variables:
\begin{align}
&y(t)\underset{t\to+\imath\infty}=\imath e^{\imath\omega}
\left(1+\frac{\imath}{|t|}\left(\frac12+12\kappa^2+\sin\omega+4\kappa\cos\omega)\right)+
\mathcal{O}\left(\frac1{|t|^2}\right)\right),\label{eq:y-spec-imaginary+kappa}\\
&z(t)\underset{t\to+\imath\infty}=-\imath\kappa\left(1+\frac{2\kappa\sin\omega-\cos\omega}{|t|}+
\mathcal{O}\left(\frac1{|t|^2}\right)\right).\label{eq:z-spec-imaginary+kappa}
\end{align}
The last asymptotics reflects the fact that for pure imaginary $t$: $|y(t)|=1$ and $\Re z(t)=0$.

Asymptotics of $y(t)$ for $\kappa=0$ \eqref{eq:y-spec-imaginary+kappa} coincides with \eqref{eq:y1-imaginary+}.
We recall that in this case $z_1(t)\equiv0$, which is consistent with \eqref{eq:z-spec-imaginary+kappa}.

Numerical values for $a_1^1=-2$:
\begin{equation}
\varphi+\frac14=\imath\kappa=\imath0.22063560\ldots, \qquad
\de =-3.48631745\ldots+\imath1.96101774\ldots.
\label{predicted}
\end{equation}

Now consider asymptotics for the negative imaginary axis ($\arg t=-\frac \pi 2$). Asymptotics for $y(t)$
in this case can be calculated via the symmetry considered in Subsection~\ref{subsec:t-to-minus-t}, since our special
solution $y(t)$ does not depend on the sign of $\Theta_\infty=\pm1$. However, the corresponding function $z(t)$
depends on this sign.
Therefore, to get asymptotics in this case we again address Theorems~\ref{th:new1} and \ref{th:allmon}
and Appendix~\ref{sec:allterms}. To distinguish from the previous case we denote basic parameters
as $\varphi_-$ and $\delta_-$. Then Theorem~\ref{th:allmon} implies:
$$
\varphi_-=-\frac1{2\pi\imath}\ln (m^0_{11})=-\frac1{2\pi\imath}\ln(-\imath2a_1^1)=-\frac14-\imath\kappa=\bar\varphi.
$$
Here $\bar{(\cdot)}$ denotes complex conjugation of the parameter $(\cdot)$.
Parameter $\delta_-$ is given by Equation~\eqref{eq:delta} but with $\varphi\to\bar\varphi$, therefore
$\delta_-=\bar\delta$.

The conditions on parameters $\nu_1$ and $\nu_2$ in Theorems~\ref{th:new1} and \ref{th:new2}, respectively,
do not depend on the imaginary part of $\varphi$, so we conclude that asymptotics of functions $y(t)$ and
z(t) as $t\to-\imath\infty$ are given by r.-h.s. of Equations~\eqref{eq:y-spec-imaginary+phi} and
\eqref{eq:y-spec-imaginary+phi} with $\varphi\to\bar\varphi$ and $\alpha\to\alpha_-$, where $\alpha_-$ is defined by
Equation~\eqref{eq:alpha-in-delta-phi} with $\varphi\to\bar\varphi$ and $\delta\to\delta_-$. Turning to
$\{\kappa,\omega\}$-variables we have to put $t=-\imath|t|$, then $\alpha_-=\bar\alpha$ and asymptotics of
functions $y(t)$ and $z(t)$ as $t\to-\imath\infty$ are given by the complex conjugation of
Equations~\eqref{eq:y-spec-imaginary+kappa} and \eqref{eq:z-spec-imaginary+kappa}, respectively.
\section{Numerical Verification}\label{sec:numerics}

The purpose of this section is twofold: 1) To check the absence of any occasional mistakes in the formulae
presented in this paper, and 2) To visualize solutions of the fifth Painlev\'e equation and related functions.

The results of our numerical verifications are presented on figures. In the online version of the paper,
graphs of numerical solutions are given in red colors while graphs of their large-$t$ asymptotics are plotted
in green. In the early version of the paper the numerical calculations were done with the help of MATHEMATICA code.
In this version we redid these calculations with MAPLE code for the purpose of presenting additionally
the connection results for function $u$, which were absent before. Inclusion of $u$ makes verification more complete.
This function appears also in some applications and helps to calculate asymptotics of some interesting integrals
with functions $y$ and $z$. On the newly produced figures we presented the plots on intervals closer to the
origin, so that the reader can see when the functions achieve their asymptotic behavior.

In this section we use the following notation:
$$
t\equiv\pm\imath x
\quad
\mathrm{with}\quad
x>0\quad
\mathrm{iff}\quad
\arg t=\pm\frac\pi2.
$$
Assuming that argument of $t$ is fixed as above we denote
$$
\tilde y(x)\equiv y(\pm\imath x),\quad
\tilde z(x)\equiv z(\pm\imath x),\quad
\tilde u(x)\equiv u(\pm\imath x).
$$
In numerical calculations we choose the point $x_0$ in the neighborhood of $x=0$, where we take initial data,
for the numeric calculation of the solution under investigation. The solution is defined by taking some particular
values for parameters $\sigma$, $s^2$, and $r$,  defining asymptotic behavior of solutions as $t\to0$,
see Section~\ref{sec:zero}. Using asymptotics presented in Theorem~\ref{th1} we calculate the
corresponding initial data. Thus it is clear that the closer $x_0$ to the origin the better initial data correspond
to the chosen parameters. On the other hand the closer $x_0$ to the origin the more precision is required in the
calculations, which increases the time of calculations.
The reader can notice that choice of $x_0$ varies between
computations, which is done intentionally.  For some solutions to get a reliable result we have to choose the initial
point $x_0$ by a factor $10^{-6}$ closer to the origin, than for some others. So the choice of $x_0$ is an important
issue in the calculations. In all examples presented in this section parameter $x_0$ is chosen with some margin for error.
It means that making it 5--10 times larger wouldn't result in a visible change of the graphs. Similarly, making it smaller
also does not have a noticeable visual effect on the plots (while certainly changing slightly the numerical values).

Note that for the solution considered in Subsection~\ref{subsec:numeric-extra-terms} we used a different scheme of
calculations which allows for $x_0$ to be zero.

We conclude the introductory part of this section by giving some details on settings for MAPLE and Mathematica codes
we used.

In MAPLE code for most of the calculations the value of parameter {\it Digits} 10--14 is enough.  The word `enough'
in the previous sentence means that by making $x_0$ smaller and increasing the accuracy of calculations we cannot
notice any visual change in the plots.
For producing of the pictures we set this
parameter to 16. We used the standard {\it dsolve} procedure with parameters $abserr=relerr=e^{-12}$ and
$maxfun=-1$. For plotting we used procedure $plots[odeplot]$ with $numpoints=600$.
The increase of accuracy with the help of these parameters does not effect any how on visual quality
of the plots though may results in substantial increase of the time of calculations.

The original plots were produced with MATHEMATICA code.
In the process of numerically solving the ODE we have to use arbitrary-precision arithmetic. We solve
System~\eqref{eq:ids1}, \eqref{eq:ids2} via NDSolve command with parameters WorkingPrecision at 40 and MaxSteps set
to Infinity. As explained in documentation on NDSolve, the parameters PrecisionGoal and AccuracyGoal are set to half
of WorkingPrecision by default. Thus, in our case they are set to 20. We have to specify the terminal, or final, point
$x_f=1000$. For calculations with such accuracy in most cases after $x=50$ we were not able to distinguish the plots
of the numerical solutions and their large-$t$ asymptotics. So that we used markers on the curves to show that it is
actually two curves rather than one.

In presentation of numerical values we give only the first 9 digits after the decimal point. In fact, as follows from
above we have done our calculations with the better accuracy.
In case a numerical value of some parameter is given with a less amounts of digits it means that it is its exact value.
\subsection{Generic Case: Theorem~\ref{th:new1}, $\Im t>0$}

In this subsection we present two examples of solutions whose asymptotics on positive imaginary semi-axis can be
constructed with the help of Theorem~\ref{th:new1}.

For our first numerical run we fix the formal monodromies as follows:
$$
\0 = 0.7,\qquad
\1 =0.3,\qquad
\Theta_{\infty}= 0.4
$$
and take the parameters defining asymptotics as $t\to0$ as follows
$$
\sigma = 0.32,\qquad
s=0.3.
$$

We use the leading terms of asymptotics given in Theorem~\ref{th1} to calculate the initial data for the numerical
solutions at $x_0=10^{-4}$:
\begin{gather*}
\tilde y(x_0) = 0.994568108\ldots-\imath0.002959252\ldots,\quad
\tilde z(x_0)=-47.584547341\ldots+\imath26.986260774\ldots,\\
u(x_0)=r(-0.020313259\ldots - \imath0.014746818\ldots).
\end{gather*}

To calculate the parameters $\varphi$ and $\delta$ defining, according Theorems~\ref{th:new1} and \ref{th:new2},
the large-$t$ asymptotics of the functions $\tilde y(x)$ and $\tilde z(x)$ we have first to find the monodromy data
with the help of Theorem~\ref{th2}. According to Theorem~\ref{th:allmon} we need only the following data:
\begin{gather*}
m^1_{11}=12.437589470\ldots+\imath6.650948445\ldots,\quad
m^1_{12}=r\imath9.125827299\ldots,\\
m^0_{21}m^1_{12}=-61.163346517\ldots+\imath188.241424656\ldots.
\end{gather*}
Now, Theorem~\ref{th:allmon} implies:
\begin{gather*}
\varphi=0.278154042\ldots-\imath0.421199315\ldots,\quad
\delta=-4.554462477\ldots+\imath6.135670701\ldots,\\
\hat u=r(-0.319233835\ldots-\imath2.211728624\ldots).
\end{gather*}
Note that by Theorem~\ref{th:allmon} $\varphi$ (see Equation~\eqref{mr1}) is defined via logarithms.
Namely, $\Re\varphi$ is defined $\mod\mathbb Z$.
To fix it uniquely we have to calculate parameters $\nu_1$ and $\nu_2$ given in Theorems~\ref{th:new1} and \ref{th:new2}:
$$
\nu_1=0.287383828\ldots+\imath1.684797263\ldots,\quad
\nu_2=1.712616171\ldots-\imath1.684797263\ldots
$$
and check whether one of the numbers $\Re\nu_k$, $k=1,2$, fits the interval $(-1/2,1)$. In our case $\Re\nu_1$ fits
the interval, while $\Re\nu_2$ does not. If neither number $\Re\nu_k$ is within the interval, we have, using
the ambiguity $\mod\mathbb Z$ in $\varphi$, to adjust one of them to fit the interval $(-1/2,1)$. If such adjustment
is not possible for a given value of $\varphi$, then the reader is addressed to part II of this work.

In our case we are in conditions of Theorem~\ref{th:new1}. So that the large-$t$ asymptotics have to be
calculated with the help of this Theorem.

The results of the calculations (with $r=1$) are presented on Figs.~\ref{C1ReY}-\ref{C1ImU}. The range of calculations
on all figures was from $x=5$ to $300$. Since qualitative  large-$t$ behavior of functions $y$ and $z$ is obvious from
Figures~\ref{C1ReY}-\ref{C1ImZ}, we bounded the range of the plot by $x=200$.
Function $u$ shows a more interesting behavior so we plotted it on the whole range it was calculated.
\begin{figure}
\begin{center}
\includegraphics[height=3.0in,width=6in]{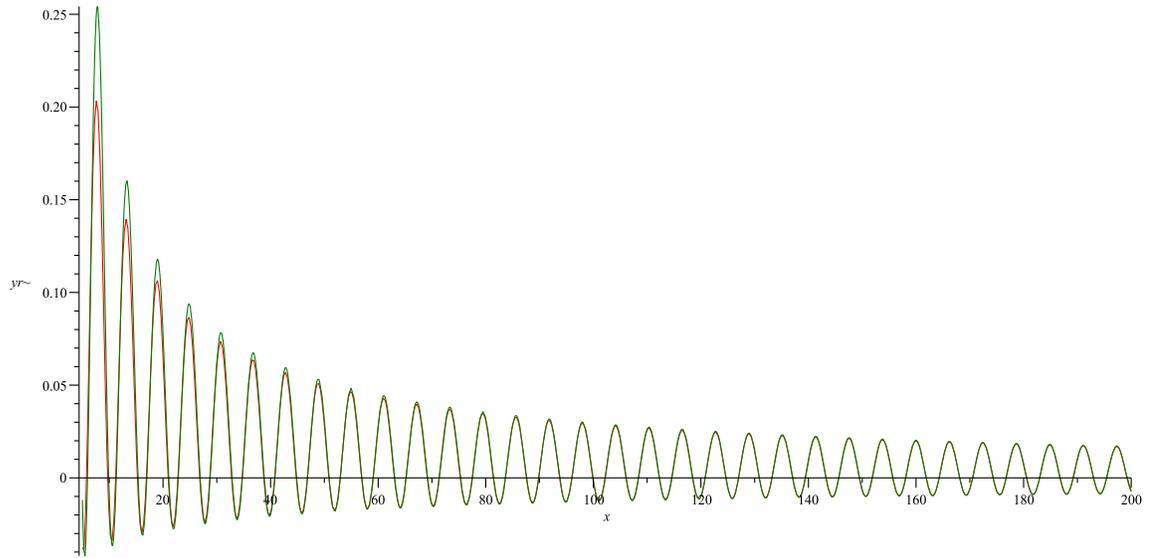}
\caption{Real part of $y$: large-$t$ asymptotics and numerical solution}\label{C1ReY}
\end{center}
\end{figure}
\begin{figure}
\begin{center}
\includegraphics[height=3.0in,width=5in]{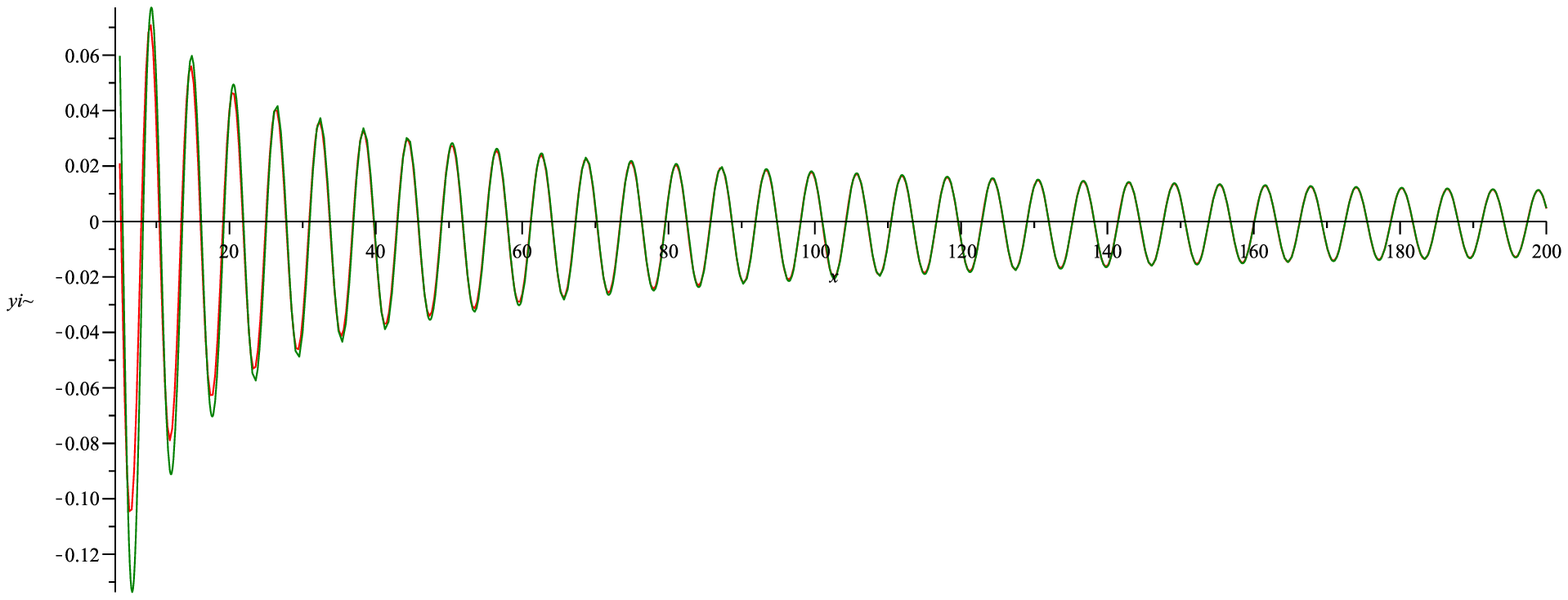}
\caption{Imaginary part of $y$: large-$t$ asymptotics and numerical solution}\label{C1ImY}
\end{center}
\end{figure}
\begin{figure}
\begin{center}
\includegraphics[height=3.0in,width=5in]{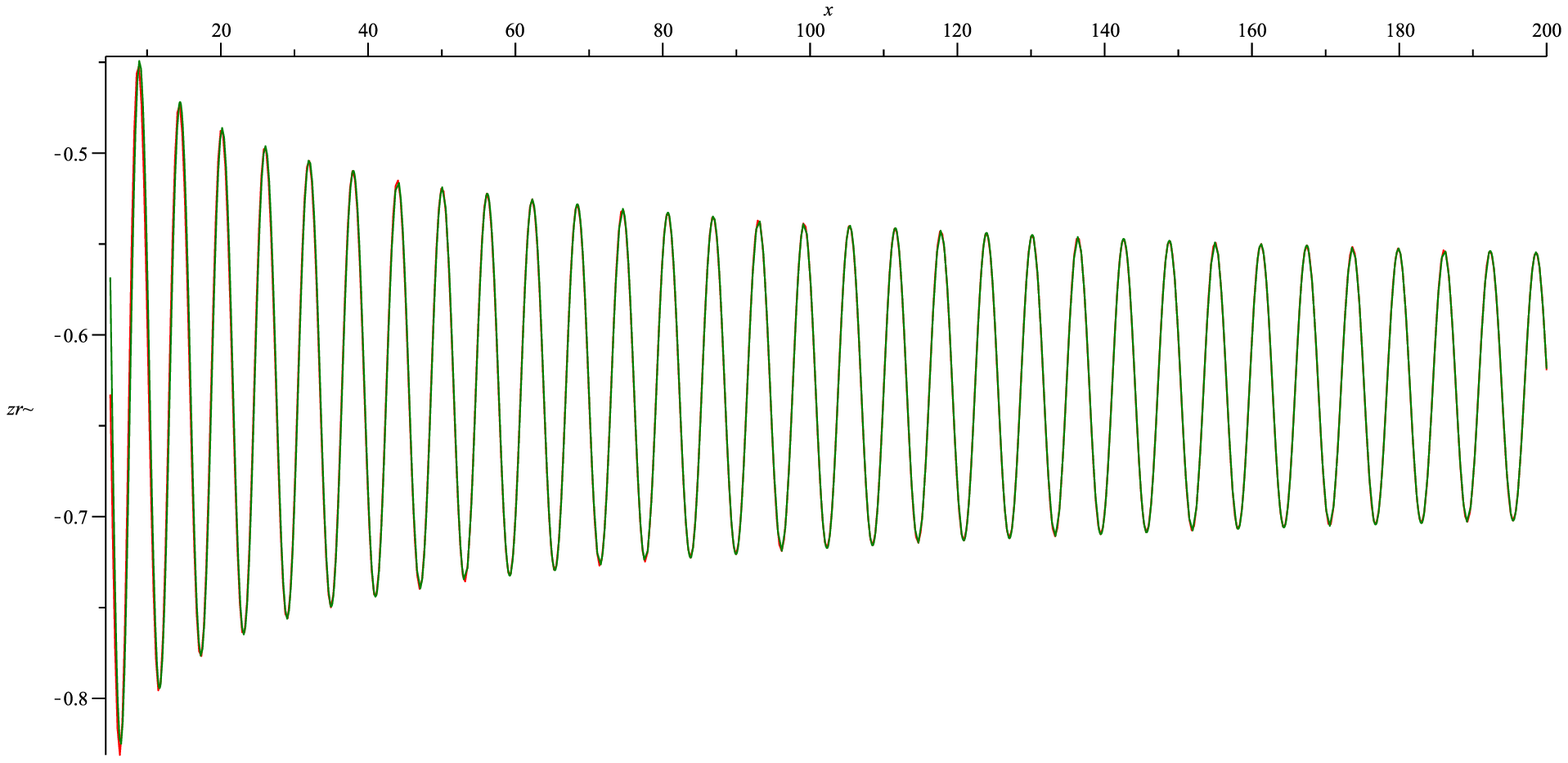}
\caption{Real part of $z$: large-$t$ asymptotics and numerical solution}\label{C1ReZ}
\end{center}
\end{figure}
\begin{figure}
\begin{center}
\includegraphics[height=3.0in,width=5in]{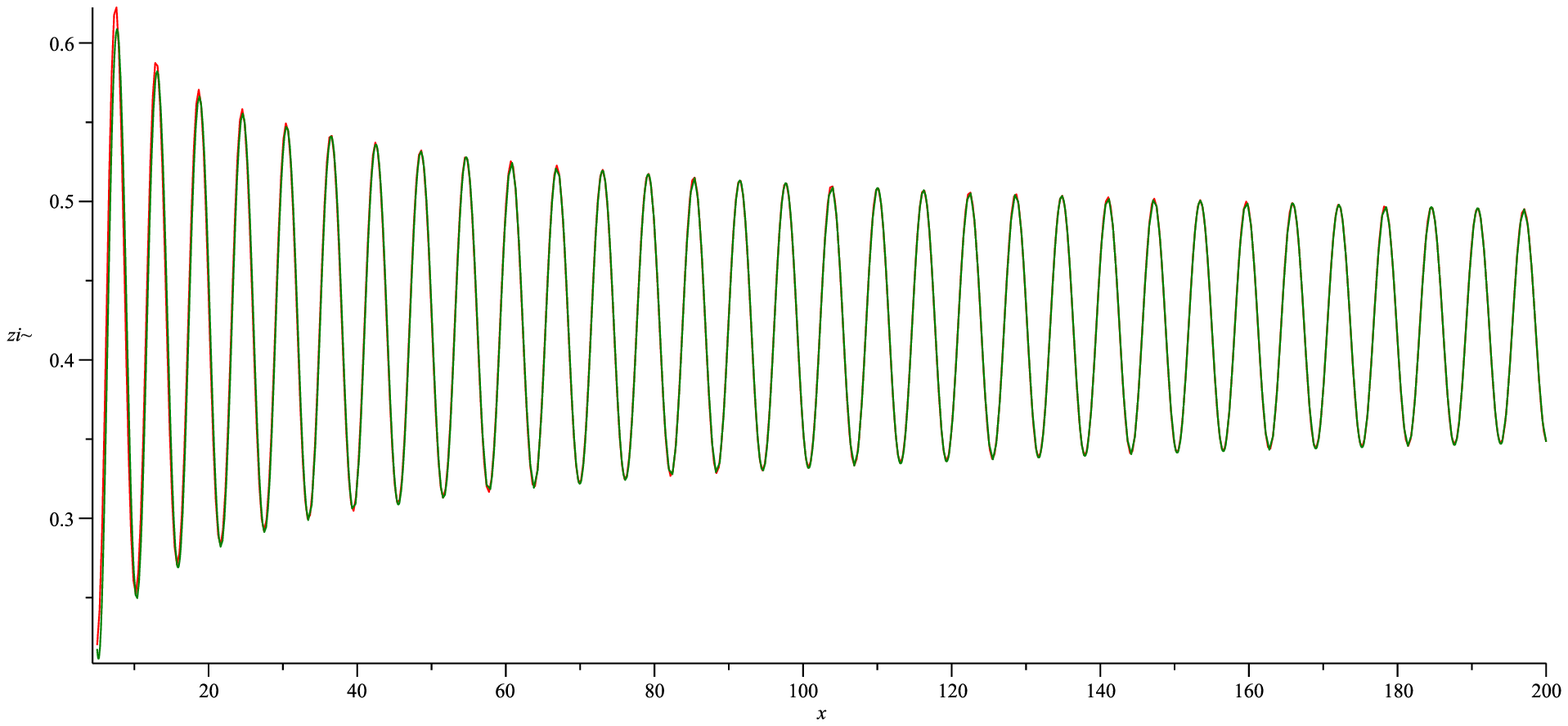}
\caption{Imaginary part of $z$: large-$t$ asymptotics and numerical solution}\label{C1ImZ}
\end{center}
\end{figure}
\begin{figure}
\begin{center}
\includegraphics[height=3.0in,width=5in]{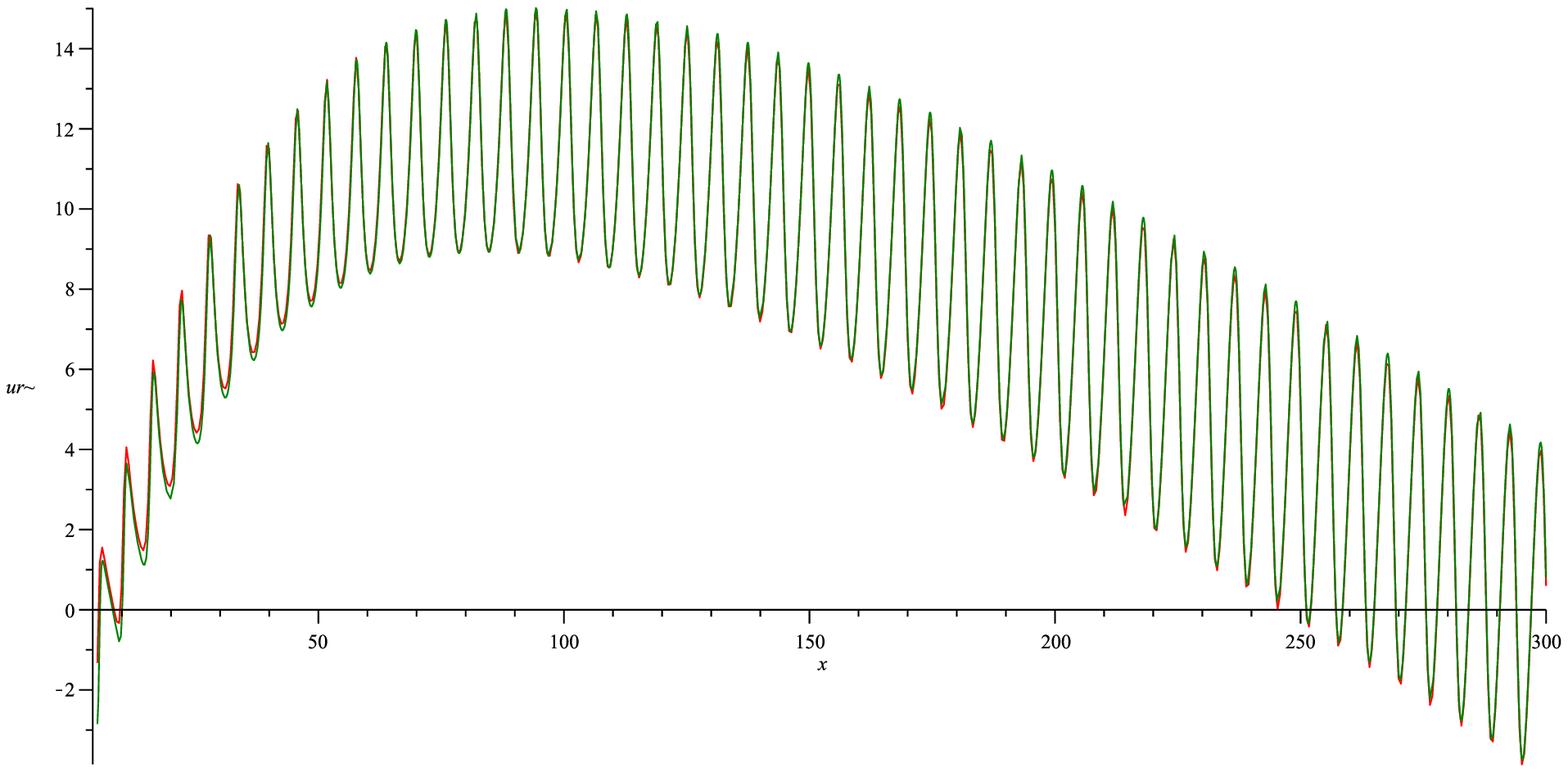}
\caption{Real part of $u$: large-$t$ asymptotics and numerical solution}\label{C1ReU}
\end{center}
\end{figure}
\begin{figure}
\begin{center}
\includegraphics[height=3.0in,width=5in]{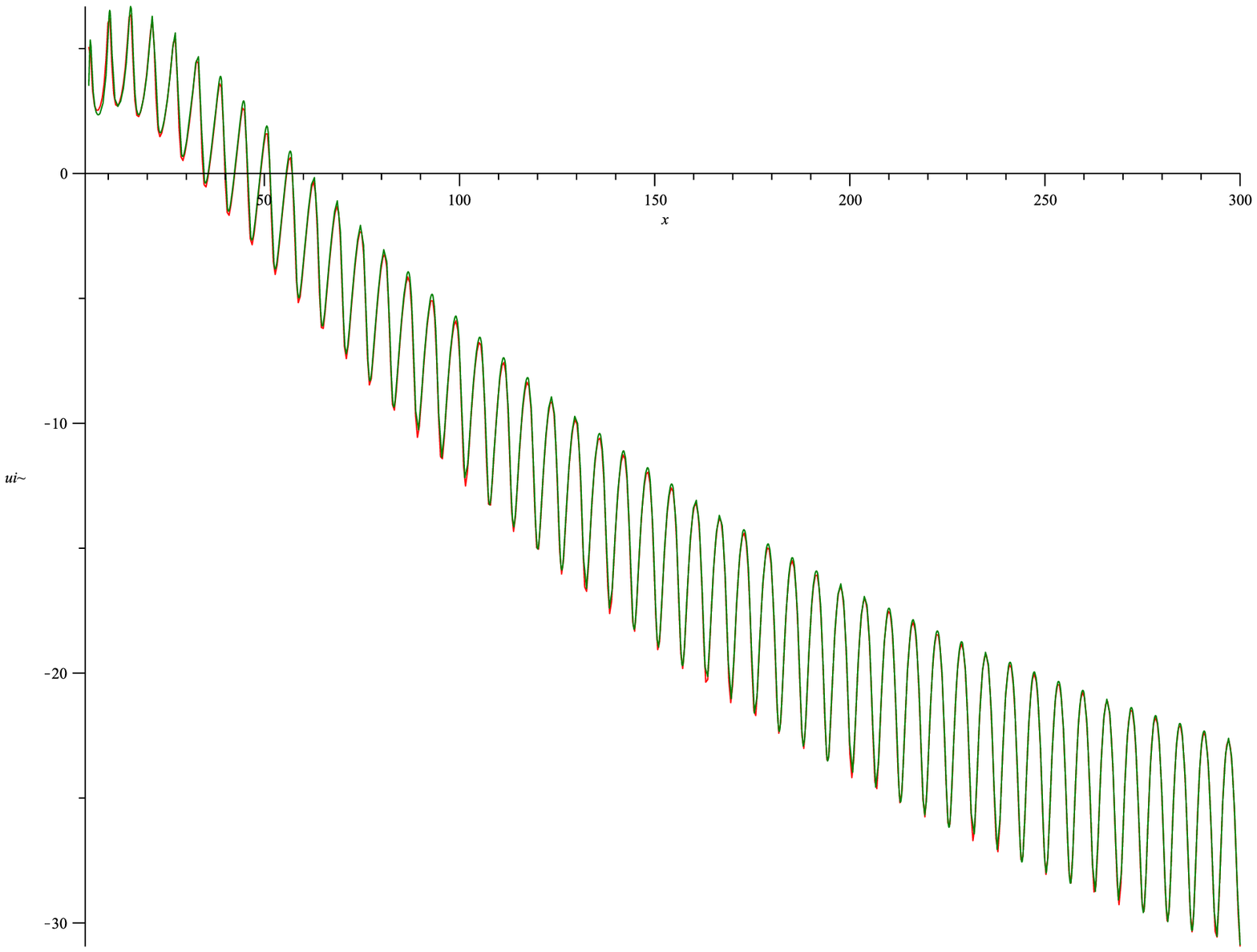}
\caption{Imaginary part of $u$: large-$t$ asymptotics and numerical solution}\label{C1ImU}
\end{center}
\end{figure}

Now we consider the second case. Since the scheme of calculations is exactly the same we present only numerical values
of the parameters and the resulting graphs.
\begin{gather*}
\Theta_0=0.3,\qquad
\Theta_1=0.4,\qquad
\Theta_\infty=-0.8\\
\sigma=0.1,\qquad
s=3.5,\qquad
x_0=10^{-5}.
\end{gather*}
Initial values for the numerical solution:
\begin{gather*}
y(x_0)=0.266805303\ldots+\imath0.358746453\ldots,\quad
z(x_0)=-0.330124183\ldots-\imath0.161506133\ldots\\
u(x_0)=r(12944.541090242\ldots+\imath18240.855333418\ldots).
\end{gather*}
The monodromy data:
\begin{gather*}
m^1_{11}= -3.948973870\ldots+\imath1.218498193\ldots,\quad
m^1_{12}= r\imath8.666566605\ldots,\\
m^1_{12}m^0_{21}=13.008291429\ldots-\imath9.451076940\ldots.
\end{gather*}
The parameters of asymptotics as $t\to+\infty$:
\begin{gather*}
\varphi=0.052366192\ldots-\imath0.225829516\ldots,\quad
\delta=0.221777960\ldots+\imath1.344890970\ldots,\\
\hat u=r(-1.459177245\ldots-\imath0.519203287\ldots).
\end{gather*}
The parameters $\nu_k$ are as follows:
$$
\nu_1= -0.009464771\ldots+\imath0.903318064\ldots,\quad
\nu_2=2.009464771\ldots-\imath0.903318064\ldots,
$$
which implies that we have to use again Theorem~\ref{th:new1}.
The results of calculations ($r=1$) are presented on Figs.~\ref{C2ReY}-\ref{C2ImU}. Here we plotted the solutions on
a small segment $[2,30]$, to show how the large-$t$ asymptotics approximate the solution at finite interval.
The settings indicated in preamble to this section allows one to plot the solutions far beyond $x=300$.
On Figures~~\ref{C2ReU}-\ref{C2ImU} we see that function $u$ has poles approaching the positive imaginary
semi-axis as $t\to+\imath\infty$.
\begin{figure}
\begin{center}
\includegraphics[height=3.0in,width=6in]{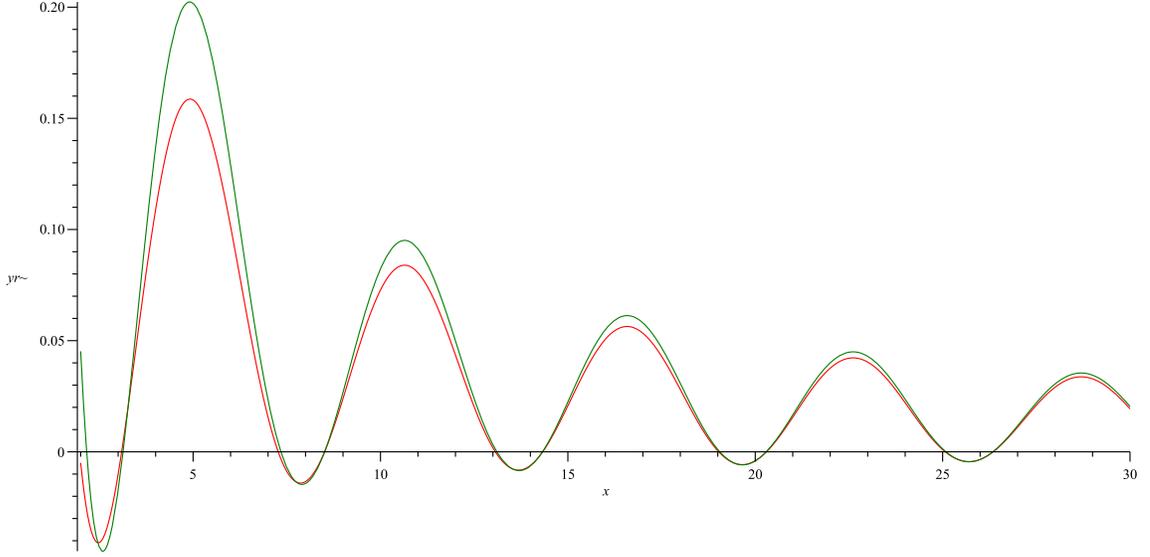}
\caption{Real part of $y$: large-$t$ asymptotics and numerical solution}\label{C2ReY}
\end{center}
\end{figure}
\begin{figure}
\begin{center}
\includegraphics[height=3.0in,width=5in]{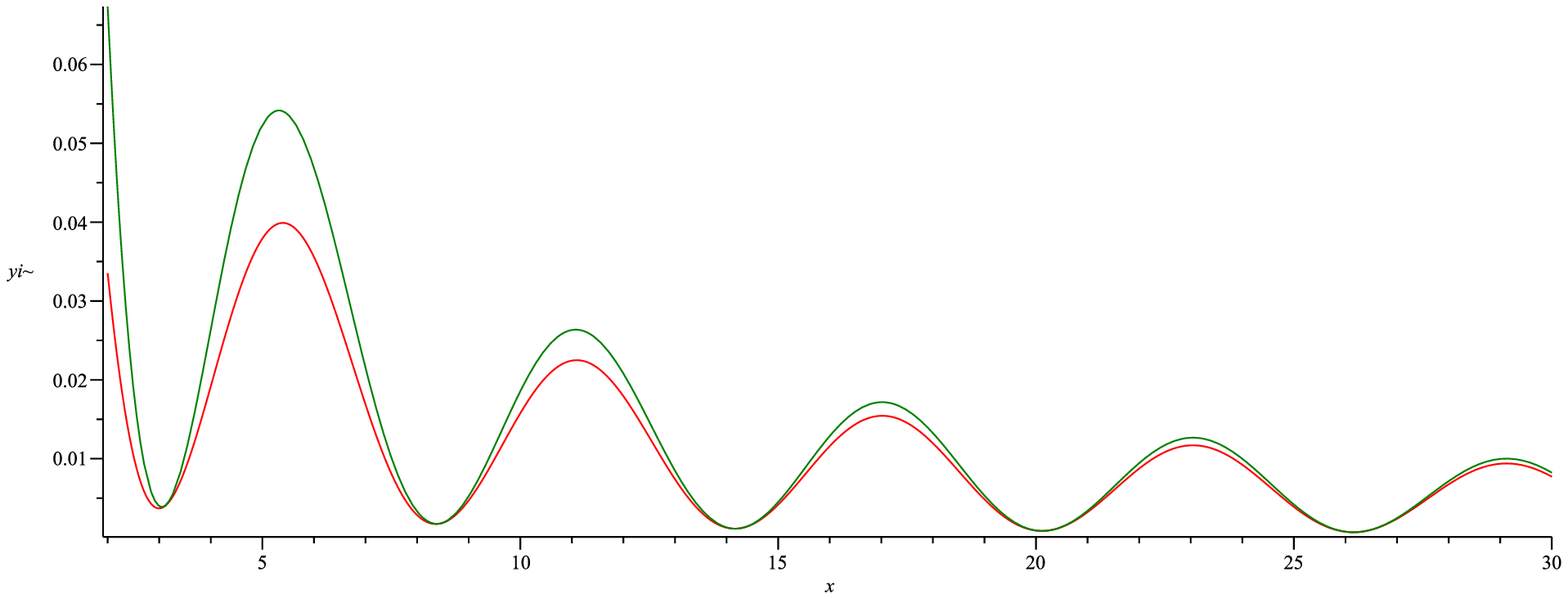}
\caption{Imaginary part of $y$: large-$t$ asymptotics and numerical solution}\label{C2ImY}
\end{center}
\end{figure}
\begin{figure}
\begin{center}
\includegraphics[height=3.0in,width=5in]{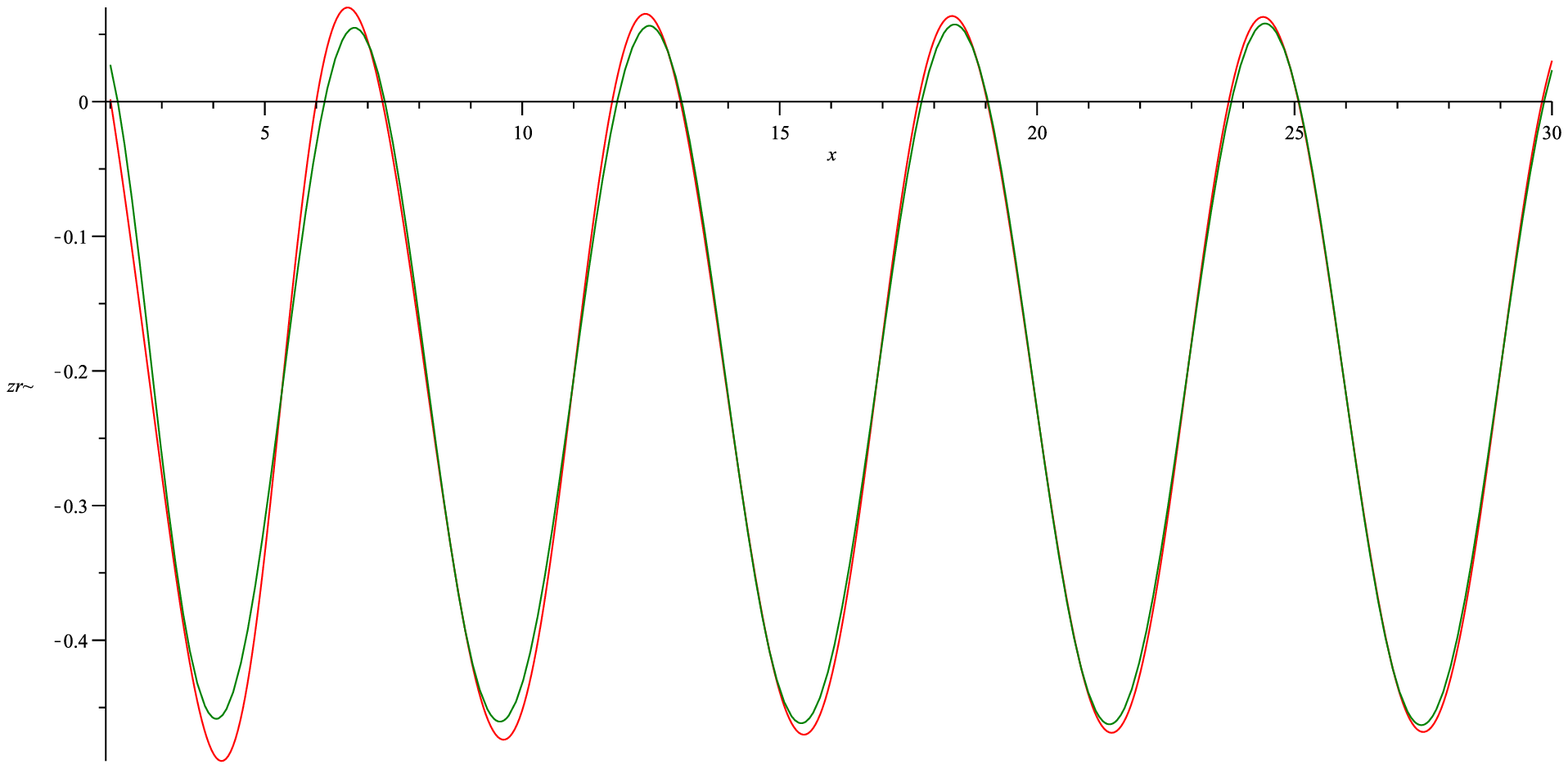}
\caption{Real part of $z$: large-$t$ asymptotics and numerical solution}\label{C2ReZ}
\end{center}
\end{figure}
\begin{figure}
\begin{center}
\includegraphics[height=3.0in,width=5in]{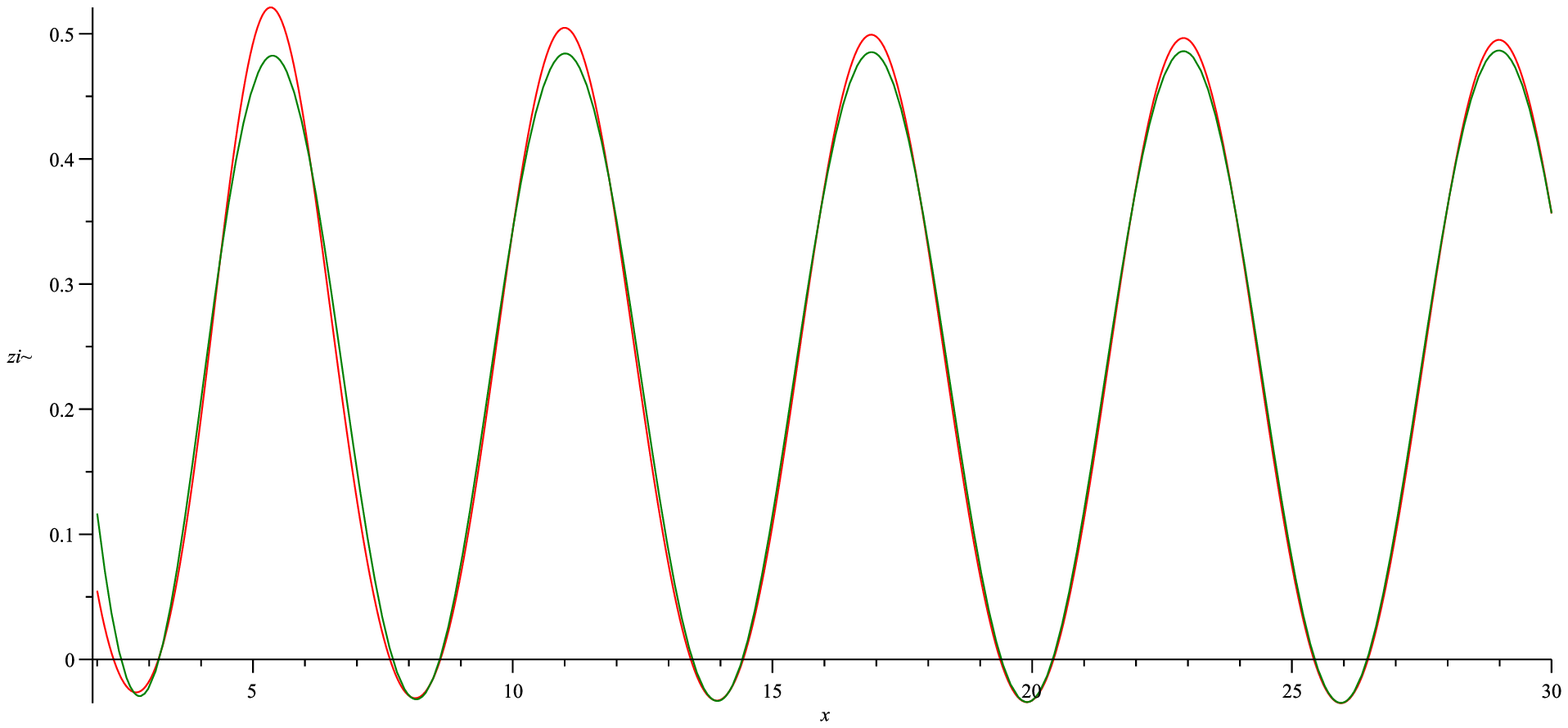}
\caption{Imaginary part of $z$: large-$t$ asymptotics and numerical solution}\label{C2ImZ}
\end{center}
\end{figure}
\begin{figure}
\begin{center}
\includegraphics[height=3.0in,width=5in]{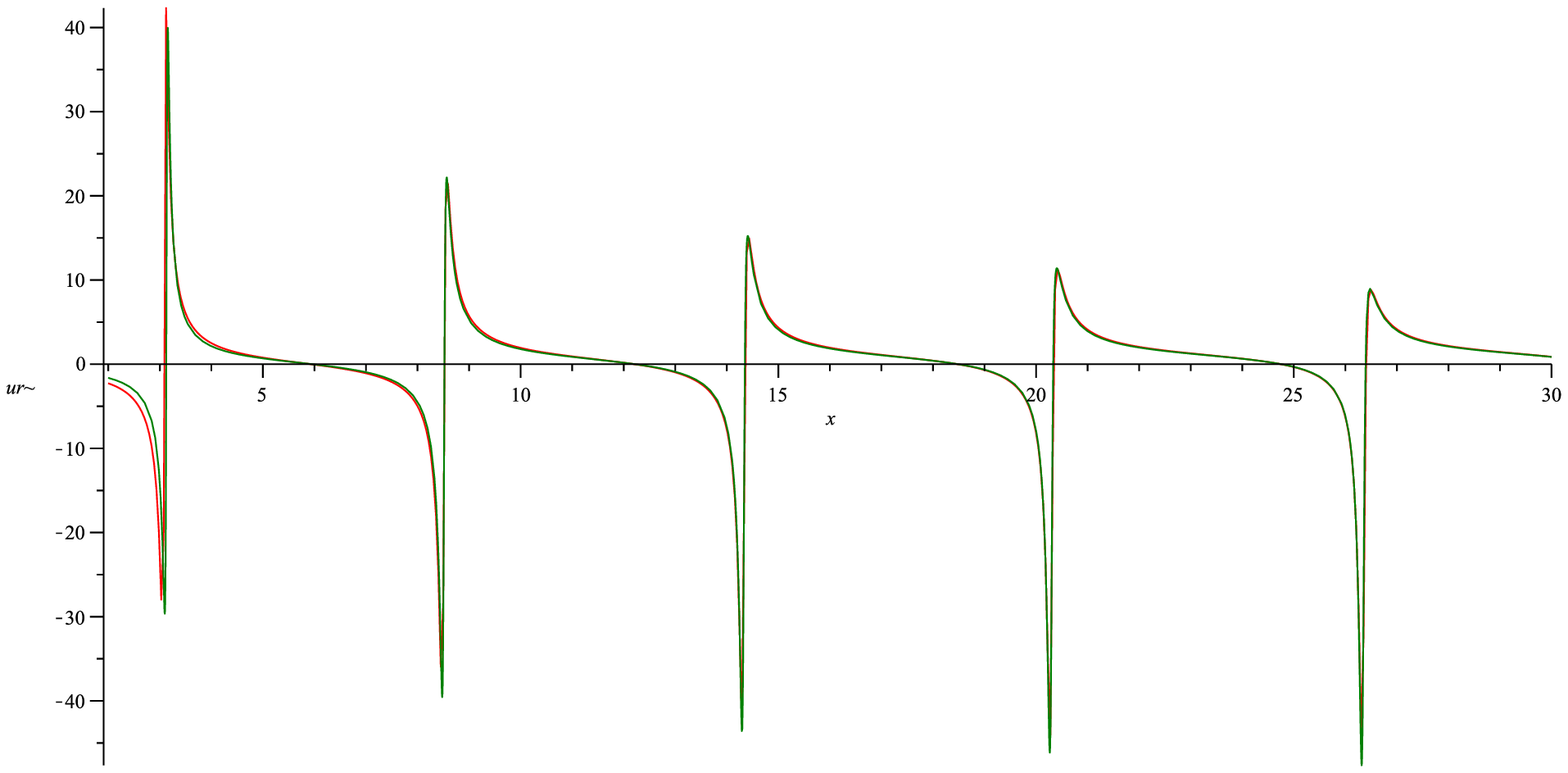}
\caption{Real part of $u$: large-$t$ asymptotics and numerical solution}\label{C2ReU}
\end{center}
\end{figure}
\begin{figure}
\begin{center}
\includegraphics[height=3.0in,width=5in]{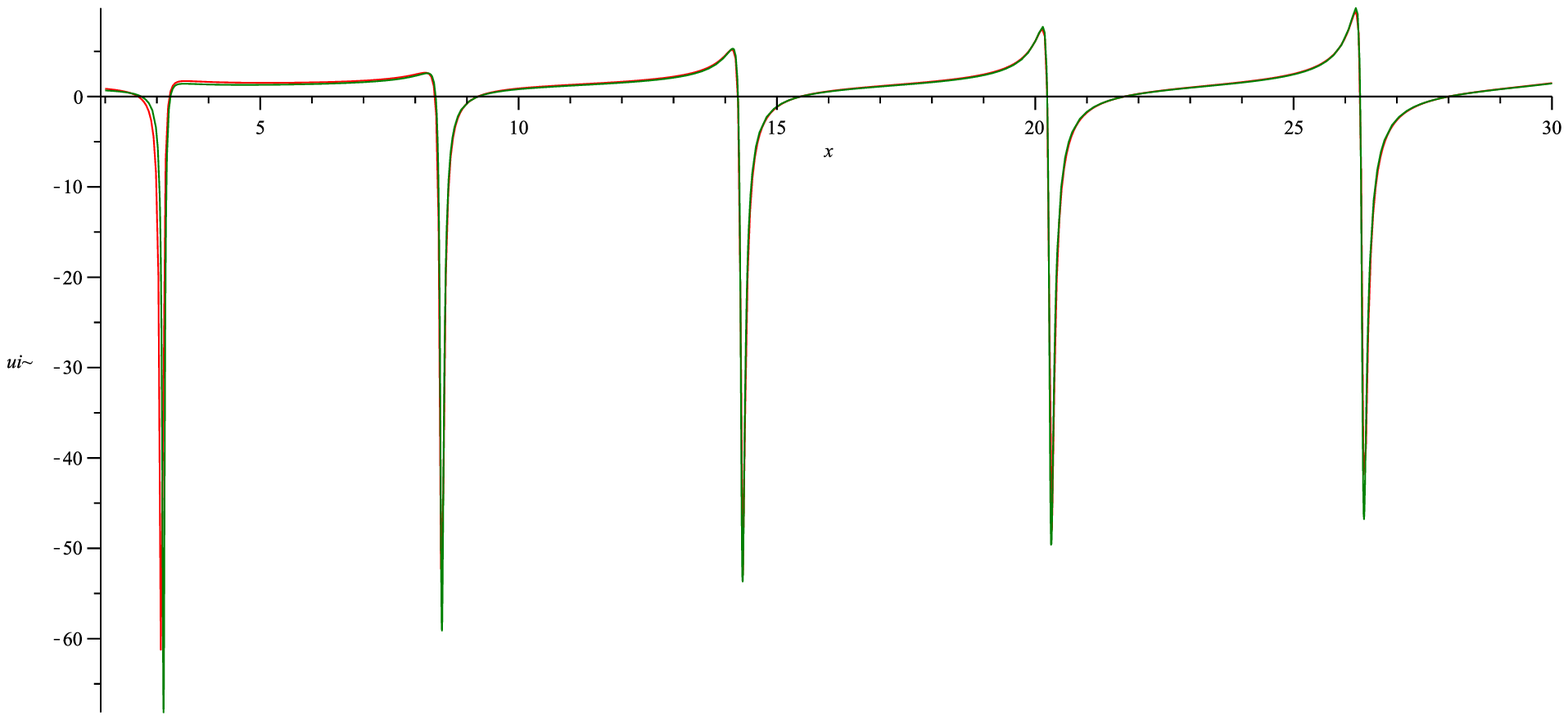}
\caption{Imaginary part of $u$: large-$t$ asymptotics and numerical solution}\label{C2ImU}
\end{center}
\end{figure}
\subsection{Generic case: $\Re\Theta_0>1$, $\Im t<0$}\label{subsec:ReTh0ge1}

In this subsection we check our asymptotic results for the case when the real part of one of the formal monodromies is
greater than $1$. In most cases this situation can be served (at least formally) with the help of the symmetry groups
acting in the space of solutions of IDS~\eqref{eq:ids1}--\eqref{eq:ids3}.
Therefore, it is enough to restrict real parts of formal monodromies within the segment $[0,1]$.
Such type of restrictions very often are imposed in studies of the Painlev\'e equations. However, application
of the symmetries to asymptotics is often related with cumbersome calculations.
Moreover, sometimes in these calculations it is not enough to deal only with the leading term of the asymptotics,
because the corresponding terms may cancel, so that one has to keep a few minor terms in asymptotic expansion to get
the correct result. At the same time none of our results includes this limitation on the real parts of the formal
monodromies. Therefore, it is reasonable to demonstrate the validity of our asymptotic results in the situation
beyond the limitations on real parts of the formal monodromies.

The set of formal monodromies for this subsection is as follows:
$$
\Theta_0 = 1.65,\qquad
\Theta_1 = 0.28,\;\;\mathrm{and}\;\;
\Theta_\infty= 0.37.
$$
The choice of the parameters defining asymptotics as $t\to-0\imath$
is the same as in the first example of the previous subsection:
$$
\sigma= 0.32,\qquad s = 0.3.
$$
It is an experimental fact that the construction of the numerical solution with the scheme explained in preamble
of this section requires to take the initial point closer to the origin, namely, $x_0=10^{-10}$. Note that in this
case variable $x>0$ is defined as $t=-\imath x$.

The initial values of the numerical solution:
\begin{gather*}
\tilde y(x_0) =0.999993054\ldots+\imath0.000003817\ldots,\
\tilde z(x_0)=-89006.808677592\ldots-\imath48934.551891324\ldots,\\
\tilde u(x_0)=r(-0.000166764279\ldots+\imath0.000109542161\ldots).
\end{gather*}
The monodromy data of the solution reads:
\begin{gather*}
m^0_{11}=2.060665876\ldots+\imath20.909643615\ldots,\quad
m^1_{12}=-r\imath15.422245661\ldots,\\
m^0_{21}m^1_{12}=-174.927577207\ldots+\imath404.233780229.
\end{gather*}
The parameters defining asymptotics as $t\to-\imath\infty$ are as follows:
\begin{gather}
\varphi=-0.234365609\ldots+\imath0.484633675\ldots,\
\delta=40.943218924\ldots+\imath12.639857745\ldots;\nonumber\\
\hat u=r(-4.935514833\ldots+\imath1.264358256\ldots),\\
\nu_1=2.307462436\ldots-\imath1.938534700\ldots,\quad
\nu_2=-0.307462436\ldots+\imath1.938534700\ldots.\label{eq:nu2theta0ge1}
\end{gather}
As follows from Equation~\eqref{eq:nu2theta0ge1} we are within the conditions of Theorem~\ref{th:new2}.
Note that parameter $\varphi$ here is calculated via Equation~\eqref{mr2}.
The corresponding asymptotic and numerical solutions are plotted  on the Figs.~\ref{C3ReY}-\ref{C3ImU}.

The accuracy of calculations (see the preamble to this section) allows one to build plot of the numerical
solution
which visually coincides with its large-$t$ asymptotics far beyond $x=300$. Specifically, we present
plots on segment $[1, 50]$, where the reader still can see difference between the numerical solution and the asymptotics.
\begin{figure}
\begin{center}
\includegraphics[height=3.0in,width=6in]{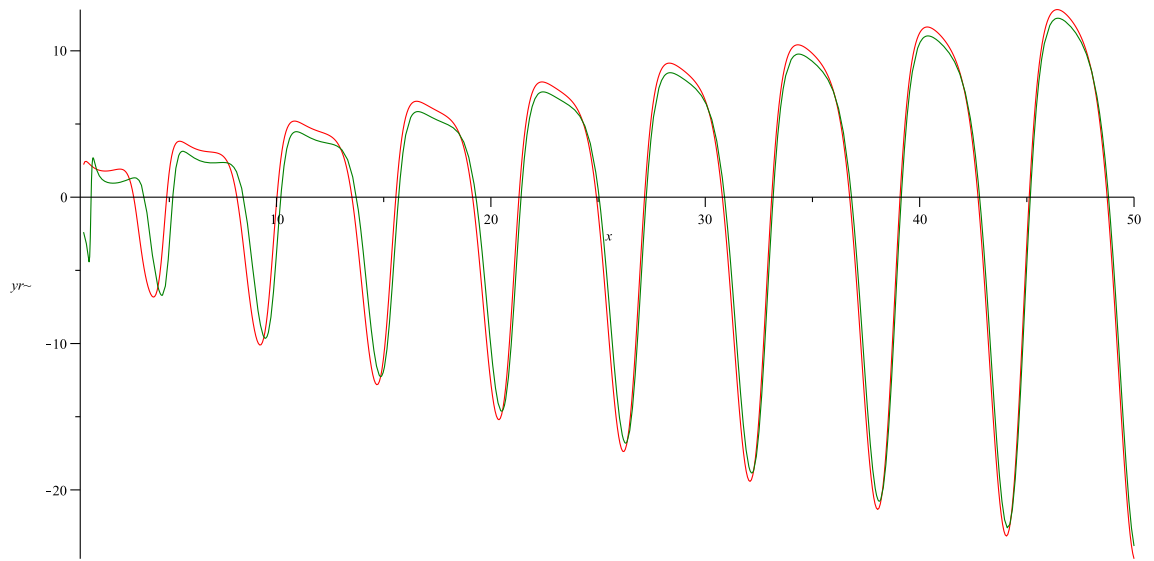}
\caption{Real part of $y$: large-$t$ asymptotics and numerical solution}\label{C3ReY}
\end{center}
\end{figure}
\begin{figure}
\begin{center}
\includegraphics[height=3.0in,width=5in]{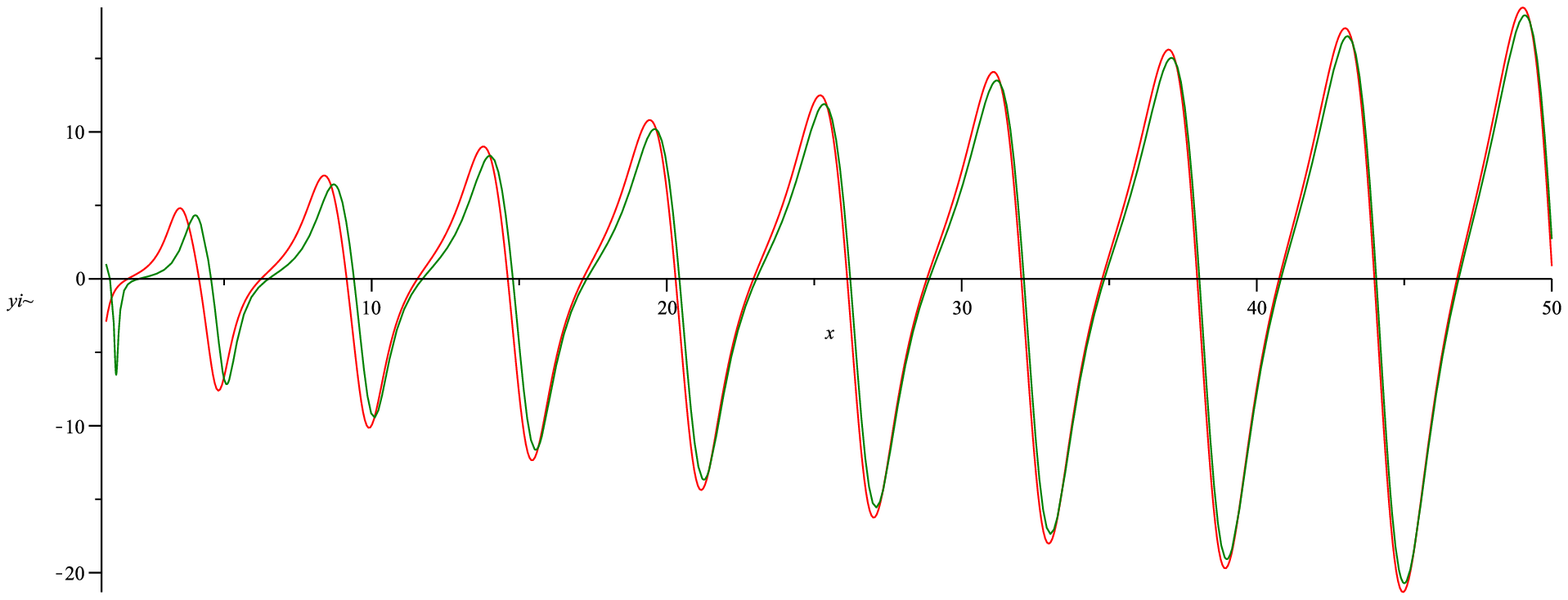}
\caption{Imaginary part of $y$: large-$t$ asymptotics and numerical solution}\label{C3ImY}
\end{center}
\end{figure}
\begin{figure}
\begin{center}
\includegraphics[height=3.0in,width=5in]{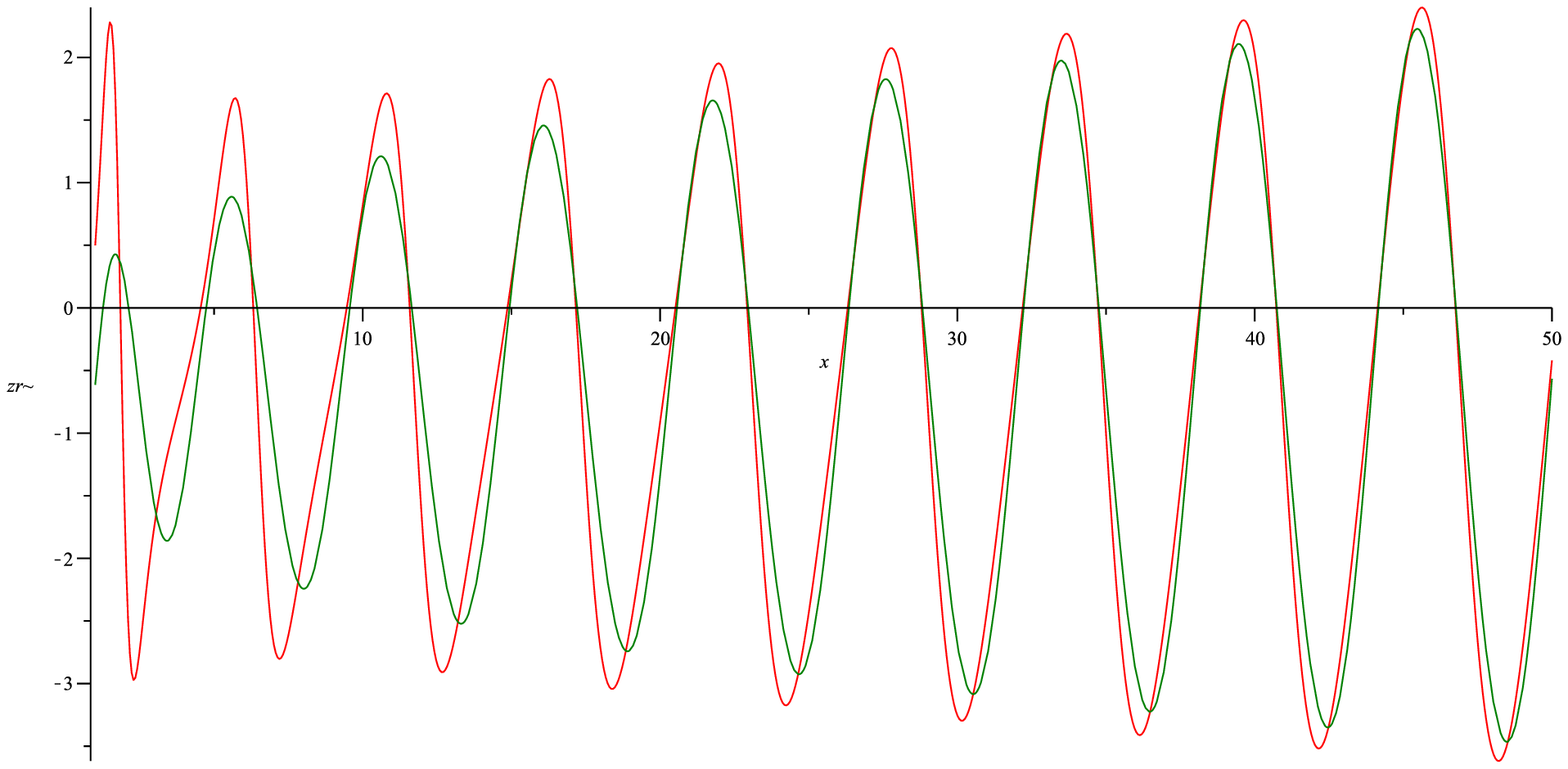}
\caption{Real part of $z$: large-$t$ asymptotics and numerical solution}\label{C3ReZ}
\end{center}
\end{figure}
\begin{figure}
\begin{center}
\includegraphics[height=3.0in,width=5in]{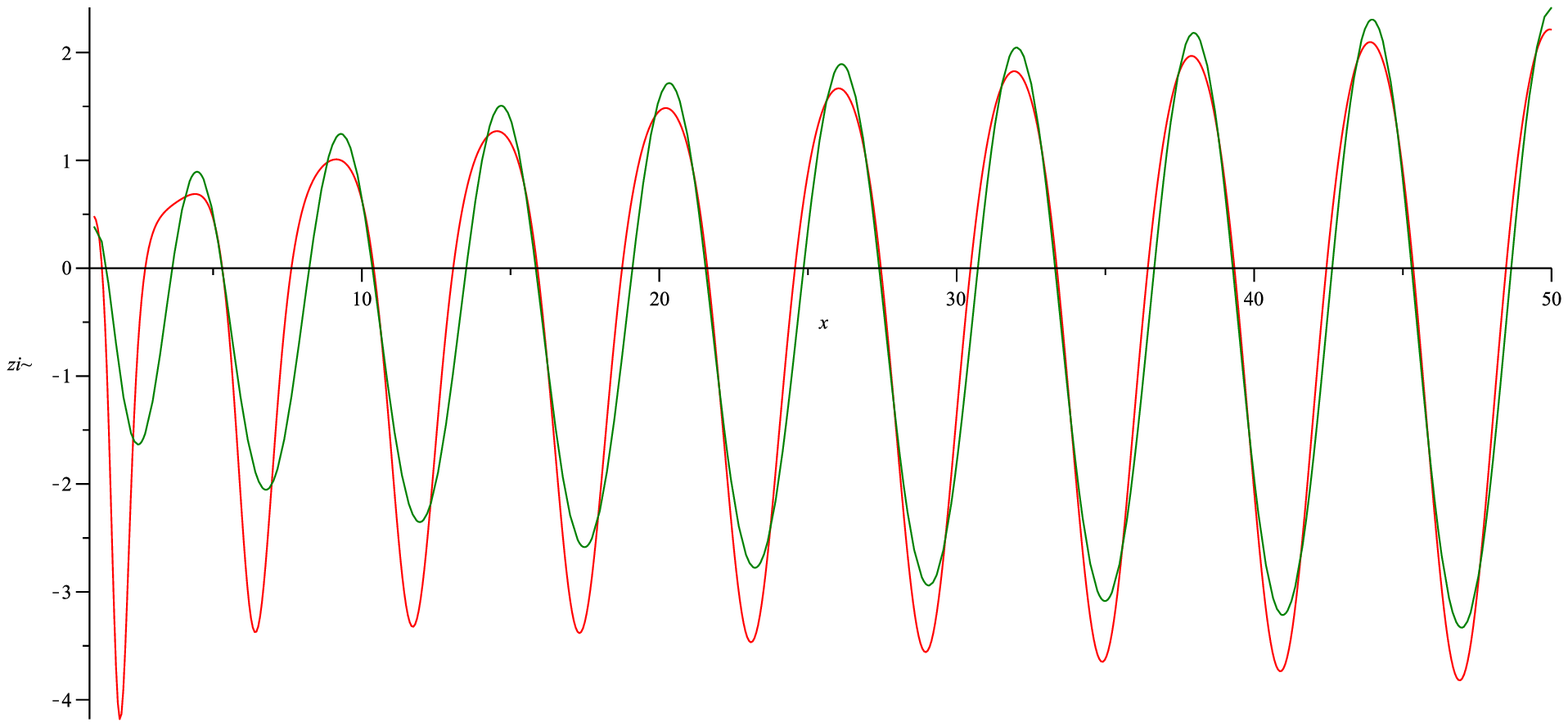}
\caption{Imaginary part of $z$: large-$t$ asymptotics and numerical solution}\label{C3ImZ}
\end{center}
\end{figure}
\begin{figure}
\begin{center}
\includegraphics[height=3.0in,width=5in]{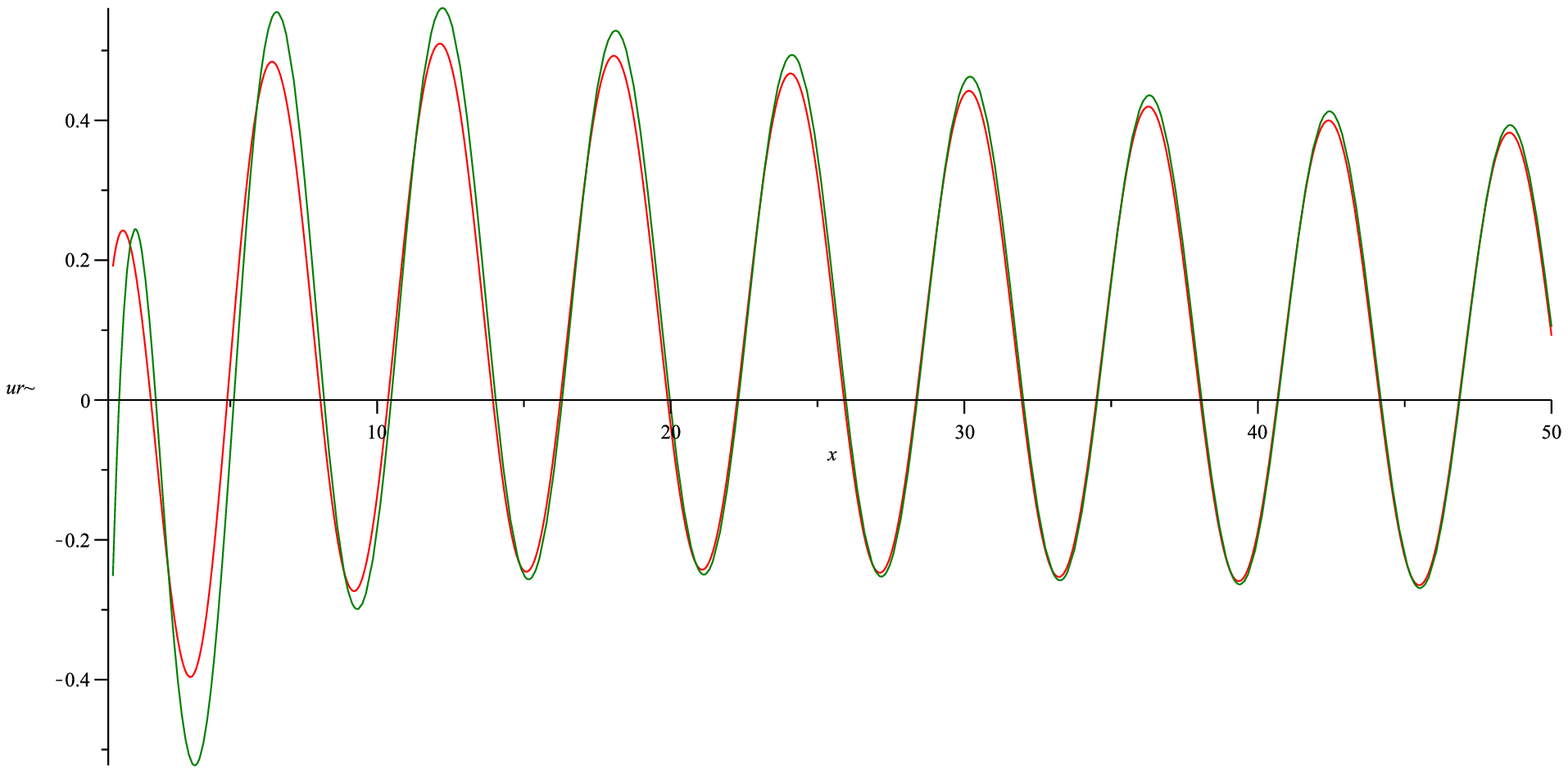}
\caption{Real part of $u$: large-$t$ asymptotics and numerical solution}\label{C3ReU}
\end{center}
\end{figure}
\begin{figure}
\begin{center}
\includegraphics[height=3.0in,width=5in]{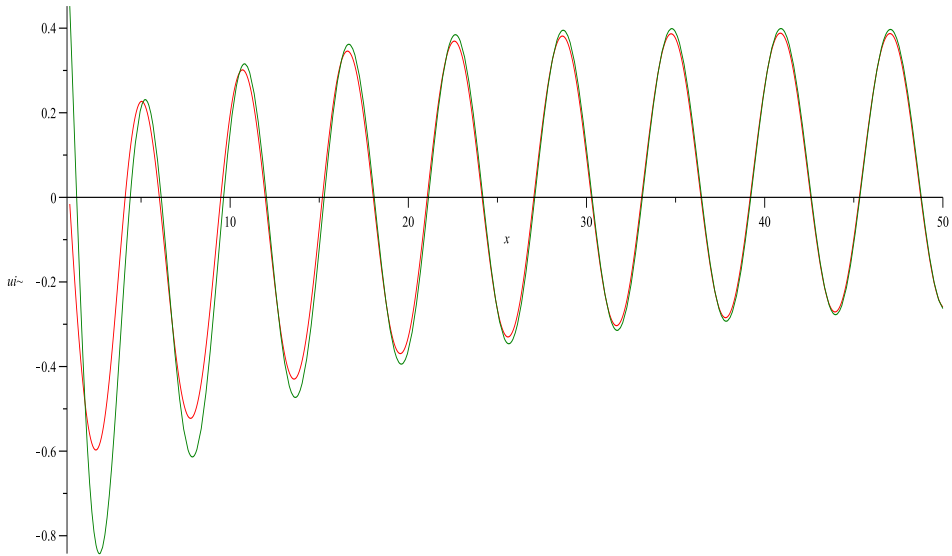}
\caption{Imaginary part of $u$: large-$t$ asymptotics and numerical solution}\label{C3ImU}
\end{center}
\end{figure}
\subsection{Numerical Illustration to Section~\ref{sec:mccoy}}\label{subsec:numeric-2-MT}
In Section~\ref{sec:mccoy} we compared our formulae with those obtained in \cite{MT1,MT2}. Here we consider a particular
example with both non vanishing Stokes multipliers.

The formal monodromies in this subsection are as follows:
$$
\Theta_0 = \Theta_1 = 0.73, \quad \Theta_\infty=0.
$$
Here we choose asymptotic parameters in a different way comparing to the previous subsections, namely,
\begin{equation}\label{eqs:sigma-phi}
\sigma=0.4,\qquad
\varphi=-0.15.
\end{equation}
The first parameter defines branching of the solution at $t=0$ the second one - branching of the solution at the point at
infinity. As in Subsection~\ref{subsec:ReTh0ge1} consider the negative imaginary semi-axis, $t=-x\imath$
with $x>0$. Theorem~\ref{th:allmon} implies that the parameter $\varphi$ determines
$$
m^0_{11}=0.587785252\ldots+\imath0.809016994\ldots.
$$
In the case $\Im t>0$ with the help of the same theorem one calculates monodromy parameter $m^1_{11}$.
Next, Theorem~\ref{th2} shows that the parameter $s^2$ defining the solution of IDS~\eqref{eq:ids1}--\eqref{eq:ids2}
as $t\to0$ can be determined as a solution of a quadratic equation, which means that in the generic situation
(including our case) there are two solutions with parameters~\eqref{eqs:sigma-phi}. We choose one of these solutions
$$
s^2=-0.164128459\ldots+\imath0.856228483\ldots,
s=0.594848223\ldots+\imath0.719703320\ldots.
$$

We started with $t_0=-10^{-8}\imath=-x_0 \imath$ resulting in initial conditions
\begin{gather*}
\tilde y(x_0)=1.000106557\ldots+\imath0.000227040\ldots,\;\;
\tilde z(x_0)=897.545278538\ldots-\imath1912.844754835\ldots,\\
\tilde u(x_0)=r(-0.999946706151\ldots+\imath0.000113501986\ldots).
\end{gather*}
Parameters $\sigma, s^2$, and $r$ with the help of Theorem~\ref{th2} allow one to find all the monodromy data:
\begin{gather}
m^1_{12}= r(-1.869666176\ldots-\imath0.261514312\ldots),\quad
m^0_{21}m^1_{12}=3.427261874\ldots+\imath0.977888928\ldots,\nonumber\\
s_1=-\frac{\imath}{r}1.175570504\ldots,\qquad
s_2=-r\imath1.175570504\ldots.\label{eqs:stokes103direct}
\end{gather}
Thus, actually, both Stokes multipliers do not vanish.

Now, using Theorem~\ref{th:allmon}, we find
$$
\delta=-5.237640067\ldots-\imath1.494437957\ldots,\qquad
\hat u= r(0.323288043\ldots-\imath2.311310283\ldots).
$$
Finally, we calculate parameters $\nu_1$ and $\nu_2$ (see Theorems~\ref{th:new1} and \ref{th:new2}):
$$
\nu_1=1-4\varphi=1.60,\qquad
\nu_2= 1+4\varphi=0.4.
$$
Thus we see that we are within the applicability of Theorem~\ref{th:new2}. The results of comparison of the
numerical and large-$t$ asymptotics for functions: $\tilde y(x)$, $\tilde z(x)$, and $\tilde u(x)$ with $r=1$
are presented on
Figs.~\ref{C4ReY}--\ref{C4ImU}. On these figures we compare numerical solution with its large-$t$ asymptotics
on the segment $[1,100]$. As usual we ensure that the accuracy in initial data for numerical solution
is such that its plot visually coincides with the plot of asymptotics on the distances far beyond $x=300$.
\begin{figure}
\begin{center}
\includegraphics[height=3.0in,width=6in]{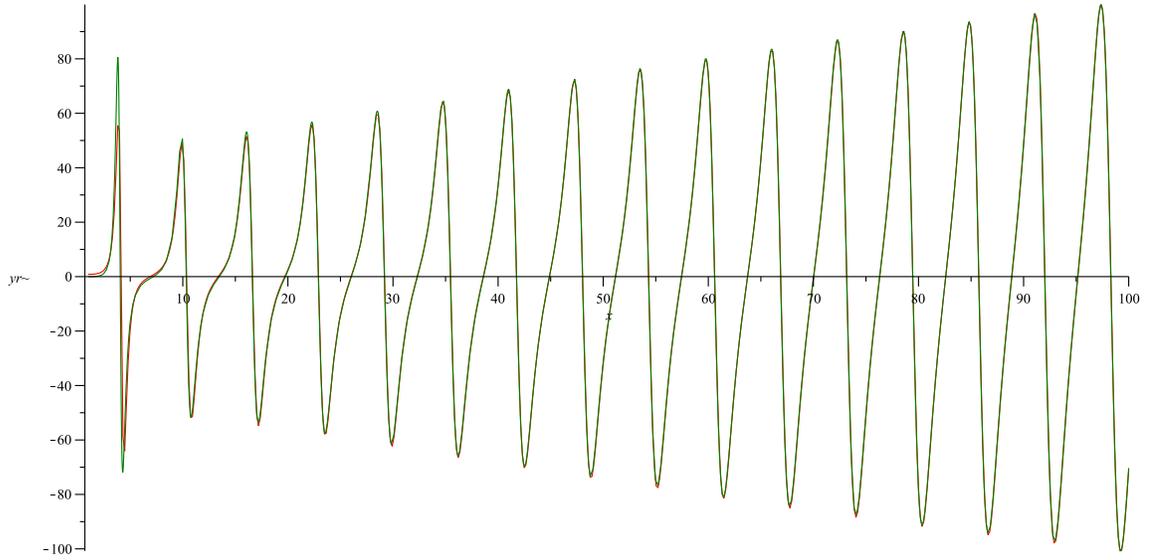}
\caption{Real part of $y$: large-$t$ asymptotics and numerical solution}\label{C4ReY}
\end{center}
\end{figure}
\begin{figure}
\begin{center}
\includegraphics[height=3.0in,width=5in]{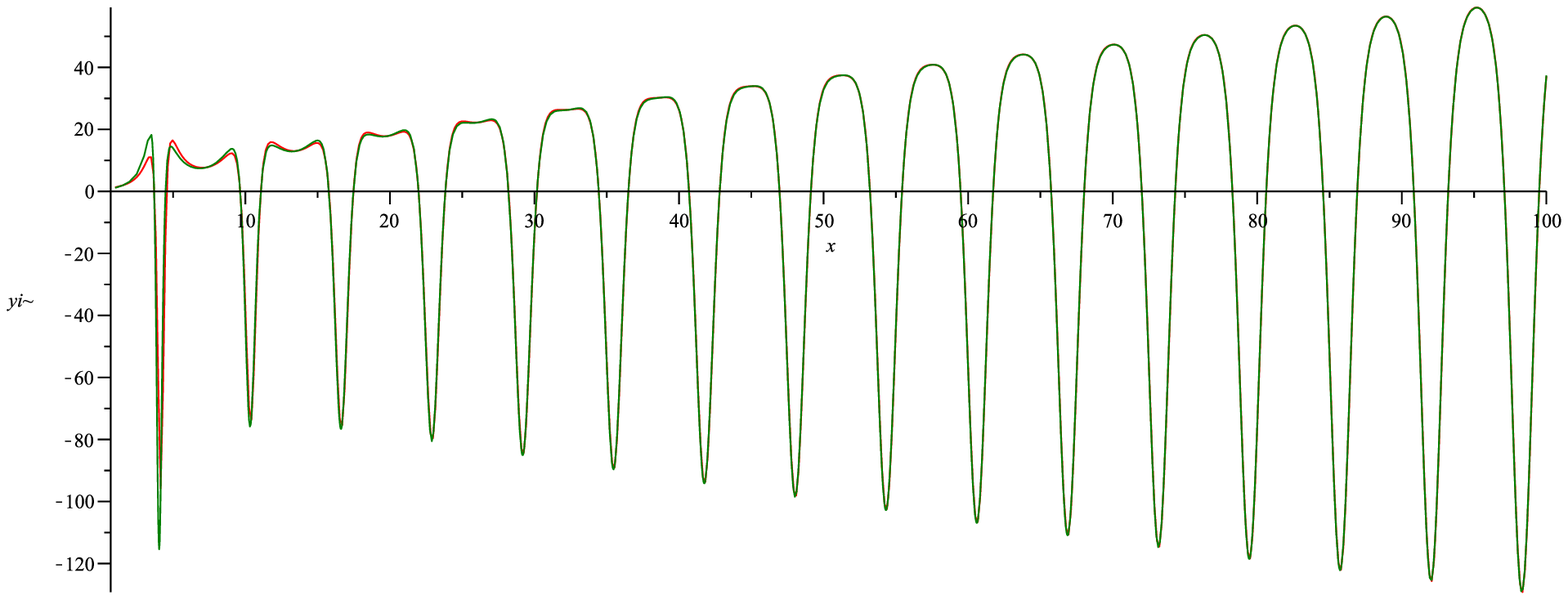}
\caption{Imaginary part of $y$: large-$t$ asymptotics and numerical solution}\label{C4ImY}
\end{center}
\end{figure}
\begin{figure}
\begin{center}
\includegraphics[height=3.0in,width=5in]{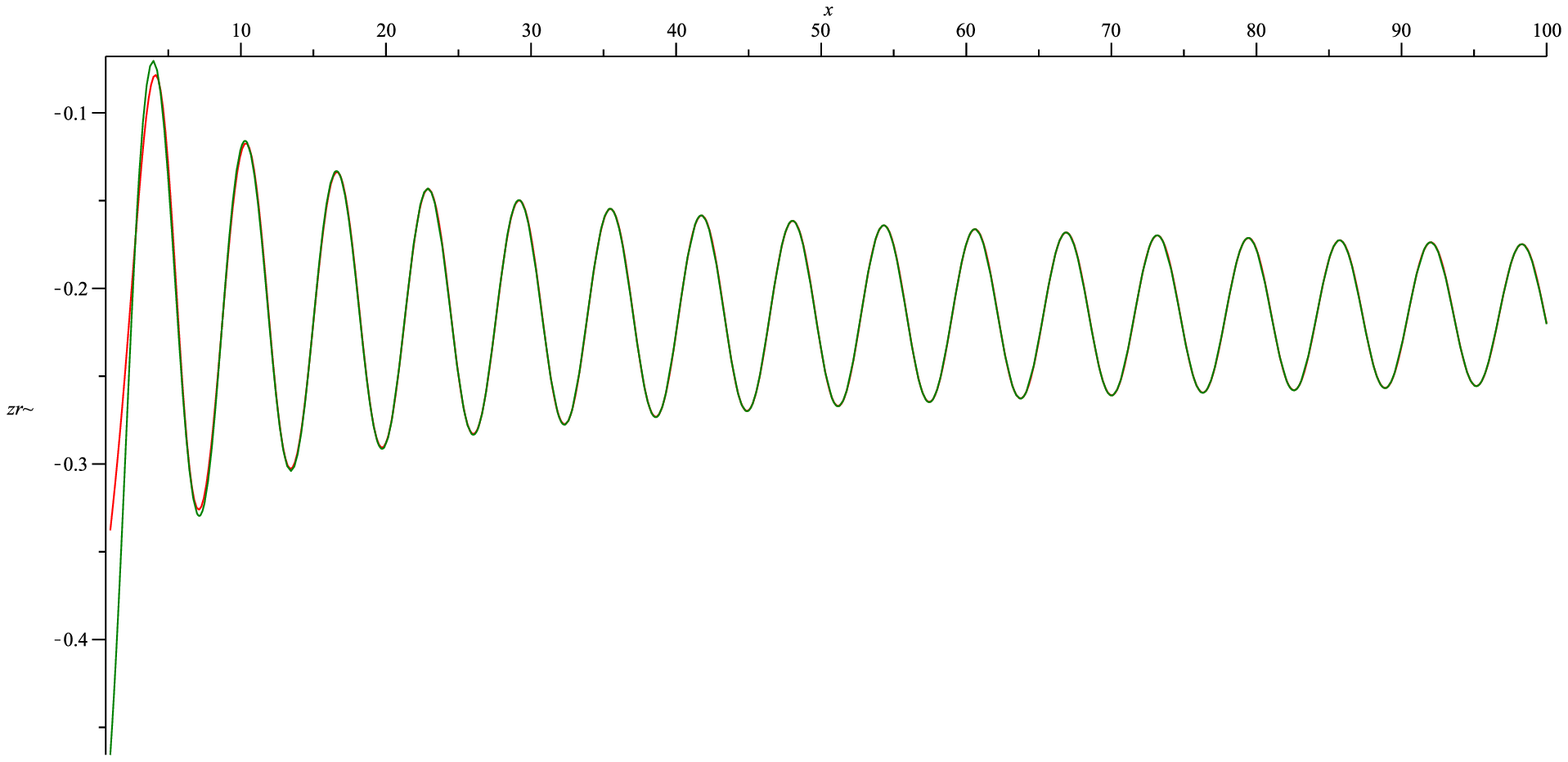}
\caption{Real part of $z$: large-$t$ asymptotics and numerical solution}\label{C4ReZ}
\end{center}
\end{figure}
\begin{figure}
\begin{center}
\includegraphics[height=3.0in,width=5in]{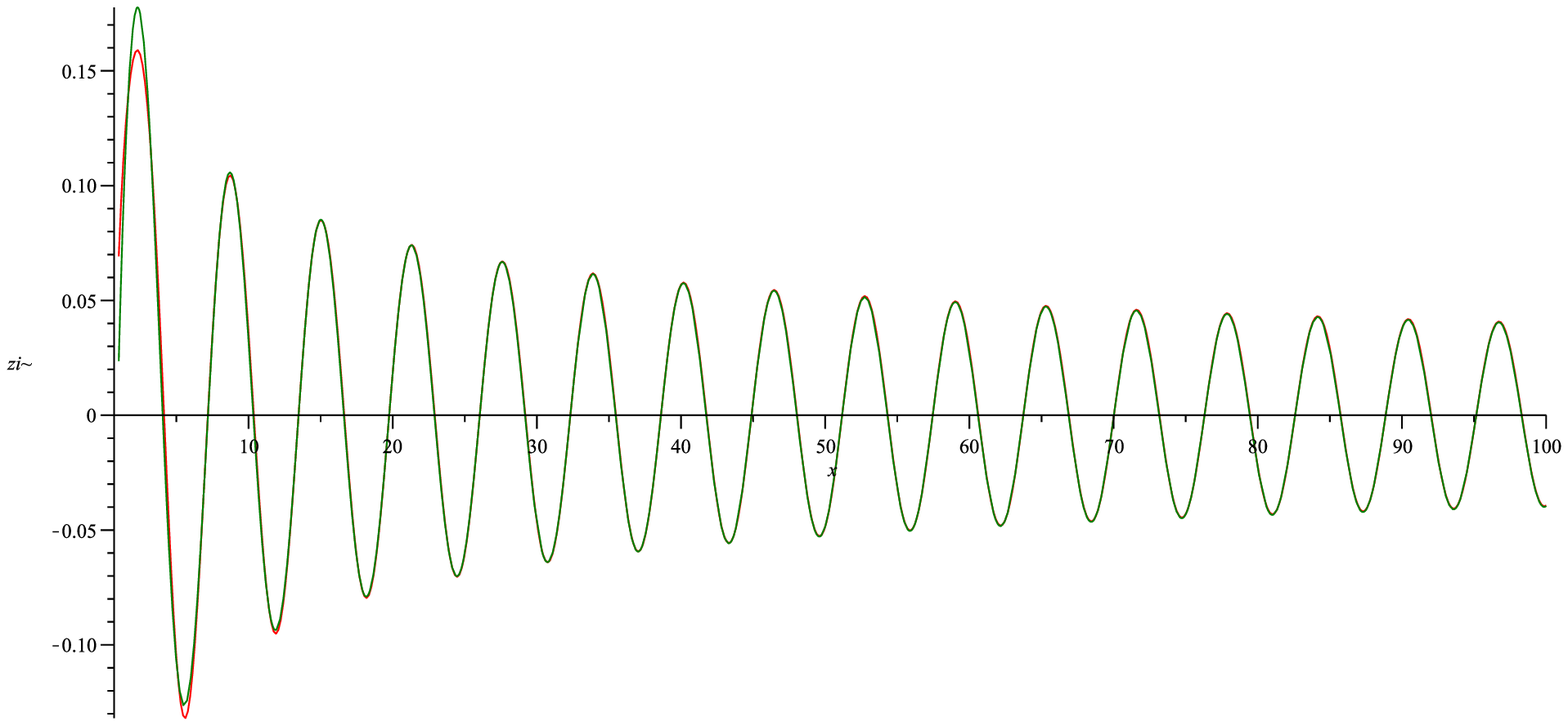}
\caption{Imaginary part of $z$: large-$t$ asymptotics and numerical solution}\label{C4ImZ}
\end{center}
\end{figure}
\begin{figure}
\begin{center}
\includegraphics[height=3.0in,width=5in]{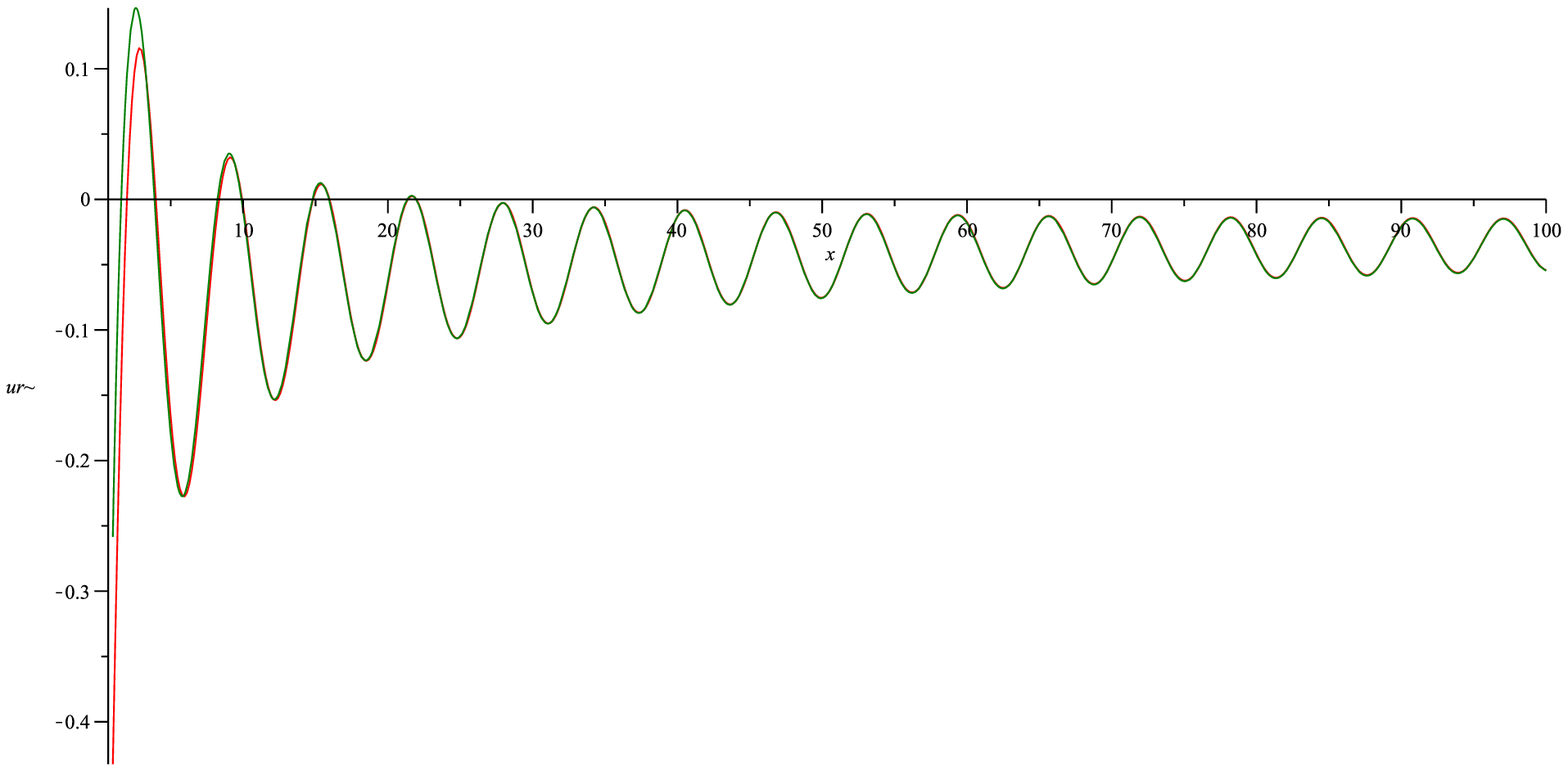}
\caption{Real part of $u$: large-$t$ asymptotics and numerical solution}\label{C4ReU}
\end{center}
\end{figure}
\begin{figure}
\begin{center}
\includegraphics[height=3.0in,width=5in]{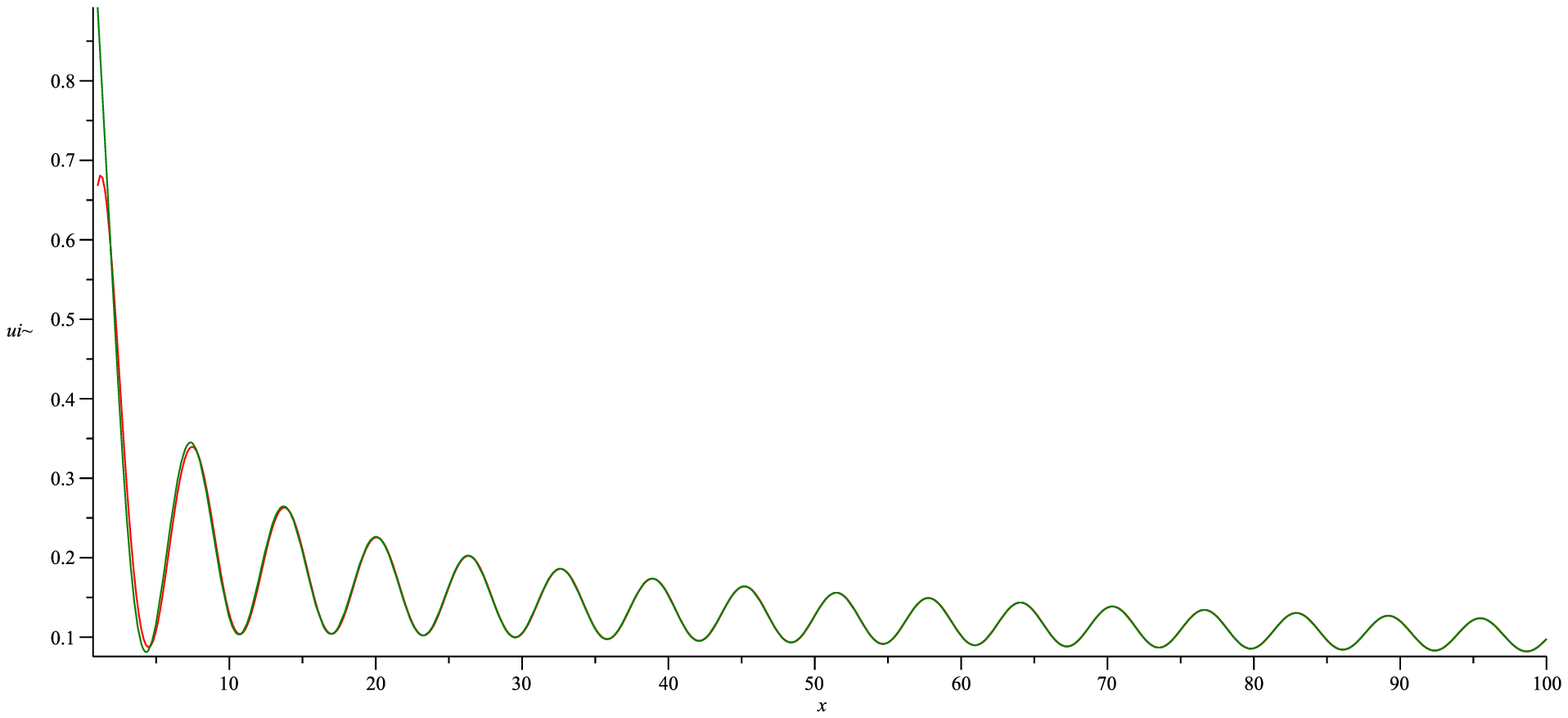}
\caption{Imaginary part of $u$: large-$t$ asymptotics and numerical solution}\label{C4ImU}
\end{center}
\end{figure}

Now, we consider the connection formulae for asymptotics which are obtained in \cite{MT1,MT2}.

As mentioned in Section \ref{sec:mccoy} our results may coincide in some important particular cases,
specifically when $\th_\infty\in2\mathbb{Z}$, $\Theta_0=\Theta_1=\Theta\in\mathbb{Z}$ and, at least, one of the
Stokes multipliers $s_1$ or $s_2$ vanishes. For the complete match of all connection results both Stokes multipliers
should vanish, $s_1=s_2=0$. In the last case there is no difference between solutions $Y_1$, $Y_2$, and $Y_3$, then
we observe an agreement between our results and those of \cite{MT1,MT2}. In particular, a solution of
Equation~\eqref{eq:P5}, describing the one-particle reduced density matrix of the one-dimensional impenetrable
Bose gas~\cite{JMMS}, belongs to this special case.

As follows from Equation~\eqref{eq:sigma-s1-s2} if one of the Stokes multipliers vanishes, then $\sigma=0$.
For the solution considered in this subsection  $\sigma\neq0$ (see Equation~\eqref{eqs:sigma-phi}), therefore
our asymptotic predictions will disagree with those following from papers~\cite{MT1,MT2}.

In paper \cite{MT2} the connection formulae are given in terms of the monodromy data, like in our work. In paper
\cite{MT1} the connection formulae are presented directly: the parameters of asymptotics at infinity are given
in terms of the parameters of asymptotics as $t\to0$. We tried both types of the connection formulae.

We begin with the connection formulae presented in \cite{MT2}. In Section~\ref{sec:mccoy} we explained that
the authors of \cite{MT2} parameterise the large-$t$ asymptotics ($\Im t<0$) in terms of the quantities $I^p$
which are related with the monodromy data via Equation~\eqref{eq:Ip-McCoy-original}. We use the definition of
the monodromy data given in Section~\ref{sec:def-monodromy} and obtained the following expressions for the
Stokes multipliers in terms of $I^p$:
$$
s_2=e^{-i \pi  \left(\0+\1\right)}\times
$$
$$
\frac{ -e^{i \pi (\0+\tin) } (I^0-1) +e^{\pi i  (2\0 +\1)}(I^1-1) -e^{\pi i \1} I^0(I^1 - 1)
+ e^{\pi i (\0 +2\1 +\tin)}(I^0-1) I^1}
{m^0_{21} (I^0-1) (I^1 - 1)}
$$
$$
s_1 = -m^0_{21} e^{\pi i \tin} \times\\
$$
$$
\frac{ e^{\pi i (\0 +2\1)} (I^0-1) -e^{\pi i (\1+\tin)} (I^1 - 1) + e^{\pi i (2\0 +\1 +\tin)}I^0(I^1-1) -
e^{\pi i \0} I^1 (I^0-1) }
{e^{\pi i \tin} -e^{\pi i (2\0 +\tin)} I^0 -e^{\pi i (\0 +\1) }(I^0-1) (I^1 - 1) - e^{\pi i (2\1+\tin)}I^1
+e^{\pi i (2\0 +2\1+\tin)}I^0 I^1}.
$$
For the case $\0=\1=\Theta$, $\tin = 0$, studied in \cite{MT1}-\cite{MT2}, they simplify to:
\begin{equation}\label{eqs:stokes103}
s_1 = m^0_{21} e^{\pi i \Theta} \frac{I^0 I^1 -1 }{1-	e^{2 \pi i \Theta}I^0 I^1} ,\quad
s_2 = e^{-\pi i \Theta} (1-e^{2 \pi i \Th} ) \frac{ 1-I^0 I^1}{(I^0 - 1)(I^1-1) m^0_{21}}.
\end{equation}
For the particular numerical case studied here Equation~\eqref{eq:Ip-McCoy-original} implies:
$$
I^0(MT)=1.056514170\ldots+\imath1.198380092\ldots,\quad
I^1(MT)=1.338535154\ldots-\imath0.824893015\ldots.
$$
Substituting these values into Equation~\eqref{eqs:stokes103} with $r=1$ we find
$$
s_1 = s_2 = -\imath1.175570504\ldots,
$$
which confirms the values of the Stokes multipliers \eqref{eqs:stokes103direct} calculated directly.

It is worth reminding that for connection formulae, if we want to remain within the results of \cite{MT1}-\cite{MT3},
the parameters $I^p$ should be (in our notation) adjusted as explained in Section~\ref{sec:mccoy}
(see Equation~\eqref{eq:Ip-McCoy-adjusted}). With parameters $I^p$ understood in that way we would find:
$$
I^0(adj)=0.541451276\ldots+\imath0.333677733\ldots,\quad
I^1(adj)=0.413939912\ldots-\imath0.469522666\ldots.
$$
Calculating the Stokes multipliers with the help of $I^0(adj)$ and $I^1(adj)$ via \eqref{eqs:stokes103} with $r=1$ we get
$$
s_1(adj)=-0.661221513\ldots+\imath0.718177045\ldots,\quad
s_2(adj)=\imath1.175570504\ldots
$$
As expected, we obtain different Stokes multipliers.

To get the large-$t$ asymptotics as suggested in \cite{MT2} we have to calculate parameters $k$ and $\tilde{x}_0$
(see Section~\ref{sec:mccoy}) Equations~\eqref{I1MT2}--\eqref{Cbar}.
First we do it with the help of $I^0(MT)$ and $I^1(MT)$:
\begin{align*}
k=k(MT)&=0.079933525\ldots-\imath0.158878848\ldots,\\
\tilde{x}_0=\tilde{x}_0(MT)&=0.745125463\ldots-\imath0.350495266\ldots.
\end{align*}
The plot of the large-$t$ asymptotics corresponding to these parameters calculated with the help
of Equation~\eqref{eq:mccoy3} is presented on Figure~\ref{MTFigORIG}.
Comparing this plot with the ones presented on Figure~\ref{C4ReY} we see the main discrepancy is that
the function on Figure~\ref{MTFigORIG} is decaying while it should grow.

We also tried in this case the direct connection formulae given in \cite{MT1} (see Equations (2.3) and (2.14)--(2.17)
of \cite{MT1}). We applied the scheme of calculations adopted in this paper: first with the help of the initial value
found above ($\tilde y(x_0)$ and $\sigma=0.4$) we calculated, with the help of Equation (2.3) of \cite{MT1}, the parameter
$\hat S$, then parameters $k$ and $\tilde{x}_0$ by means of Equations (2.14)--(2.17)of \cite{MT1}. Thus obtained $k$ and
$\tilde{x}_0$ differ with $k(MT)$ and $\tilde{x}_0(MT)$ obtained above. Nevertheless, the qualitative behavior
of asymptotics is the same as on Figure~\ref{MTFigORIG}. The difference in the values of the parameters could be
explained because our small-$t$ asymptotics, being formally equivalent to the one used in \cite{MT1}, in fact,
in most cases numerically is more precise. It is not an obstacle in usage of Equation (2.3) of \cite{MT1}, but
the point $x_0$ for calculation of initial value should be chosen closer to the origin, which may require a higher
precision of calculations.
\begin{figure}
\begin{center}
\includegraphics[height=2.0in,width=5in]{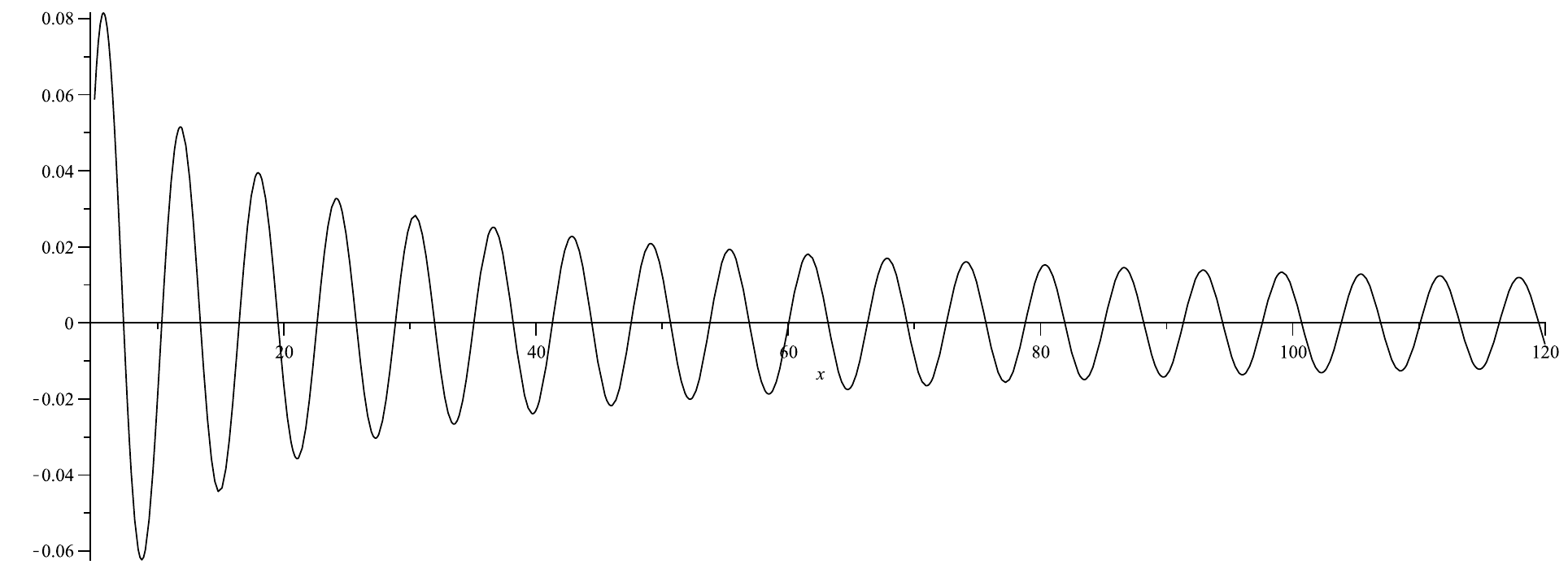}
\caption{$\Re y$: the large-$t$ asymptotics for pure imaginary negative $t$ obtained in \cite{MT2}}\label{MTFigORIG}
\end{center}
\end{figure}

We also calculated the plots of asymptotics of $y$ for parameters $I^0(adj)$ and $I^1(adj)$.
In this case asymptotic parameters are as follows:
$$
k=k(adj)= \imath0.15,\qquad
\tilde{x}_0=\tilde{x}_0(adj)=-0.619264010\ldots+\imath0.413313056\ldots.
$$
We have further (pure experimentally) adjusted parameter $\tilde{x}_0(adj)$, namely,
we changed $\tilde{x}_0(adj)\to \tilde{x}_0(adj)+\pi/4$.
On Figures~\ref{MTFigReYadjusted} and \ref{MTFigImYadjusted} we compared asymptotics of \cite{MT2}, adjusted as
explained above, with the ones presented in this work. In the online version our asymptotics colored in green
while those of \cite{MT1} are blue. At first glance, asymptotics for $\Re y$ looks quite satisfactory, while for
$\Im y$ there is some discrepancy which also does not look fatal. Because one may hope that the higher order terms
may further correct the situation.
However, as we see above our asymptotic formulae from Theorems~\ref{th:new1} and \ref{th:new2}, gives good
numeric approximations starting from quite small values of $t$; Asymptotics~\eqref{eq:mccoy3} works at these
values of $t$ not that good. We compared the solutions on the interval $x\in [2,30]$ and observed significant
disagreement between adjusted solutions from [7] and the numerical results. Of course, the interval $[2,30]$ may
not be asymptotically "large enough" yet. Therefore, we compared asymptotics in segment $x\in [800,820]$ where
they are supposed to coincide (visually) with the numerical solution. At least with the growth of $x$ both curves
should approach each other.
However, the difference between the curves looks "stable" with variation of $x$ within the range $x\in[400,1200]$.
This corresponds with what is written in Section~\ref{sec:mccoy}. As explained in Section~\ref{sec:mccoy},
the adjustment suggested there does not repair the situation, one have to put zero at least one of the Stokes
multiplier. So, Figures~\ref{MTFigReYadjusted} and \ref{MTFigImYadjusted} provide us an illustration to the
conclusions of Section~\ref{sec:mccoy}.
\begin{figure}
\begin{center}
\includegraphics[height=2.0in,width=5in]{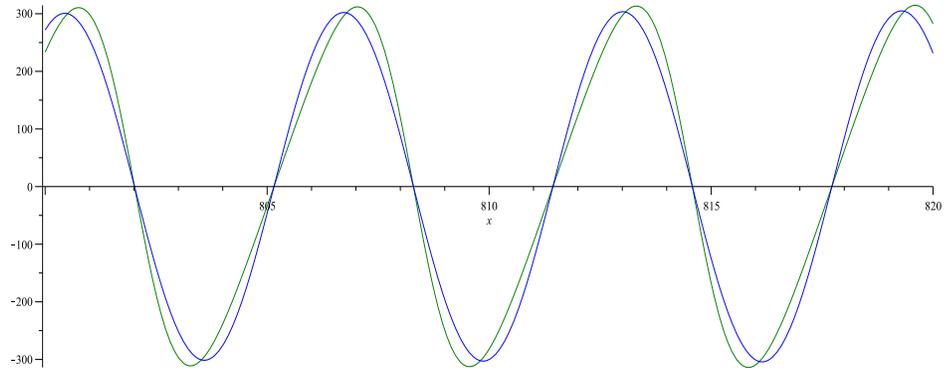}
\caption{$\Re y$: Large-$t$ asymptotics for pure imaginary negative $t$: Theorem~\ref{th:new2}
vs the adjusted one from \cite{MT2}}\label{MTFigReYadjusted}
\end{center}
\end{figure}
 \begin{figure}
\begin{center}
\includegraphics[height=2.0in,width=5in]{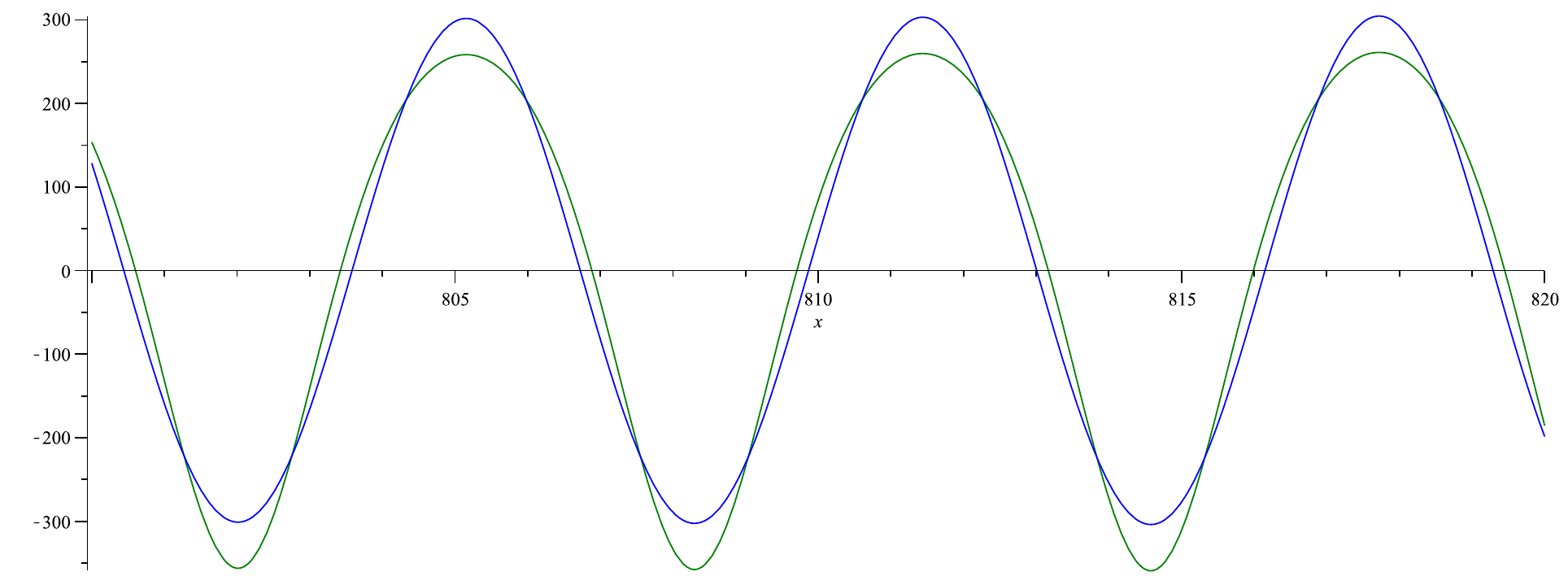}
\caption{$\Im y$: Large-$t$ asymptotics for pure imaginary negative $t$: Theorem~\ref{th:new2}
vs the adjusted one from \cite{MT2}}\label{MTFigImYadjusted}
\end{center}
\end{figure}
\subsection{The Special Meromorphic Solution}\label{subsec:numeric-extra-terms}

In this subsection we deal with the solution studied in Section~\ref{sec:spec-meromorphic-solution}. Since it is a very
special solution we have here an opportunity to to choose a different numeric scheme and not the one used in previous
subsections. We put $t=\imath x$, with $x\in\mathbb{R}$. We also assume $a_1^1\in\mathbb R$. Then the second
transformation in Subsection~\ref{subsec:t-to-minus-t} implies:
$$
\bar{y(t)}y(t)=1,\quad
\bar{z(t)}=-z(t),
$$
where the bar denotes complex conjugation. We recall that $y(0)=-1$ and $z(0)=0$, therefore we make the following
change of variables:
$$
y(t)=-e^{\imath\phi(x)},\qquad
z(t)=\imath xw(x),
$$
where $\phi(x)$ and $w(x)$ are realvalued functions of real variable $x$.

Taking into account that $\Theta_0=\Theta_1=1/2$ and $\Theta_\infty=-1$ we rewrite System~\eqref{eq:ids1}, \eqref{eq:ids2}
as follows:
\begin{equation*}
\begin{aligned}
  x\frac{d\phi}{dx}&=x+8xw\cos^2{\frac{\phi}2}-\sin\phi,\\
  x\frac{dw}{dx}&=w(\cos\phi-1+2xw\sin\phi).
\end{aligned}
\end{equation*}
The main difference in numeric calculation of this solution comparing with the previous examples is that we can
take the initial condition exactly at $x=0$:
$$
\phi(0)=0,\qquad
w(0)=-\frac{1+2a_1^1}8.
$$
On Figures~\ref{fig:SpecialCaseRe-y-0-60} and \ref{fig:SpecialCaseIm-y-0-60} we present the plots of real and
imaginary parts of solution $y(t)$ for $a_1^1=2$, respectively. Each figure contains two plots:  the numerical
plot and asymptotic one. To plot asymptotics of $y(t)$ we use only the leading term
denoted as $\alpha$, see Equation~\eqref{eq:alpha-in-delta-phi} in Subsection~\ref{subsec:asympt-special-imaginary}.
We see that the leading term of the large $x$-asymptotics gives already a very good approximation as early as $x=2.5$.
A rather remarkable property of that example is that there is no visible transition
interval from the behavior of $y(t)$ at $t\to 0$ to its large $t$-behavior.
\begin{figure}[htbp]
\begin{center}
\includegraphics[height=2.5in,width=5in]{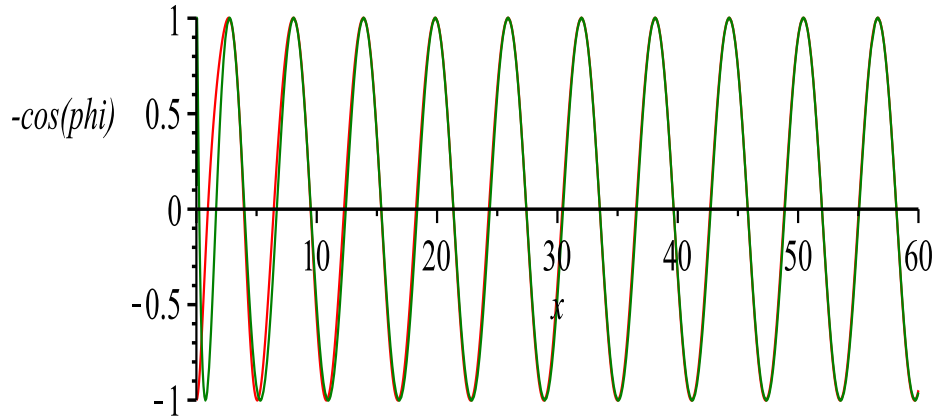}
\caption{$\Re y$; the numerical plot starts at the point with the coordinates $(0,-1)$.\label{fig:SpecialCaseRe-y-0-60}}
\end{center}
\end{figure}
\begin{figure}[htbp]
\begin{center}
\includegraphics[height=2.5in,width=5in]{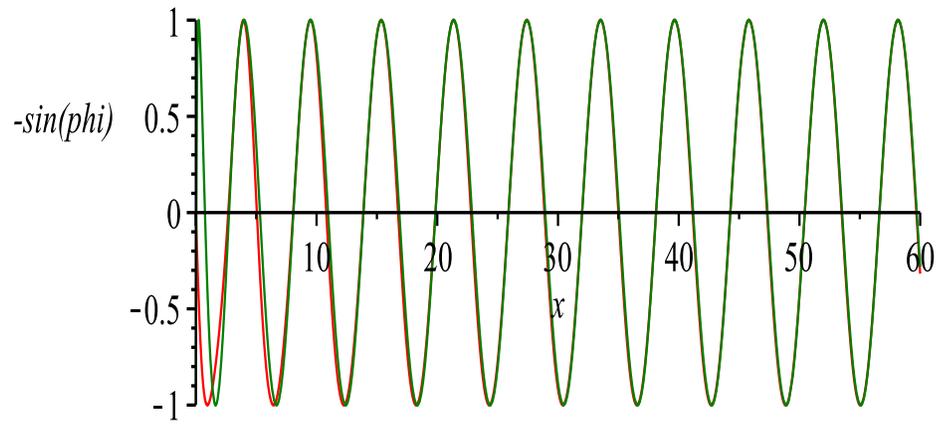}
\caption{$\Im y$; the numerical plot approach the point with the coordinates $(0,0)$ from below;
the asymptotical one --- from above.\label{fig:SpecialCaseIm-y-0-60}}
\end{center}
\end{figure}
Since the leading asymptotic term of $z(t)$ is a constant we used the explicit form of the oscillating correction
term ($\mathcal{O}(1/t)$), to compare it with the numerical solution on Figure~\ref{fig:SpecialCaseIm-z-0-60}.
\begin{figure}[htbp]
\begin{center}
\includegraphics[height=2.5in,width=5in]{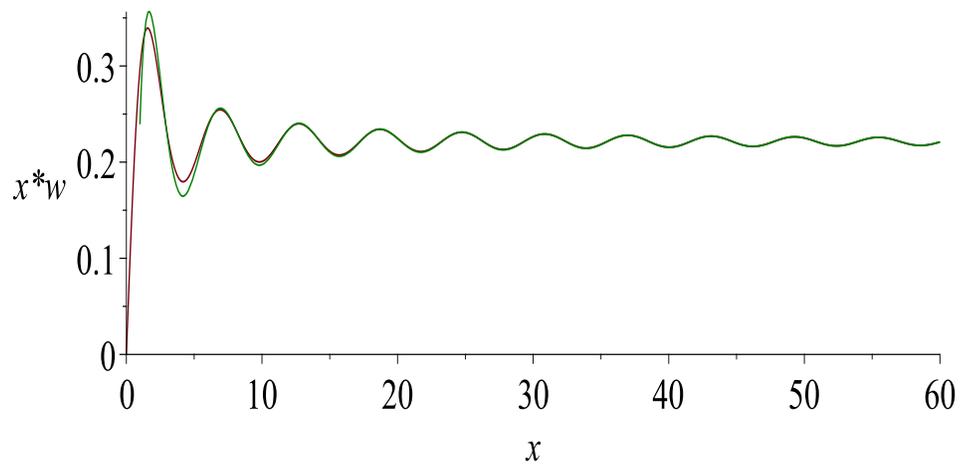}
\caption{$\Im z$; the numerical plot approach the point with the coordinates $(0,0)$. The plot of asymptotics
starts at $x=1$.\label{fig:SpecialCaseIm-z-0-60}}
\end{center}
\end{figure}
\appendix
\clearpage
\renewcommand{\thesection}{\Alph{section}}
\renewcommand{\theequation}{\Alph{section}.\arabic{equation}}
\renewcommand{\thetheorem}{\Alph{section}.\arabic{theorem}}
\section{Appendix: Schlesinger transformations}\label{sec:trans}
Here for completeness we present the Schlesinger transformations of Equation~\eqref{mainl} together with the
corresponding B\"acklund transformations of IDS~\eqref{eq:ids1}-\eqref{eq:ids3}, which we use in this paper.

We denote $Y_k$, $k\in\mathbb{Z}$, the canonical solution of
Equation~\eqref{mainl} corresponding to the set of variables $t,\lambda,y,z,u,\Theta_0,\Theta_1,\Theta_\infty$,
the notation $\tilde{Y}_k$ states for the canonical solution of Equation~\eqref{mainl} but with variables:
$t,\lambda,\tilde{y},\tilde{z},\tilde{u},\tilde{\Theta}_0,\tilde{\Theta}_1,\tilde{\Theta}_\infty$. The Schlesinger
transformation maps $Y_k$ to $\tilde{Y}_k$ such that both solutions have the same monodromy data, except the formal
monodromy at the point at infinity which is shifted by $\pm2$. Due to these properties the Schlesinger transformations
are called sometimes the discrete isomonodromy transformations.

In transformation presented below in Theorem~\ref{dress} the original non-tilde variables might be arbitrary.
In that case the tilde variables represent just a covariant transformation of Equation~\eqref{mainl}.
As long as we require dependence of the tilde variables on $t$, such that they satisfy IDS~\eqref{eq:ids1}-\eqref{eq:ids3}, the corresponding
monodromy data are independent of $t$ (continuous isomonodromy deformations). Obviously, in that case
the monodromy data for $\tilde{Y}_k$ are also independent of $t$, so that the Schlesinger transformations
generate the so-called B\"acklund transformations for solutions of IDS~\eqref{eq:ids1}-\eqref{eq:ids3}, i.e.,
the discrete and continuous isomonodromy flows commute.
\begin{theorem}\label{dress}
There exists the following Schlesinger transformation of
Equation~(\ref{mainl})
\begin{eqnarray*}
&Y_k = \bm \dfrac p\be +\la & -\be \\ \dfrac 1\be& 0 \em  \ti Y_k,&\\
&\be = \dfrac t{A^1_{21}+A^0_{21}}, \quad p = -\dfrac
\be t (\tin +1 +
\be A^1_{21}),&\\
&\ti \Th_\infty = \tin +2,\quad \ti \Th_0 = \Th_0,\quad \ti
\Th_1 = \Th_1,&\\
\end{eqnarray*}
where $A^0_{21}$ and $A^1_{21}$ are the corresponding elements of matrices $A_0$ and $A_1$, respectively,
defined for the non-tilde variables.
\begin{eqnarray*}
\ti z& =& -z - \Th_0 - \dfrac{(\tin+1) z } {z -\dfrac1y \left(z +\dfrac{\Th_0 +\Th_1 + \tin}2\right)}+ 
\dfrac{tz\left(z +\dfrac{\Th_0+\Th_1 +\tin}2\right)}{y \left(z - \dfrac1y
\left(z+\dfrac{\Th_0+\Th_1 + \tin}2\right)\!\right)^2},\\
\ti y& =& \dfrac 1y\cdot \dfrac{ \ti z+\Th_0} z \cdot
\dfrac{z+\dfrac{\Th_0+\Th_1+\tin}2}{\ti z+\dfrac{\Th_0-\Th_1+\ti \tin}2},\\
\ti u &=& u\cdot\dfrac{zt^2}{\ti z+\Th_0}\cdot
\dfrac1{\left(z-\dfrac1{y}\left(z+\dfrac{\Th_0+\Th_1+\tin}{2}\right)\!\right)^2}.
\end{eqnarray*}

There exists the following inverse Schlesinger transformation  of
Equation~(\ref{mainl})
\begin{eqnarray*}
&Y_k = \bm 0 & \dfrac 1\be \\ -\be& \dfrac p\be +\la \em \ti Y_k,&\\
&\be = -\dfrac t{A^1_{12}+A^0_{12}},
\quad p = -\dfrac \be t (\tin -1 -
\be A^1_{12}),&\\
&\ti \Th_\infty = \tin -2,\quad \ti \Th_0 = \Th_0,\quad \ti
\Th_1 = \Th_1,&
\end{eqnarray*}
\begin{eqnarray*}
\ti z& = &-z - \Th_0 - \dfrac{(\tin-1)( z+\Th_0 )}
{z+\Th_0 - y \left(z +\dfrac{\Th_0 -\Th_1
+ \tin}2\right)}- 
\dfrac{ty\left(z+\Th_0\right)\left(z +\dfrac{\Th_0-\Th_1 +\tin}2\right)}
{\left(z+\Th_0 -y \left(z
+\dfrac{\Th_0-\Th_1 + \tin}2\right)\!\right)^2},\\
\ti y &=& \dfrac 1y\; \dfrac{ z+\Th_0} {\ti z}\; \dfrac{
\ti z+\dfrac{\Th_0+\Th_1+\ti \Th_\infty}2}{z+\dfrac{\Th_0-
\Th_1+\tin}2},\\
\ti u& = &u\cdot \dfrac{z+\Th_0}{\ti zt^2}\cdot\left(\ti z-\dfrac 1{\ti
y}\left(\ti z+\dfrac{\Th_0+\Th_1+\ti \Th_\infty}2\right)\!\right)^2.
\end{eqnarray*}
\end{theorem}

\section{Appendix: Higher order terms in asymptotic expansion}\label{sec:allterms}

We begin with the complete asymptotic expansions at large pure imaginary $t$ corresponding to solutions of
System~\eqref{eq:ids1}--\eqref{eq:ids2} described in Theorem~\ref{th:new1}. The first expansion reads:
\begin{eqnarray}
&z=\sum^\infty_{i=0}\sum^{i}_{j=-i}z_{ij}t^{-i}\al^{j},&\label{eq:complete-z-asympt-3.1}\\
&y =\al \sum^\infty_{i =0} \sum^i_{j=-i}y_{ij}t^{-i}\al^{j},&\label{eq:complete-y-asympt-3.1}\\
&\al = \de t^{-4\vp+\tin}e^{t}.\label{eq:alpha-def}&
\end{eqnarray}
These are the formal (divergent) series with coefficients uniquely determined by substitution into
System~\eqref{eq:ids1}--\eqref{eq:ids2} under a recurrence procedure. Below, the reader will find the first terms
corresponding to $i=0,1,2$ which we found with the help of MATHEMATICA code. This formal expansion becomes asymptotic
one provided $0<\Re\nu_1<2$, where $\nu_1=1-4\vp+\tin$  is introduced in Theorem~\ref{th:new1}.

The justification of these asymptotics
can be obtained with the help of the scheme described in Section~33 of \cite{W}. The major difference is that the
sector where our expansions are valid has the vanishing angle (parameter $q=\infty$ in \cite{W}). This fact, however,
does not destroy the proof because of Remark~\ref{rem:mu-modifications}, where we explain that, in fact, asymptotics
is working in domains $\mathcal D^1(\mu_1)$ rather than only on the imaginary axis. Therefore, in case $0<\Re\nu_1<2$
there exists a solution of System~\eqref{eq:ids1}--\eqref{eq:ids2} with
Asymptotics~\eqref{eq:complete-z-asympt-3.1}--\eqref{eq:complete-y-asympt-3.1}. Since according to Theorem~\ref{th:allmon}
parameter $\alpha$ uniquely defines the solution of System~\eqref{eq:ids1}--\eqref{eq:ids2}, the solution which exists
by the application of the Wasow scheme coincides with the one that described by the monodromy theory and there is no
any other solution with the same asymptotic expansion.

We find the coefficients of this asymptotic expansion to be:
\begin{equation}\label{eqs:asymptseries-y-1}
\begin{gathered}
y_{00}=1,\quad
y_{10}=12\vp^2-6\vp\tin+\frac 12(\tin^2-\Th_0^2-\Th_1^2),\\
y_{11}= 2\vp +\frac 12 (\Th_0+\Th_1-\tin),\quad
\quad y_{1,-1}= -2\vp+\frac
12(\Th_0+\Th_1+\tin),
\end{gathered}
\end{equation}
\begin{equation}\label{eqs:asymptseries-z-1}
\begin{gathered}
z_{00} = -\vp -\frac 12 \Th_0,\quad
z_{10}=0,\\
z_{11}=\vp^2+\frac\vp2(\Th_0+\Th_1-\tin)+\frac {\Th_0}4(\Th_1-\Th_\infty),\\
\quad z_{1,-1}=\vp^2-\frac\vp2(\Th_0+\Th_1+\tin)+\frac{\Th_0}4(\Th_1+\Th_\infty),\\
\end{gathered}
\end{equation}
\begin{align*}
y_{22}&=3\vp^2+\frac{3\vp}2(\Th_0+\Th_1-\tin)+
\frac14(\Th_0+\Th_1^2+\tin^2+\Th_0\Th_1-\Th_0\Th_2-2\Th_1\Th_2),\\
y_{2,-2}&=\vp^2-\frac\vp2(\Th_0+\Th_1+\tin)+\frac{\Th_0}4(\Th_1+\tin),\\
\end{align*}
\begin{align*}
y_{2,1}&=48\vp^3+12\vp^2(\Th_0+\Th_1-3\tin+1)+\\
&+2 \vp\big(4\tin^2-3\tin(\Th_0+\Th_1+1)-\Th_0^2-\Th_1^2+2\Th_0+2\Th_1+1\big)-\\
&-\frac{\tin^3}2+\frac{\tin^2}2(\Th_0+\Th_1+1)+\frac{\tin}2(-1-\Th_1-3\Th_0+\Th_0^2+\Th_1^2)+\\
&+\Th_0\Th_1-\frac{\Theta_0+\Theta_1}2(\Theta_0^2+\Theta_1^2-1),\\
y_{2,-1}&=-12\vp^2 +2\vp(1+2\Th_0+2\Th_1 +3\tin)-\\
&-\frac 12(\Th_1+\Th_0+\tin +\Th_1\tin+2 \Th_0\Th_1+3 \Th_0\tin),\\
y_{20}&=72\vp^4-4\vp^3(18\tin-10)-2\vp^2(15\tin+3\Th_0^2+3\Th_1^2-12\tin^2+2)+\\
&+\vp\big(-3\tin^3+6\tin^2+(3\Th_1^2+2+3\Th_0^2)\tin+\Th_1-3\Th_0^2-3\Th_1^2+\Th_0\big)+\\
&+\frac{\tin^4}8-\frac{\tin^3}4-(\Th_0^2+\Th_1^2+1)\frac{\tin^2}4+(\Th_1^2+5\Th_0^2-2\Th_0)\frac{\tin}4+\\
&+\frac{(\Th_0^2+\Th_1^2)^2}8+\frac{(\Th_0+\Th_1)^2}4,\\
\end{align*}
\begin{align*}
z_{22}&=z_{2,-2}=0,\\
z_{20}&= -4\vp^3+3\vp^2\tin+\frac{\vp}2(\Th_0^2-\tin^2+\Th_1^2)-\frac{\tin\Th_0^2}4,\\
z_{21}&= 12\vp^4+2\vp^3(2+3\Th_0+3\Th_1-6\tin)+\\
&+\frac{\vp^2}2\Big(7\tin^2-6(1+{\Th_1}+2\Th_0)\tin+4\Th_0+2+4\Th_1-\Th_1^2+6\Th_0\Th_1-\Th_0^2\Big)-\\
&-\frac{\vp}4\Big(\tin^3-(7\Th_0+\Th_1+2)\tin^2-(\Th_0^2+\Th_1^2-6\Th_0\Th_1-6\Th_0-2\Th_1-2)\tin-\\
&-2\Th_0-2{\Th_1}-4\Th_0\Th_1+\Th_1^3+\Th_0^2\Th_1+\Th_0\Th_1^2+\Th_0^3\Big)-\\
&-\frac{\Th_0}8\Big(\tin^3-(\Th_1+2)\tin^2-(\Th_0^2+\Th_1^2-2\Th_1-2)\tin+(\Th_0^2+\Th_1^2-2)\Th_1\Big),\\
z_{2,-1}&= -12\vp^4+2\vp^3(2+3\Th_0+3\Th_1+6\tin)+\\
&-\frac{\vp^2}2\Big(7\tin^2+6(1+\Th_1+2\Th_0)\tin+4\Th_0+2+4\Th_1-\Th_1^2+6\Th_0\Th_1-\Th_0^2\Big)+\\
&+\frac{\vp}4\Big(\tin^3+(7\Th_0+\Th_1+2)\tin^2-(\Th_0^2+\Th_1^2-6\Th_0\Th_1-6\Th_0-2\Th_1-2)\tin+\\
&+2\Th_0+2\Th_1+4\Th_0\Th_1-\Th_1^3-\Th_0^2\Th_1-\Th_0\Th_1^2-\Th_0^3\Big)-\\
&-\frac{\Th_0}8\Big(\tin^3+(\Th_1+2)\tin^2-(\Th_0^2+\Th_1^2-2\Th_1-2)\tin-(\Th_0^2+\Th_1^2-2)\Theta_1\Big),\\
z_{33}&=z_{3,-3}=0,\qquad y_{3,-3}=0,\qquad y_{3,3}\neq 0,\qquad\ldots.
\end{align*}
Substituting the above formulae into Equation~\eqref{eq:zeta-def} we find the following result,
\begin{eqnarray*}
&\ze=(\vp+\frac {\Th_0}2)t-2\vp^2+ \vp\tin +\frac{\Th_0}2(\Th_0+\tin) +&\\
&+\frac1t\bigg(\alpha(-\vp^2+\frac{\vp}2(\tin-\Th_1-\Th_0)+\frac{\Th_0}4(\tin-\Th_1))&-\\
&-4\vp^3+3\vp^2\tin+\frac{\vp}2(\Th_0^2-\tin^2+\Th_1^2) -\frac 14\tin\Th_0^2+&\\
&+\frac1\alpha(\vp^2-\frac{\vp}2(\tin+\Th_1+\Th_0) +\frac {\Th_0}4(\tin+\Th_1))
\bigg)+\mathcal{O}\left(t^{-2}\left(|\alpha|^2+|\alpha|^{-2}\right)\right),&
\end{eqnarray*}
which is formulated in Corollary~\ref{cor:results-zeta}.
Up to $\mathcal{O}\left(t^{-2}\left(|\alpha|^2+|\alpha|^{-2}\right)\right)$ the equality
$$
\frac{\cd \ze}{\cd t}=-z
$$
is satisfied.

The leading term of asymptotics in Theorem~\ref{th:new1} is valid also for negative values of $\Re\nu_1$. The corresponding
asymptotic expansion can be constructed as follows.
Define
\begin{equation}\label{eq:beta-def}
\beta=\delta t^{-4\vp +\tin+1}e^t\eq \delta t^{\nu_1}e^t.
\end{equation}
Consider the formal series:
\begin{align}
z&= \sum_{i=0}^\infty\sum_{j=-i-1}^{[\frac i2]} \hat{z}_{ij} t^{-i} \beta^j,\label{eq:beta-series-z}\\
ty&= \sum_{i=0}^\infty\sum_{j=-i-1}^{1+[\frac i2]} \hat{y}_{ij} t^{-i} \beta^j.\label{eq:beta-series-y}
\end{align}
where $[\cdot]$ is the integer part of a number (the integer floor). One proves that these series are asymptotic
as $t\to\infty$, iff $-1<\Re\nu_1<1$. Comparing Equations~\eqref{eq:alpha-def} and \eqref{eq:beta-def} we see that
$\beta=t\alpha$.
Substituting this relation into Series~\eqref{eq:beta-series-z} and \eqref{eq:beta-series-y}, it is easy to observe
that after a rearrangement of terms they coincide with
Series~\eqref{eq:complete-z-asympt-3.1} and \eqref{eq:complete-y-asympt-3.1}, respectively. Any partial sum of
Series~\eqref{eq:complete-z-asympt-3.1} or \eqref{eq:complete-y-asympt-3.1} can be presented as a partial sum of
the corresponding Series~\eqref{eq:beta-series-z} or \eqref{eq:beta-series-y} with some extra higher order terms
and vice versa so that both pairs of series solve IDS~\eqref{eq:ids1} and \eqref{eq:ids2}. So,
Series~\eqref{eq:beta-series-z} and \eqref{eq:beta-series-y} represent asymptotics of the same solution as
Series~\eqref{eq:complete-z-asympt-3.1} and \eqref{eq:complete-y-asympt-3.1} but in a shifted domain of the parameter
$\nu_1$: $-1<\Re\nu_1<1$. We discuss this expansion in a more detail in part II of this paper.

For the solutions defined in Theorem~\ref{th:new2} we have to write another asymptotic expansions.
Define
$$
\tilde{\beta}=\frac 1\delta t^{4\vp -\tin+1}e^{-t}\eq \frac1\delta t^{\nu_2}e^{-t},
$$
then for $\nu_2$ such that $-1<\Re\nu_2<1$, the following formal series are asymptotic:
\begin{align*}
z&= \sum_{i=0}^\infty\sum_{j=-i-1}^{[\frac i2]}\tilde{z}_{ij} t^{-i}\tilde{\beta}^j,\\
\frac ty&= \sum_{i=0}^\infty\sum_{j=-i-1}^{1+[\frac i2]}\tilde{y}_{ij} t^{-i}\tilde{\beta}^j.
\end{align*}
For $1<\Re\nu_2<2$ we have to use asymptotic expansion analogous to \eqref{eq:complete-z-asympt-3.1} and
\eqref{eq:complete-y-asympt-3.1}:
\begin{eqnarray}
z&=\sum^\infty_{i=0}\sum^{i}_{j=-i}\check{z}_{ij}t^{-i}\tilde{\al}^{j},\label{eq:complete-z-asympt-3.2}\\
\frac1y& =\tilde\al \sum^\infty_{i =0} \sum^i_{j=-i}\check{y}_{ij}t^{-i}\tilde{\al}^{j},\label{eq:complete-y-asympt-3.2}\\
\tilde\al &= \frac1{\de} t^{4\vp-\tin}e^{-t}=\tilde{\beta}/t.\label{eq:tilde-alpha-def}
\end{eqnarray}

\end{document}